\documentclass{article}	
\usepackage[utf8]{inputenc}				%(only for the pdftex engine)
\usepackage[big,online]{dgruyter}	%values: small,big | online,print,work
\usepackage{lmodern} 
\usepackage{microtype}
\usepackage[all]{xy}
\usepackage{graphicx}%
\usepackage{multirow}%
\usepackage[title]{appendix}%
\usepackage{manyfoot}%
\usepackage{hyperref}%Internal cross-referencing equation, diagrams
\usepackage{cleveref}
\usepackage{tikz-cd}
\usetikzlibrary{matrix, calc, arrows}
\usepackage{appendix}
\usepackage[numbers,square,sort&compress]{natbib}
\theoremstyle{dgthm}
\newtheorem{thm}{Theorem}
\newtheorem{cor}{Corollary}
\newtheorem{prop}{Proposition}
\newtheorem{lem}{Lemma}

\theoremstyle{dgdef}
\newtheorem{defn}{Definition}
\newtheorem{example}{Example}
\newtheorem{rem}{Remark}
\newcommand{\vast}{\bBigg@{3}}
\newcommand{\Vast}{\bBigg@{4}}
\newcommand {\mbbX}{[\mathbb{X}\colon X_1 \rightrightarrows X_0]}
\newcommand {\mbbY}{[\mathbb{Y}\colon Y_1 \rightrightarrows Y_0]}

\newcommand{\AtP}{{\mathcal A}{\mathfrak {t}}(P)}
\newcommand{\AdP}{{\mathcal A}{\mathfrak {d}}(P)}
\newcommand{\ActTM}{{\mathcal A\mathfrak {ct}}(\pi^* TM)}

\begin{document}

\bibliographystyle{plainnat}
\title{Connections  on a principal Lie groupoid bundle and representations upto homotopy}
\runningtitle{Connections on a principal Lie groupoid bundle}
%\subtitle{Insert subtitle if needed}

\author[1]{Saikat Chatterjee}
%\ use * to mark the author as the corresponding author
\author*[2]{Naga Arjun S J}
\runningauthor{Saikat Chatterjee, Naga Arjun S J}
\affil[1]{\protect\raggedright School of Mathematics,
Indian Institute of Science Education and Research, Thiruvananthapuram, Bonacaud Road, Kerala-695551, India, email:saikat.chat01@gmail.com}
\affil[2]{\protect\raggedright 
School of Mathematics,
Indian Institute of Science Education and Research, Thiruvananthapuram, Bonacaud Road, Kerala-695551, India, email:arjun21@iisertvm.ac.in}



\abstract{
A Lie groupoid principal  $\mbbX$ bundle is a surjective submersion $\pi\colon P\to M$ with an action of $\mathbb{X}$ on $P$ with certain additional conditions. This paper offers a suitable definition for the notion of a connection on such bundles. Although every Lie groupoid  $\mathbb{X}$  has its associated Lie algebroid $A:=1^*\ker ds\to X_0$, it does not admit a natural action on its Lie algebroid. There is no natural action of $\mathbb{X}$ on $TP$ either. Choosing a connection $\mathbb{H}\subset TX_1$ on the Lie groupoid $\mathbb{X},$ and considering its induced action up to homotopy of  $\mathbb{X}$ on graded vector bundle $TX_0\oplus A,$ we prove the existence of a short exact sequence of diffeological groupoids over the discrete category $M$ (with appropriate graded vector space structures on the fibres) for the $\mbbX$ bundle  $\pi\colon P\to M.$ We introduce a notion of connection on $\mbbX$ bundle  $\pi\colon P\to M,$ and show that such a connection $\omega$ splits the sequence. Finally, we show that a connection pair $(\omega, \mathbb{H})$ on $\mbbX$ bundle  $\pi\colon P\to M$ is isomorphic to any other connection pair.}


\maketitle

\section{Introduction}\label{Section:Introduction}
In the last few decades or so, one of the active fields of research has been the various structures on Lie groupoids and their associated geometry. One of the reasons behind this surge of interest ought to be how Lie groupoids naturally appear in multiple areas of mathematics and mathematical physics. For example, in foliation theory (as \textit{foliation groupoid} or \textit{monodromy groupoid} for a foliated manifold), theory of orbifolds (as \textit{proper, \'etale Lie groupoids}), differentiable stacks (as \textit{Mortita equivalence classes of Lie groupoids}) and theory of gerbes (as \textit{Morita equivalence classes of Lie groupoid extensions}), Higher gauge theories (as \textit{Lie $2$-group  bundles over Lie groupoids}), just to name a few. This makes the geometry of Lie groupoids and their related structures a vital direction of inquiry. 

A Lie groupoid can be viewed as a generalization of either a smooth manifold or a Lie group. Consequently, typical objects of  study are either a bundle over a Lie groupoid \cite{MR2270285, MR4721218, MR3150770, MR4403617, chatterjee2023parallel, MR4592876, 
MR4403617} or a principal Lie groupoid bundle over a manifold \cite{MR755163, MR2157566, MR1871544}, respectively.  The current paper deals with the latter.  Specifically, our paper proposes a connection structure on such a Lie groupoid principal bundle. 
Before discussing the content of this paper, we give a general overview of work already done in related areas.

Perhaps, historically, credit goes to Haefliger for introducing the notion of a Lie groupoid bundle over a smooth manifold \cite {MR755163}. Subsequently, through the works of Moerdijk, Behrend, Xu, and others  \cite{MR2817778, MR2778793, moerdijk2002introduction, MR2270285, MR2119241, MR2017529}, Lie groupoids and the Lie groupoid principal bundles become the central focus for studying differentiable stacks and gerbes. Let $\mbbX$ be a Lie groupoid. Then a Lie groupoid bundle (or a Lie groupoid torsor, as it has been called in some papers) is a smooth surjective submersion $\pi\colon P\to M$ with structure group replaced by a Lie groupoid, which means a ``Lie groupoid action''  on $P$ with some additional compatibility conditions. The category $\mathcal B{\mbbX}$ of such $\mathbb{X}$ bundles define a differentiable stack, and conversely given a differentiable stack one can find a Lie groupoid $\mbbX$ such that the stack $\mathcal{B}{\mathbb{X}}$ is isomorphic to the given stack \cite{MR2817778, MR2778793}. In other words, the Lie groupoid bundles classify differentiable stacks.
A slightly more general notion of a  $(\mathbb{X}-\mathbb{Y})$ bibundle $\pi\colon P\to M$, where $P\to M$ is a (right) $\mbbX$-bundle with   an equivariant (left) action of $\mbbY$ on $P,$ has also been studied(\cite{MR2778793,MR2817778}).
 Moreover   $(\mathbb{X}-\mathbb{Y})$ bibundle classify the maps of  stacks $\mathcal{B}(\mathbb{X})\to\mathcal{B} (\mathbb{Y})$ \cite{MR2778793}. 
 In \cite{MR2017529}, authors classify regular Lie groupoids using Girad's theories of gerbes and Non-abelian cohomology.
 Behrend and Xu \cite{MR2817778} classify differentiable stacks by Morita equivalent classes of Lie groupoids. Several papers have appeared investigating along this line.  Here, we mention a few. One may consult   \cite{MR2223406, moerdijk2002introduction} for a broad review of groupoids and stacks.  In \cite{MR4124773}, the authors have explored the correspondence between a Morita equivalence class of Lie groupoid extensions and a gerbe over a stack. In \cite{MR4170276}, equivariant cohomology for a differentiable  stack equipped with an action of a Lie group $G$ (namely, a $G$-stack)
has been worked out. In the book \cite{moerdijk2003introduction}, the ``unit bundle'' associated with a Lie groupoid provides the set-up for the proof of an extension result from a morphism between  Lie algebroids to their corresponding Lie groupoids. In \cite{MR1871544}, the author has constructed a Lie groupoid principal bundle out of appropriate cocycle conditions. 
For a Lie crossed module $(G, H),$ in \cite{ginot2013introduction}, a principal $(G, H)$-bundle over a differentiable stack has been studied. Though it may not be directly relevant to this paper, another recurring object of study involving Lie groupoids is the bundle over a Lie groupoid. For example, \cite{MR2270285, MR4721218, MR3150770} studied principal $G$-bundles over Lie groupoids and their connection structures, where $G$ is a Lie group. Whereas articles such as \cite{MR3480061, MR3566125, MR4403617, chatterjee2023parallel, MR3894086} discuss Lie $2$-group (or, equivalently, Lie crossed module) bundles over Lie groupoids. Some papers that deal with the bibundles of Lie crossed modules and their relation with non-abelian gerbes are \cite{MR2117631, MR3004105}.

While Lie groupoid principal bundles have recently received considerable attention, the connection structures on such bundles have remained virtually uncharted territory. Though, as one may perceive, a thorough study of such connections has many scopes; for example, it paves the way for introducing such a structure on differentiable stacks or allows the study of  Chern-Weil theory for such bundles. This apparent lack of interest may be attributed to the ambivalent relationship between a Lie groupoid $\mbbX$ and its Lie algebroid $1^*\ker ds:=A\to X_0$. Specifically, a Lie groupoid has no natural action on its Lie algebroid. For a $\mbbX$-bundle $\pi\colon P\to M,$ neither there is any induced action of $\mbbX$ on $TP;$ thus rendering it impossible to conceive a connection as an ``$\mbbX$ equivariant'' ``Lie algebroid valued'' differential form on $P.$ To circumvent this problem, we proceed as follows. 

We choose an (Ehresmann)  connection $\mathbb{H}$ on the Lie groupoid $\mbbX;$ that is a subbundle $\mathbb{H}\subset TX_1$ such that $\mathbb{H}\oplus \ker ds=TX_1$ and $\mathbb{H}_{1_x}=(d1)_x(T_xX_0)$, for every $x \in X_0$  \cite{cran}. It should be mentioned here that the definition of a connection on a Lie groupoid available in the literature is not unique, and various authors have considered various closely related but non-equivalent definitions depending on their requirements. For example, \cite{BEHREND2005583} considers a sub Lie groupoid of the tangent groupoid $TX_1\rightrightarrows TX_0$ to be a connection on the Lie groupoid (often called a \textit{Cartan connection on a Lie groupoid}), whereas  \cite{MR4592876} defines a connection with an additional condition of functoriality on an Ehresmann connection. More recently, the paper  \cite{doi:10.1142/S0219199721500929} introduced the multiplicative connection on $\mbbX$ as a vector bundle connection on the tangent bundle $TX_1\to X_1$ with certain additional compatibility conditions. Some other papers which deal with Lie groupoid connections are  \cite{MR2238946,  MR3886162}. A choice of connection $\mathbb{H}\subset TX_1$ on $\mbbX$ allows us to define a quasi action $\lambda$
of the Lie groupoid on the graded bundle $A\oplus TX_0.$ In fact, this quasi action is an action up to a homotopy, which roughly means  $\lambda_g\lambda_h-\lambda_{gh}$ does not vanish, instead measured by a homotopy relation (see \cref{defn:quasi action of X on both A and TX_0}).   Abad and Crainic have developed the necessary theory for representation up to a homotopy in \cite{cran}. We call this action (up to a homotopy) the ``adjoint action''. In this paper, we see that  the same choice of connection $\mathbb{H}$ defines a quasi action of $\mbbX$ on the graded vector bundle $TP \oplus a^*TX_0$ (see \cref{eqn:quasi action on TP}
in \cref{Subsection:Groupoid AtP}), and quite remarkably, respective homotopy terms involved in the quasi actions on $A\oplus a^*TX_0$ and $TP \oplus a^*TX_0$ are mutually well behaved to allow us to define a connection on the $\mathbb{X}$-bundle $P\to M$ (\cref{defn:connection as 1-form}). On the other hand, with the actions (up to homotopy) of $\mbbX$ on $A \oplus a^*TX_0$ and $TP \oplus a^*TX_0$, the natural step would be to construct the Adjoint bundle and the Atiyah bundle over $M$. However, due to the obstruction of the homotopy relation, we obtain a pair of groupoids $\AdP$  and $\AtP$ instead, corresponding to the actions upto homotopy on $A \oplus a^*TX_0$ and $TP \oplus a^*TX_0$ respectively, over the discrete category $M$ (\cref{Subsection:Groupoid AtP}, \cref{Subsection:Groupoid AdP}). 
The construction of these groupoids explicitly depends on the homotopy relations for the quasi-actions on $A$ and $TP.$ 
In general, $\AdP, \AtP$  are not Lie groupoids but are diffeological groupoids. We refer to \cite{iglesias2013diffeology} for a general discussion on diffeology. Among other works on diffeology, which we have relied on, particularly for diffeological vector bundles, are \cite{MR3467758, PERVOVA2016269}. Intriguingly, the fundamental vector field map turns out to be well-behaved under the ``categorization'' of the adjoint bundle and the Atiyah bundle, and 
induce a well defined diffeologically smooth functor $\bar d\delta\colon \AdP\to \AtP$
(\cref{prop:functor bardelta}) such that everything fits into a sequence of diffeological groupoids. This sequence of functors is short and exact in an ``appropriate sense'' (\cref{Thm:Atiyah sequence for Lie groupoid bundle existence}). Borrowing the classical terminology, we call this sequence the \textit{Atiyah sequence} of $\mathbb{X}$-bundle $\pi\colon P\to M.$  Moreover, we show that a connection on the $\mathbb{X}$ bundle defines a splitting of the Atiyah sequence. The converse holds for a Cartan connection $\mathbb{H} \subset TX_1$ (\cref{subsection:connections and quasi connections}). One should note that the entire framework we develop here depends on the choice of the connection $\mathbb{H}\subset TX_1$ on the Lie groupoid $\mbbX.$ Consider a connection pair $(\omega, \mathbb{H}),$ where $\omega$ is a connection on the $\mathbb{X}$-bundle  $P\to M$ for the choice of a connection $\mathbb{H}\subset TX_1$ on the Lie groupoid $\mbbX.$
To understand how different choices of connections on the Lie groupoid $\mbbX$ affect our framework, we define a category $\mathbb{CONNECTION}\bigl(\mathbb{X}, P\to M\bigr)$ whose objects are the connection pairs $(\omega, \mathbb{H})$
(\cref {eqn:Category of connections}). We show that two such connection pairs are isomorphic in the category $\mathbb{CONNECTION}\bigl(\mathbb{X}, P\to M\bigr)$ (\cref{prop:inducing a connection form from a theta}).

Among only very few works which address connection structures on a Lie groupoid  $\mbbX$ principal bundle $\pi\colon P\to M$ is the book by Moerdijk and Mr\u cun \cite{moerdijk2003introduction}. There, the authors have introduced ``$\mathcal{F}$-partial connection'' for a foliation $\mathcal{F}$ of $M$ (see \cref{rem:Casual Moerdijk connection} for an observation relating our framework with the partial connection in \cite{moerdijk2003introduction}). Very recently, with a physical application in mind,  specifically to study ``curved Yang-Mills-Higgs theories'',   a notion of a connection as a dg-Lie groupoid functor constructed from a \' Cech cocyle description of a principal Lie groupoid bundle has been introduced in \cite{fischer2024adjustedconnectionsidifferential} (one may also find an infinitesimal version of this approach in \cite{MR4202191}) \footnote{We thank Simon-Raphael-Fischer for referring these papers to us.}. Our framework of connections in this paper differs, both in motivation and in approach, from those of the others.  
The significance of our theory of connections on the Lie groupoid principal bundle lies in its very natural adaptability with the representations up to the homotopy of a Lie groupoid, leading to a sequence of diffeological groupoids analogous to the Atiyah sequence of a classical principal bundle. Quite interestingly, the connection defined here can also be characterised as a morphism of fibred categories.  We believe that the framework of connections on a Lie groupoid principal bundle developed here is new and will open up many future directions, particularly along the lines of differentiable stacks and foliation theory. In this paper,  we provide several examples in each section to illustrate the naturality of our framework. The framework of connection on the $\mbbX$-bundle $P\to M$ in this paper seems to have interesting properties concerning foliations on the base manifold $M$, and it also admits an extension along a Lie groupoid extension. 
Considering this paper's length, we thought it would be prudent to cover those topics in detail in a separate paper.  A \textit{stacky groupoid} is a groupoid whose morphism space is a differentiable stack rather than a smooth manifold \cite{MR2439561}. Bursztyn,  Noseda, and Zhu have recently characterized the principal actions of stacky groupoids \cite{MR4139032}. It would be interesting to see whether the framework developed in our paper can be extended to the more general setup of stacky groupoids.  In fact, in \cref{rem: relation between action of stacky groupoid and bundle upto homotopy}, we see that our description of quotient is closely related to the stacky quotient introduced in \cite{MR4139032}.

\subsection*{A brief outline of the paper}

The first three sections cover the background material for the rest of the paper. The results and construction in these sections are already available in the literature.  

In \cref{section:Principal Lie Groupoid Bundle over manifolds}, we collect the definitions and basic examples of Lie groupoids and Lie groupoid principal bundles.

The \cref{section:connection of a Lie groupoid} is crucial for the rest of the paper. We review the definition and examples of Lie algebroids, the notion of a connection on a Lie groupoid, and the theory of representations up to the homotopy of a Lie groupoid. In particular, we focused on the representation (up to homotopy) of a Lie groupoid on its Lie algebroid. The entire section is based on \cite{cran}. Although the results are known, given the importance of this section for this paper and the ease of reading, we adopt an elaborate treatment, occasionally providing proofs or results omitted in the original paper. 

In \cref{section:Diffeological spaces}, we recall the basics of diffeology and the notion of diffeological vector spaces.

In \cref{section:Atiyah sequence}, we show that by fixing an Ehresmann connection on the Lie groupoid $\mbbX$, one can induce actions upto homotopy of $\mbbX$ on $P \times_{X_0} A \oplus a^*TX_0$ and $TP \oplus a^*TX_0$ for a principal $\mathbb{X}$-bundle $P\to M$. We construct diffeological groupoids $\AdP$ and $\AtP$ associated with these two actions. The natural action of $\mathbb{X}$ on $TM$ gives the Lie groupoid $\ActTM.$ All three groupoids can be viewed as categories over the discrete category of $M$ with connected components of each fibre forming a graded vector space  (\cref{subsubsection:adjonit bundle AtP over M}, \cref{subsubsection:Atiyah bundle over M}, \cref{subsubsection:The bundle ActTM over M}). One of the central results in this section is that the fundamental vector field map $d\delta:P \times_{X_0} A \rightarrow TP$ respects the homotopy relation (\cref{prop:fundamental vector field map preserve error of quasi action}). Moreover $d\delta$
induces a smooth functor $\bar d\delta\colon \AdP\to \AtP.$  Whereas the projection $\pi\colon P\to M$ induces a smooth functor $\bar d\pi\colon \AtP\to \ActTM,$ 
resulting in the final theorem of this section,  that the functors 
$\bar d\delta\colon \AdP\to \AtP$ and $\bar d\pi\colon \AtP\to \ActTM$ form a sequence of diffeological groupoids with appropriate linearity properties (\cref{Thm:Atiyah sequence for Lie groupoid bundle existence}). We call it the Atiyah sequence of $\mathbb{X}$-bundle $P\to M$ for the choice of connection $\mathbb{H}\subset TX_1$ on $\mbbX.$

In \cref{section:connection as 1-form},  for a choice of an 
Ehresmann connection $\mathbb{H}$, we introduce a notion of connection form $\omega$ (Associated Lie algebroid valued $1$-form) on a principal $\mathbb{X}$ bundle (\cref{defn:connection as 1-form}), and provide an equivalent description of it (\cref{thm:chartacterization of connection form in terms of subbundle}) in terms of a subbundle $\mathcal{H}\subset TP$.  We also characterize this connection as a morphism of fibred categories (\cref{prop: natural transformation between F and G coming from ddelta and omega}). We show that connection forms  on $\mbbX$-bundle $P\to M$ obtained for two different choices of the Lie groupoid connections on $\mbbX$ are isomorphic to each other in the category $\mathbb{CONNECTION}\bigl(\mathbb{X}, P\to M\bigr)$ whose objects are the connection pairs $(\omega, \mathbb{H})$
(\cref {eqn:Category of connections}). Finally, we relate the construction of the Atiyah sequence in \cref{section:Atiyah sequence} by showing that a connection $1$-form splits the Atiyah sequence. If the Lie groupoid connection is Cartan, then the converse holds (\cref{prop:Connection and quasi connection in case of Cartan connection}).

We include an appendix for some lengthy calculations required for certain results. 

\section{Principal Lie groupoid bundle over a manifold}\label{section:Principal Lie Groupoid Bundle over manifolds}
We start by recalling the basic definitions of Lie groupoid and principal Lie groupoid bundles. 
Let $\mathfrak Man$ be the category of smooth, Hausdorff, second countable manifolds. 
\begin{defn}[Lie groupoid]\label{defn:Lie Groupoid}
  A \textit{Lie groupoid} is a groupoid internal to the category of smooth manifolds $\mathfrak Man$, along with the source and target maps being submersions.
  \end{defn} 
  Explicitly,  a Lie groupoid is a groupoid with both the object and morphism sets to be smooth manifolds along with smooth structure maps. Moreover, the source (hence, the target) map is required to be a submersion. We denote a Lie groupoid by $\mbbX$ or simply by $\mathbb{X}$. The source and target maps will typically be denoted by $s$ and $t.$ The composition map will be denoted as $m\colon X_1\times_{s, t}X_1\to X_1$ when we need to specify it explicitly; otherwise we write $m(g, h)$ as $gh.$ Associated to a Lie groupoid $\mbbX$ we have its \textit{tangent grouopid} $[T\mathbb{X}:TX_1\rightrightarrows TX_0]$ with composition maps are differentials of those of $\mbbX.$ In particular $dm\bigl((g, \beta), (h, \alpha)\bigr)$ will be denoted as $(gh, \beta\alpha):$
  \begin{equation}\label{notation:differential of composition map}
dm\bigl((g, \beta), (h, \alpha)\bigr):=(gh, \beta\alpha).
    \end{equation}
  
\begin{rem}\label{Rem:manifold type}
  For our purpose, we consider our  smooth manifolds in 
  ${\mathfrak Man}$ to be Hausdorff and second countable. However, the choice is not completely unambiguous. For example, in foliation theory, as it turns out, the holonomy groupoids are not Hausdorff; thus, Hausdorff condition is dropped.  Even the object manifold is occasionally assumed to be Hausdorff, whereas the morphism manifold is not \cite{moerdijk2003introduction}.
\end{rem}
\begin{rem}\label{Rem:Necessity of submersion}
    Observe that the category of smooth manifolds is not closed under fibre product or pull-back. However, the submersion condition on the source and target maps of a Lie groupoid $[\mathbb{X}\colon X_1 \rightrightarrows X_0]$ ensures that the set $X_1 \times_{X_0} X_1$ is indeed a smooth manifold.
\end{rem}

  It is often interesting to investigate the Lie groupoids enjoying some additional properties. For instance, if the source, target maps of a Lie groupoid $\mbbX$ are \'etale, then the  Lie Groupoid $[X_1 \rightrightarrows X_0]$ is called \textit{\'etale}.  Whereas, if the  map $(s,t)\colon X_1 \rightarrow X_0 \times X_0$
is proper, the Lie groupoid is called \textit{proper}. The proper, \'etale Lie groupoids are quite well studied because of their correspondence with orbifolds and Mumford stacks(\cite{MR2778793}).

In the following passage, we provide some common examples of Lie groupoids.
\begin{example}[A manifold as Lie groupoid]\label{example:manifold as groupoid}
    A smooth manifold $M$ defines a Lie groupoid $[M \rightrightarrows M]$ with both source and target maps are identity maps. This Lie groupoid is called the \textit{discrete groupoid} associated with the manifold $M$.
\end{example} 
\begin{example}[Pair groupoid]\label{example:pair groupoid}
    For every manifold $M$, the product manifold $M \times M$ is a Lie Groupoid over $M$. So, source and target maps are the second and first projections, respectively, and the composition map is given as: 
    $(m'',m') \circ (m',m)=(m'',m).$
\end{example}

 \begin{example}[Cover or \u Cech groupoid]\label{example:cover groupoid}
    Let $\mathcal{U}=\bigl\{U_i\bigr\}_{i \in I} $ be an open cover of a manifold $M$. Let $\bigsqcup \limits_{i \in I} U_i$ be the disjoint union of the cover, and  $\bigsqcup\limits_{i,j \in I} U_i \cap U_j$  the disjoint union of the double intersections of the same. Evidently we have a Lie groupoid $\mathcal{U}^{\rm Grpd}:=\Bigg[\bigsqcup  U_i\rightrightarrows \bigsqcup U_i \cap U_j\Bigg]$ with  the respective source-target maps $s(i,j,x)=(i,x)$, $t(i,j,x)=(j,x)$ with the composition $(j, k, x)\circ (i,j,x)=(i, k, x).$  This Lie Groupoid $\mathcal{U}^{\rm Grpd}$ is called the \textit{cover groupoid} or the \textit{Cech groupoid} of the cover $\mathcal{U}.$
\end{example}
 \begin{example}[Lie group as a Lie groupoid]\label{example:Lie group as lie groupoid}
    A Lie group $G$ itself can be viewed as a Lie Groupoid $[G\rightrightarrows *]$ with the object manifold being the singleton $0-$dimensional manifold with composition given by the group multiplication.
  \end{example}
 
\begin{example}[Action groupoid of a Lie group]\label{example:action groupoid}
      Given a smooth right action of a Lie group $G$  on a manifold $P$, we define a groupoid, namely the  \textit{action groupoid}, $[P\times G\rightrightarrows P].$ The structure maps are respectively given by $s(p, g)=p$, $t(p, g)=pg$ and $(pg, h)\circ (p, g))=(p,gh).$
 \end{example}
\begin{example}[Vector bundle as a Lie groupoid]\label{example:vector bundle as lie groupoid}
Let $p\colon E \rightarrow M$ be a smooth vector bundle. Consider the Lie groupoid $[E\rightrightarrows M]$ with source and target maps $p$ and the composition as the fibrewise addition of vectors. 
\end{example} 
\begin{example}[Atiyah Lie groupoid]\label{example:Atiyah groupoid}
Let $\pi\colon P \rightarrow M$ be a smooth principal $G$ bundle. Consider the pair groupoid $[P\times P\rightrightarrows P].$ The diagonal action $(p, p')g=(pg, p'g)$ of $G$ on $P\times P$ and the action of $G$ on $P$ reduce the pair Lie groupoid to a Lie groupoid $[\frac{P\times P}{G}\rightrightarrows M].$ The Lie groupoid is called \textit{Atiyah groupoid} or \textit{Gauge groupoid} of the principal $G$-bundle $\pi\colon P\to M$.
\end{example} 

As the introduction mentions, a Lie groupoid is a generalization of both a manifold and a Lie group. 
 In its later avatar, natural directions of inquiry are towards its action on a smooth manifold, its principal bundle structure \cite{MR2778793}, and the representation theory  \cite{cran}. This paper will primarily be interested in this aspect of the Lie group. The Lie groupoid principal bundles are particularly interesting for their association with differentiable stacks \cite{MR2778793}.

 \begin{defn}[Right action of a Groupoid]\label{defn:right action of lie groupoid}
 A \textit{Right action} of a Lie groupoid $\mbbX$ on a manifold  $P$ consisted of a pair of smooth maps, the \textit{anchor map} $a\colon P \rightarrow X_0$ , and the \textit{action map} $\mu\colon P \times_{a,X_0,t} X_1 \rightarrow P$ satisfying the following conditions:
 \begin{enumerate}
     \item $a\left(\mu(p,g)\right)=s(g)$ for all $(p,g) \in P \times_{X_0} X_1$,
     \item $\mu\left(\mu(p,g),h\right)=\mu(p,gh)$ for every $(p,g,h) \in \left((P \times_{X_0} X_1) \times_{X_0} X_1 \right) $,
     \item $\mu\left(p,1_{a(p)}\right)=p$ for every $(p,1_{a(p)}) \in P \times_{X_0} X_1$.
 \end{enumerate}
 An action map $\mu\colon P \times_{a, X_0, t} X_1 \rightarrow P$ satisfying all the conditions except (2) and (3) is called a \textit{quasi action}. When there is no scope of confusion, we write $\mu(p, g)$ as $pg,$ and similarly for the differential $d\mu_{(p,g)}\bigl((p,u),(g,\alpha)\bigr)$ will be denoted in short as $(pg,u.\alpha).$
 \end{defn}
 It should be kept in mind that a given  $g \in X_1$ can only act on the elements of the fibre $a^{-1}(t(g))$.  

 \begin{rem}\label{rem:equivalance of right and left action}
Likewise, one defines the left action of the Lie groupoid on a smooth manifold as a smooth map  $X_1 \times_{s, X_0, a} P\rightarrow P$.  
 It is obvious that a right action can be turned into a left action, and vice versa, by inverting the elements of $X_1.$
 \end{rem}

\begin{rem}\label{rem:Action groupoid of a lie groupoid action}
    Imitating the construction in \cref{example:action groupoid}, one can produce a Lie groupoid $[P \times_{X_0} X_1 \rightrightarrows P]$ with obvious structure maps     
    for a Lie groupoid $\mbbX$ acting on a manifold $P$.
\end{rem}

\begin{example}\label{example:group acting on a manifold}
Suppose a Lie group $G$ acts on a smooth manifold $P.$ Then this action can be naturally interpreted as the action of the Lie groupoid $[G\rightrightarrows *]$ (see \cref{example:Lie group as lie groupoid}) on $P$  with anchor map the constant map $P\to *$ and the action map $(p, g)\mapsto pg.$

 \end{example}

\begin{example}\label{example:action groupoid acting on a manifold}
Suppose a Lie group $G$ acts on a smooth manifold $P.$ Then we have an action of the action groupoid $[P\times G\rightrightarrows P]$ (see \cref{example:action groupoid}) on $P$ with anchor map being the identity $p\mapsto p$ and the action map $p' (p, g)=p' g^{-1}.$

 \end{example}

 \begin{example}\label{example:groupoid acting on itself}
     A Lie groupoid  $[X_1 \rightrightarrows X_0]$ acts on $X_1$ with source map $s$ as the anchor map and the action map $(g', g)\mapsto g'\circ g$ given by the composition in Lie groupoid.
 \end{example}

\begin{example}\label{example:Atiyah groupoid action}
The Atiyah Lie groupoid $[\frac{P\times P}{G}\rightrightarrows M]$ associated with the $G$-bundle $\pi\colon P\to M$ in \cref{example:Atiyah groupoid} acts on $P$ with respect to the anchor map $\pi\colon P\to M$ (see Example (2.12) in \cite{MR4592876}).

\end{example} 

\begin{example}\label{example:Action of groupoid on its object space}
    Suppose $\mbbX$ is a Lie groupoid; then there is a natural action of $\mbbX$ on $X_0$ with the identity map taken as the anchor map.  Action map $\mu:X_0 \times_{{\rm Id},X_0,t} X_1 \rightarrow X_1$  is given by $\mu(x,g)=s(g).$
\end{example}
\begin{defn}[Principal Lie groupoid bundle]\label{defn:torsor}
    Let $[\mathbb{X}:X_1 \rightrightarrows X_0]$ be a Lie groupoid.  A \textit{principal $\mathbb{X}$-bundle} over the manifold $M$ is a smooth manifold $P$ with a surjective submersion $\pi\colon P \rightarrow M$ and  an action of the Lie groupoid $\bigl(a\colon P \rightarrow X_0, \, \mu\colon P \times_{a,X_0,t} X_1 \rightarrow P\bigr)$ on the manifold, satisfying the following conditions
    \begin{enumerate}
        \item $\pi(\mu(p,g))=\pi(p)$ for every $p \in P, \, g \in X_1$ such that $a(p)=t(g),$\\
        \item the map $(\text{pr}_1,\mu):P \times_{a,X_0,t} X_1 \longrightarrow P \times_{\pi,M,\pi} P, (p,g)\mapsto (p,pg)$ is a diffeomorphism.\\
    \end{enumerate}
    As a Diagrammatic illustration, we represent a principal Lie groupoid bundle by,
    
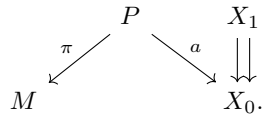
\[
    \begin{tikzcd}
        & P \arrow[dr, "a"]  \arrow[dl, "\pi"']&  X_1 \arrow[d, shift left] \arrow[d, shift right]\\
        M & & X_0.\\
    \end{tikzcd}
    \]
\end{defn}
\begin{defn}[Map between two principal $\mathbb{X}$-bundles]
    Let $\pi:P \rightarrow M$ and $\pi^{'}:P' \rightarrow M'$ be a pair of principal $\mbbX$-bundles with respective anchors $a_P:P \rightarrow X_0$ and $a_{P'}:P' \rightarrow X_0$.  Then a \textit{morphism between principal Lie groupoid bundles} consists of smooth maps $f:P \rightarrow P'$ and $\Bar{f}:M \rightarrow M'$ such that ${a_{P'}}\circ f=a_{P}$ and  $f(pg)=f(p).g$, for every $(p,g)$ with $a_P(p)=t(g)$.  Moreover, the following diagram commutes:
    
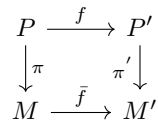
\[
    \begin{tikzcd}
        P \arrow[r,"f"] \arrow[d, "\pi"] & P' \arrow[d,"\pi^{'}"'] \\
        M \arrow[r, "\Bar{f}"] & M' \\
    \end{tikzcd}
    \]
   \end{defn}
   Conventionally, $\Bar{f}=\text{Id}_M$, when the base space of two principal bundles $M$ and $M'$ coincide.
The category of principal $\mbbX$-bundles will be denoted as ${\mathcal B}\mbbX.$

\begin{example}\label{example:ordinary bundle}
Let $P\to M$ be a principal $G$-bundle. Then 
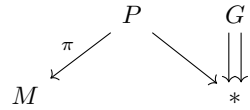
\[
    \begin{tikzcd}
        & P \arrow[dr]  \arrow[dl, "\pi"']&  G \arrow[d, shift left] \arrow[d, shift right]\\
        M & & *\\
    \end{tikzcd}
    \]

defines a $[G\rightrightarrows *]$-bundle $P\to M$ over $M$ for the action in \cref{example:group acting on a manifold}.
\end{example}

\begin{example}\label{example:ordinary bundle with action groupoid}
Let $P\to M$ be a principal $G$-bundle. Then 
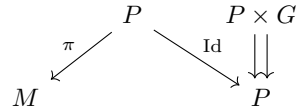
\[
    \begin{tikzcd}
        & P \arrow[dr, "{\rm Id}"]  \arrow[dl, "\pi"']&  P\times G \arrow[d, shift left] \arrow[d, shift right]\\
        M & & P\\
    \end{tikzcd}
    \]

defines a principal $[P\times G\rightrightarrows P]$-bundle $P\to M$ over $M$ with the action map 
defined as $\mu:P \times_{P} (P \times G) \rightarrow P,\mu\bigl(pg, (p, g) \bigr)=pg.g^{-1}=p.$ 
\end{example}

\begin{example}[Unit bundle]\label{example:unit bundle}
    The Lie groupoid itself provides an important class of examples for Lie groupoid bundles. For a Lie Groupoid $\mbbX$, the target map $t\colon X_1\to X_0$ is a principal $\mathbb{X}$-bundle over the space of objects $X_0$.  The source map $s: X_1 \rightarrow X_0$ is the anchor map while the groupoid composition gives the action, $(g, g')\mapsto g\circ g'$. These bundles are called \textit{Unit bundles}. The following diagram represents a unit bundle.
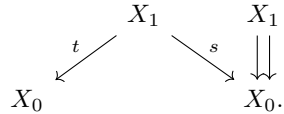
\[
\begin{tikzcd}
        & X_1 \arrow[dr, "s"]  \arrow[dl, "t"']&  X_1 \arrow[d, shift left] \arrow[d, shift right]\\
        X_0 & & X_0.\\
    \end{tikzcd}
\]
\end{example}

\begin{example}[Pullback bundle]\label{example:pullback principal bundle}
    Suppose $\pi\colon P \rightarrow M$ is a principal $\mathbb{X}$-bundle over the manifold $M$ and $f\colon N \rightarrow M$  a smooth map. Then the map $\text{pr}_1\colon N \times_{M} P \rightarrow N$ produces a principal $\mathbb{X}$-bundle over $N$, namely the \textit{pull-back bundle} of $\pi\colon P \rightarrow M$ by $f.$ We denote this pull-back bundle by $f^*P\to N.$ To elaborate,  the action of $\mathbb{X}$ on $N\times_{M} P$ is given by the anchor map $a \circ {\rm pr}_2\colon N \times_{M} P \rightarrow X_0,  (n, p)\mapsto a(p)$ and the action map $\tilde{\mu}\colon \left(N \times_{M} P\right)_{X_0} X_1 \rightarrow \left(N \times_{M} P\right), \, \bigl((n, p), g\bigr)\mapsto (n, pg).$
\end{example}
In particular, a pullback of the unit bundle $t\colon X_1 \rightarrow X_0$
along a smooth map $f\colon M \rightarrow X_0$ produces the pullback bundle $f^* X_1:=M \times_{X_0} X_1$ over $M$.  A bundle obtained via pullback of a unit bundle is often called a \textit{trivial bundle}. 
\begin{example}[Restriction bundle]\label{example:restriction bundle}
Let $\pi\colon P\to M$ be an $\mathbb{X}$-bundle and $N\subset M$ be an embedded sub-manifold. Then the pullback along the inclusion $\mathfrak{i}\colon N\hookrightarrow M$ produces the \textit{restriction $\mathbb{X}$-bundle}  $P_N:=\mathfrak{i}^*P\to N$ over $N.$

\end{example}

\begin{rem} \label{rem:local triviality of X-torsor}
Any principal $\mbbX$-bundle $\pi\colon P\to M$ is locally trivial \cite{moerdijk2003introduction}.    \end{rem}

\begin{rem}\label{remark:unit bundle pull back stack}
It is well known that the map 
${\mathcal B}\mbbX\to {\mathfrak Man},  
\hskip 0.2 cm  (P\to M) \mapsto M$
    defines a stack over the site (with respect to \'etale topology) of the category of smooth manifolds ${\mathfrak Man}$. Let $\underline{X_0}$ be the comma category over the manifold $X_0$ in $\rm Man;$ that is, its objects are smooth maps $f\colon N\to X_0,$ and an arrow from $(f_1\colon N\to X_0)$ to $(f_2 \colon M\to X_0)$ is a smooth map $h\colon N\to M$ such that $f_2\circ h=f_1.$ Then we have a functor $\underline{X_0}\to {\mathcal B}\mbbX$ which sends $(f\colon N\to X_0)$ to the pull-back $\mathbb{X}$-bundle $f^*X_1\to N$ of the unit principal $\mathbb{X}$-bundle $t\colon X_1\to X_0$ along $f$, and in turn $X_0$ provides an atlas or representative of the stack ${\mathcal B}\mbbX.$
\end{rem}

%%%%%%%%%%%%%%%%%%SECTION%%%%%%%%%%%%%%%%%%%%%%%%%%%%%%%%%%%%%%%%%%%%%

\section{Representations of  a Lie groupoid}\label{section:connection of a Lie groupoid}
In this section, we focus on the representation theory of Lie groupoids, particularly representations of a Lie groupoid up to a homotopy, which plays a central role in this paper. The content is more or less standard and well-studied in literature \cite{cran}, though we occasionally provide our insights and proofs.

\begin{defn}[Lie algebroid]\label{defn:lie algebroid}[\cite{MR1973056}]     A \textit{Lie algebroid} over the manifold $M$ is a vector bundle $p:A \rightarrow M$ along with a vector bundle map (anchor) $A \xrightarrow{\rho} TM$ with a Lie bracket on the space $\Gamma(A)$ of the sections of $A$
    satisfying the Leibniz identity:
    \[
    [\alpha, f\beta]=f[\alpha, \beta]+\rho(\alpha)(f)\beta,
    \]
for every $\alpha, \, \beta \in \Gamma(A)$ and $f \in C^{\infty}(M)$.
Diagrammatically, one can express this as follows:
\[
\begin{tikzcd}
    A \arrow[r, "\rho"] \arrow[d, "p"] & TM\\
    M
\end{tikzcd}
\]
\end{defn}

\begin{example}\label{examp:Ordinary Lie algebra}
  Any Lie algebra $\mathcal L$ trivially defines a Lie algebroid, given by the vector bundle $\mathcal L\to *$.  
\end{example}

\begin{example}[Tangent Bundle]\label{example:tangent bundle as lie algebroid}
    The tangent bundle $TM$ of a manifold $M$ is a Lie algebroid over the manifold $M$. Here, the anchor map is the identity map, while the usual Lie Bracket of the vector fields gives the Lie bracket. 
\end{example}
\begin{example}[Atiyah algebroid]\label{example:Atiyah  algebroid}
Let $\pi\colon P\to M$ be a principal $G$-bundle. Then the vector bundle ${\rm At}(P):=TP/G\to M$ along with the (anchor) map ${\rm At}(P)\to TM,\hskip 0.2 cm  
[(p, v)]\mapsto (\pi(p), d\pi_{p}(v))$ defines an algebroid  known as \textit{Atiyah algebroid}.

\end{example}

Let $\mbbX$ be a Lie groupoid. For a fixed $h\in X_1$, we have a pair of smooth maps, respectively, for the right and left compositions 
\begin{equation} \label{Equation:Action of fixed arrow}    \begin{split}
 R_h\colon & s^{-1}({t(h)})\to X_1, \, g\mapsto gh\\
L_h\colon & t^{-1}({s(h)})\to X_1, \, g\mapsto hg.
\end{split}
\end{equation}

A smooth vector field  $\chi$ on $X_1$ is called  right  invariant,  if 
 \begin{itemize}
\item  {$\chi$ is tangent to the fibres of the source map:
    $\chi_g \in  \ker(ds)_g,   \forall g \in X_1.$}
\item {If $g,h$ are two arrows such that $s(g)=t(h)$, then $\chi_{gh}=(dR_h)_{g}(\chi_g),$ where $(dR_h)_g\colon T_g s^{-1}({t(h)})\to T_{gh}X_1$ is the differential of the map in \cref{Equation:Action of fixed arrow}.}    
\end{itemize}
Note that the first condition implies right invariant vector field $\chi\colon X_1\to TX_1$ is a section $\chi\colon X_1\to \ker(ds) $ of the subbundle $\ker(ds)\to X_1.$
It is not difficult to see that the right invariant vector fields are closed under the usual Lie bracket of vector fields.

\begin{example}[Associated Lie Algebroid]\label{example:associated Lie algebroid}
    For a Lie groupoid  $[X_1 \rightrightarrows X_0]$, one can associate a Lie Algebroid $A$ over $X_0$ as follows. Consider 
    the differential of the source map $ds\colon TX_1\to TX_0.$ Then $\ker ds\subset TX_1$ is a subbundle over $X_1,$ and pulling back $\ker ds\to X_1$ along the unit map $1\colon X_0\to X_1, x\mapsto 1_x$ we obtain a vector bundle $A:=1^*\ker ds=\bigl\{(x, 1_x, \alpha) \mid \alpha \in T_{1_x} s^{-1}(x)\bigr\}$ over $X_0.$ The anchor map is given by the restriction of $dt$ along a fibre  $T_{1_{x}}s^{-1}(x);$
    \begin{equation}\label{equation:Anchor map associated}
        \begin{split}
&\rho\colon A \rightarrow TX_0,\\
&(x, 1_x, \alpha) \mapsto \bigl(x, (dt|_{s^{-1}(x)})_{1_x}(\alpha)\bigr).            
        \end{split}
    \end{equation}
     The Lie algebra structure on $\Gamma(A)$ is given by the well-known one-to-one correspondence between $\Gamma(A)$ and the space of right invariant vector fields \cite{cran}. Thus, we have the associated  Lie algebroid:\[
\begin{tikzcd}
    1^* \ker(ds) \arrow[r, "dt"] \arrow[d, "\text{pr}"] & TX_0\\
    X_0.\\
\end{tikzcd}
\]

\end{example}

\begin{example} \label{example:ordinary associated Lie algebra}
In particular, as a special case of the last example, viewing a Lie group $G$ as a Lie groupoid  $G \rightrightarrows *$,  one sees that the associated Lie algebroid  $L(G)\rightarrow *$ is nothing but the Lie algebra $L(G)$ of the Lie group $G.$
\end{example}

As we saw in \cref{example:associated Lie algebroid}, one can associate a Lie algebroid with a Lie groupoid, as for a Lie algebra with a Lie group. However, the relation between a Lie groupoid and its algebroid  distinctly departs from that of a Lie group and Lie algebra in two crucial aspects:
\begin{enumerate}
    
\item Failure of ``\textit{Lie's third Theorem}"; That is, given a Lie algebroid, it is not always possible to find a Lie groupoid such that the corresponding Lie algebroid of it coincides with the first one. 

\item There is no natural (adjoint) action or representation of a Lie groupoid on its Lie algebroids.  
\end{enumerate}
    For (1), a Lie algebroid which corresponds to a (unique) Lie groupoid is called an \textit{integrable Lie algebroid}. For example, the Atiyah Lie algebroid of \cref{example:Atiyah  algebroid} integrates to the Atiyah Lie groupoid of \cref{example:Atiyah groupoid}. One can find an example of a nonintegrable Lie algebroid in \cite{MR1973056} arising from a smooth differential closed $2$-form on a manifold. In the same paper, Crainic and Fernandes discovered the obstructions that prevent a Lie algebroid from being integrable. We refer to \cite{fernandes2006lecturesintegrabilityliebrackets} for a more detailed exposition and examples of nonintegrable Lie algebroids. For (2),  the situation can be partially salvaged by introducing the notion of representation of a Lie groupoid up to homotopy.

%%%%%%%%%%%%%%%%%%%%%%%%%%%%%%

\subsection{Representation up to homotopy of a Lie groupoid}
    It is apparent from \cref{defn:right action of lie groupoid} that the natural object for defining a representation of a Lie groupoid $\mbbX$ would be a vector bundle $E\to X_0$ over $X_0.$
\begin{defn}[Representation of a Lie groupoid, \cite{cran}]\label{defn:representation of a Lie groupoid}
    Let $\mbbX$ be a Lie groupoid and $\pi\colon E \rightarrow X_0$  a vector bundle. A \textit{representation} of $\mathbb{X}$ on $E$ is a  (left) action $\lambda\colon X_1\times_{\pi, s} E \to E$ of (see, \cref{rem:equivalance of right and left action})    of $\mbbX$ on $E$ such that  
     restriction of $\lambda$ along a fibre of $g$,    
    $$\lambda_g\colon E_{s(g)} \rightarrow E_{t(g)}$$ is a linear isomorphism.  
    \end{defn}

According to conditions (2), (3) in \cref{defn:right action of lie groupoid} 
\begin{equation} \label{eqn:functoriality of representation}
\begin{split}
 &\lambda_{gh}=\lambda_g  \lambda_h,\, \,  {\rm{for\,  every}}\,  (g,h) \in X_1\,\,  {\rm{satisfying}}\,\, s(g)=t(h), \\
        &\lambda_{1_x}={\rm Id}_{E_x}, {\rm{for\, every}}\, x \in X_0.
\end{split}
       \end{equation}
We often refer to the first and second conditions of the above equation as \textit{functoriality} and \textit{unitality of representation}, respectively.  A \textit{quasi representation} of $\mathbb{X}$ on $E$ is merely a (left) quasi action $\lambda\colon X_1 \times_{X_0} E \to E$ such that $\lambda_g \colon E_{s(g)} \rightarrow E_{t(g)}$ is a fibrewise linear isomorphism, for each $g \in X_1.$
\begin{example}
 It is obvious that a representation of a Lie group $G$  on a vector space $V$ is the same as the representation of the Lie groupoid  $G\to *$ on the trivial vector bundle $V\to *.$
\end{example}
\begin{example}\label{example:quasi representation on trivial bundle}
    Suppose $\mbbX$ is a Lie groupoid and $X_0 \times V \rightarrow X_0$  a trivial vector bundle $X_0.$ Then the map $\lambda:X_1 \times_{X_0} (X_0 \times V) \rightarrow X_0 \times V$, $\lambda(g,s(g),v)=(t(g),-v)$ defines a representation which is neither unital nor functorial.  More generally, if $\lambda$ is a representation of $\mbbX$ on a vector bundle $E \rightarrow M,$ then $\tilde{\lambda}(g,s(g),v)=(t(g),-\lambda_g(v))$ provides a quasi representation which is not a representation.
\end{example}
As it turns out, this definition is too restrictive and limited in scope. For example, as mentioned earlier, there is no natural representation of a Lie groupoid on its Lie algebroid. In \cite{cran}, authors have introduced the concept of representation of a Lie groupoid up to homotopy and extended the adjoint representations of Lie Groups to the Lie Groupoid setup. The main ingredient of the theory is to define the representation on a graded vector bundle rather than on a traditional vector bundle so that the functroiality condition of action can be replaced by a condition which holds up to a ``homotopy'' relation. In what follows, relying on \cite{cran}, we review the theory of representations up to the homotopy necessary for this article. 

Another perspective on viewing the notion of representation of a Lie groupoid is via the nerve of a Lie groupoid.  Let $\mbbX$ be a Lie groupoid. Since the source and target maps are both smooth submersions, for
	each $i\geq 0,
X_i\,:=\,\underbrace{X_1\times_{X_0}\cdots\times_{X_0}X_1}_{i-{\rm times}}$
 is a smooth manifold, and  
	associated to that there are  $i+1$ canonical projection maps ${\rm pr}_{i, \alpha}\colon X_i\rightarrow X_{i-1}.$ Thus we get a simplicial manifold $\mathbb{X}_{\bullet}$, known as the \textit{nerve of the Lie groupoid} $\mbbX.$
\begin{equation}\label{equation:Simplicial manifold associated to lie groupoid}
		\xymatrix{
		\mathbb{X}_\bullet\,=\,\biggl\{\ldots\ar@<-1.5ex>[r]\ar@<-.5ex>[r]\ar@<1.5ex>[r]\ar@<.5ex>[r] & X_2
			\ar[r]\ar@<1ex>[r]\ar@<-1ex>[r] & X_1\ar@<-.5ex>[r]\ar@<.5ex>[r]
			&X_0\,}\biggr\}.
	\end{equation}
The source-target maps extend to  $s\colon X_k \rightarrow X_0,s(g_1,\ldots,g_k)=s(g_k)$ and $t\colon X_k \rightarrow X_0,t(g_1,\ldots,g_k)=t(g_1)$.  Suppose $E \rightarrow X_0$ is a vector bundle over $X_0$ and let $\Gamma(X_k,t^*E)$ denote the space of smooth sections of the vector bundle $t^* E \to X_k.$  Now, form a graded vector space $\Gamma_{\bullet}(\mathbb{X},E)$ whose degree $k$ part is $\Gamma(X_k,t^*(E));$
\begin{equation}\label{equation:Graded vector space with non graded bundle}
\Gamma_{\bullet}(\mathbb{X},E)=\bigoplus \Gamma(X_k,t^{*}E).
\end{equation}
In particular for the trivial line bundle $X_0 \times \mathbb{R} \to X_0$ we get that 
$\Gamma(X_k,t^*(X_0 \times \mathbb{R}))=\Omega^{0}(X_k),$ the space of real-valued smooth functions  on $X_k$ and the corresponding  graded vector space has a cochain complex structure with the differential $\delta(f)=\sum\limits_{i=1}^{k}(-1)^{i}f \circ \text{pr}_{i, k}$:
\begin{equation}\label{equation:complex associated to the nerve}
\begin{split}
    &\xymatrix{
     \Omega^{0}(X_0) \xrightarrow{\delta} \Omega^{0}(X_1) \xrightarrow{\delta} \Omega^{0}(X_2) \ldots
    }\\
    &\Omega_{\bullet}(\mathbb{X}):=\bigoplus_k \Omega^{0}(X_k).
    \end{split}
\end{equation}
  Moreover, $\Omega_{\bullet}(\mathbb{X})$ has a natural algebra structure 
\begin{equation}\label{equation:algebra structure of complex}
    f  h(g_1,g_2,\ldots,g_{k+l})={} (-1)^{kl}f(g_1,\ldots,g_k)h(g_{k+1},\ldots,g_{k+l}),
\end{equation}
for every $f \in \Omega^{0}(X_k)$ and $h \in \Omega^{0}(X_l),$ and  the differential $\delta$ satisfies the Leibniz identity
\begin{equation}\label{equation:Leibniz identity}
\delta(f  h)=\delta(f)  h + (-1)^{\text{deg} \, f}f  \delta(h).
\end{equation}
Thus, the complex $\Omega^{\bullet}(\mathbb{X})$ in \cref{equation:complex associated to the nerve} becomes a  Differential graded algebra (DGA).

   For an arbitrary vector bundle $E \rightarrow X_0$ the space $\Gamma_{\bullet}(\mathbb{X}, E)$ is a  right graded module over the DGA $\Omega_{\bullet}(\mathbb{X})$ with a scalar product defined in the following way:  
   Given  $\eta \in \Gamma({X_k,t^*E})$ and $f \in \Omega^{0}(X_k)$, define the scalar product $\eta  f \in \Gamma(X_{k+k'},t^*E)$ as
\begin{equation}\label{eqn:module structure over the complex}
    (\eta  f)(g_1,\ldots,g_{k+k'})={} (-1)^{kk'})f(g_{k+1},\ldots,g_{k+k'})*\eta(g_1,\ldots,g_k).
\end{equation}
Additionally, one may also explore the normalized subspace $\hat{\Gamma}_{\bullet}(\mathbb{X},E)$ of $\Gamma_{\bullet}(\mathbb{X},E)$ comprising of those $\eta$ with the property that $s_i^*(\eta)=0$, for all $i,$ where $s_i^*(\eta)$ denotes the pullback of $\eta$ along the degeneracy map $s_i:X_k \rightarrow X_{k+1},s_i(g_1,\ldots,g_k)=(g_1,\ldots,g_i,1,g_{i+1},\ldots,g_k)$ of the simplicial manifold $\mathbb{X}_{\bullet}$.  

Given a (left) quasi representation  $\lambda$ of $\mbbX$ on $E \rightarrow X_0$ we have a linear isomorphism $\lambda_{g}\colon E_{s(g)} \rightarrow E_{t(g)}$, for each $g \in X_1.$ Observe that for  $\eta \in \Gamma(X_k,t^{*}E)$ and $(g_1,\ldots, g_k)\in X_k,$
$$\eta(g_1, \ldots, g_k)\in t^* E |_{g_1, \ldots, g_k}=E_{t(g_1)}=E_{s(g_2)}.$$
Then, one may associate a degree one operator 
$D_{\lambda}\colon \Gamma_{\bullet}(\mathbb{X},E) \rightarrow \Gamma_{\bullet}(\mathbb{X},E)$ which sends a degree $k$ vector $\eta \in \Gamma(X_k,t^{*}E)$ to a degree $k+1$ vector $D_{\lambda}(\eta) \in \Gamma(X_{k+1},t^*E^l)$:
\begin{align}\label{equation:degree one operator associated to a quasi action}
   \bigl(D_{\lambda}(\eta)\bigr)(g_1,\ldots, g_{k+1}) = {} & (-1)^k \biggl(\lambda_{g_1}\bigl(\eta(g_2,\ldots,g_{k+1})\bigr) \notag \\
     & {} + \sum\limits_{i=1}^{n}(-1)^{i}\eta(g_1,\ldots,g_i \circ g_{i+1},\ldots,g_{k+1}) + (-1)^{k+1}\eta(g_1,\ldots,g_k)\biggr).
\end{align}
The following characterisation of quasi-action is due to Abad and Crainic (Lemma 2.6 in \cite{cran}).   The association $\lambda \mapsto D_{\lambda}$ provides a one-to-one correspondence between quasi representations $\lambda$ of a Lie groupoid $\mbbX$ on a vector bundle $E \rightarrow X_0$ and degree $1$ operators $\mathcal D$ on $\Gamma_{\bullet}(\mathbb{X}, E)$ satisfying the Leibniz identity ${\mathcal D}(\eta  f)={\mathcal D}(\eta) f + (-1)^{k}\eta \delta(f),$  for every $\eta \in \Gamma(X_k,t^*E)$ and $f \in \Omega^{0}(X_k)$.  Moreover, $\lambda$ is unital if and only if the normalized subspace $\hat{\Gamma}_{\bullet}(\mathbb{X},E)$ is invariant under $D_{\lambda}$, and $\lambda$ is a representation if and only if it is unital and $D_{\lambda}^2=0$.

Now, suppose we have a graded vector bundle $E=\bigoplus\limits_{l \in Z} E^l$ over $X_0$ instead of a usual vector bundle, then for each $l$ and $k$ we have the space of sections $\Gamma(X_k,t^{*}E^l)$. In turn we arrive at  a bigraded vector space $C(\mathbb{X},E)^n$ with total grading given as
\begin{equation}\label{Equation: E Valued cochain}
    \begin{split}
&  C(\mathbb{X},E)^n=\bigoplus\limits_{k+l=n}
\Gamma(X_k,t^*E^l),\\
& C(\mathbb{X},E)^\bullet=\bigoplus_n C(\mathbb{X},E)^n.   \end{split}
\end{equation}
As in \cref{equation:algebra structure of complex}, $C(\mathbb{X},E)^{\bullet}$ becomes a bigraded right $\Omega^{\bullet}(\mathbb{X})$-module in an obvious manner. 
 The bijective correspondence $\lambda \rightarrow D_{\lambda}$  in Lemma 2.6 in \cite{cran} allows Abad and Crainic to define a representation up to homotopy as follows.

\begin{defn}[Representation upto homotopy,\cite{cran}]\label{defn:representation upto homotopy}
A representation up to homotopy of a Lie Groupoid $\mathbb{X}$ consists of a graded vector bundle $E=\bigoplus E^l$ over $X_0$ together with a linear degree one operator (structure operator)
\[
D\colon C\left(\mathbb{X},E\right)^{\bullet}
  \rightarrow  C\left(\mathbb{X},E\right)^{\bullet}
\]  
satisfying $D^2=0$ and the Leibniz identity.
\[
D(\eta \, \star \, f)=D(\eta) \, \star \, f + (-1)^k \eta \, \star \, \delta(f)\\
\]
for $\eta \in \Gamma(X_k,t^*E)^k$ and $f \in \Omega_{\bullet}(\mathbb{X}).$
\end{defn}
% In fact, there is another characterisation of a representation up to homotopy, which will be more useful for our paper. 

Fix a $k\in +{\mathbb Z}.$ For a graded vector bundle $E=\bigoplus\limits_{l \in Z} E^l$ over $X_0,$ consider the graded vector bundles $s^*(E^\bullet):=s^*(\bigoplus_l E^l)$ and $t^* (E^{\bullet+1-k}):= t^* (\bigoplus_l E^{l+1-k})$ over $X_k.$ Let ${\rm Hom}\bigl(s^*(E^\bullet), t^* (E^{\bullet+1-k}\bigr)$ be the Hom-bundle  over $X_k.$  Conventionally, for the case $k=0$, ${\rm Hom}\bigl(s^*(E^{\bullet}),t^*(E^{\bullet+1})\bigr)$ denotes the Hom-bundle ${\rm Hom}\bigl(E^{\bullet}, E^{\bullet+1}\bigr)$ over $X_0$.

\begin{prop}[Proposition 3.2 in \cite{cran}]\label{characterization of representation upto homotopy}
    There is a one-to-one correspondence between a representation up to homotopy of a Lie groupoid  $\mathbb{X}$ on a graded vector bundle $E=\bigoplus\limits_{l \in Z} E^l$ and a sequence of operators
    ${(R_k)}_{k \geq 0} \in \Gamma\bigl(X_k,{\rm Hom}(s^*(E^{\bullet}),t^*(E^{\bullet+1-k})\bigr)$  satisfying the condition
\begin{align}\label{equation:compatability relations between R_k's}
    \sum_{j=1}^{k-1}(-1)^jR_{k-1}(g_1,g_2,\ldots,g_j \circ g_{j+1},\ldots,g_k)= {} & \sum_{j=0}^{k} (-1)^jR_{j}\left(g_1,g_2,\ldots,g_j\right) \circ R_{k-j}\left(g_{j+1},\ldots,g_{k}\right).
    \end{align}
\end{prop}

\begin{rem}
  In a general setup, one has to keep track of higher $R_j, j>2.$ Our primary interest will be the representation (up to homotopy) of a Lie groupoid on its Lie algebroid, which will involve a bi-graded vector bundle; that is $l\in \{0, 1\}.$ And thus, we do not need to go beyond $R_2$ for practical purposes.
\end{rem}
\subsection{Connection of a Lie groupoid}\label{subsection:Connection of a Lie groupoid}
To define a representation (up to homotopy) of a Lie groupoid on its Lie algebroid, we need to introduce the notion of a connection on a Lie groupoid. One may say that the underlined idea is motivated by the case where we have a space of smooth paths $PM$ over a manifold $M.$ Then a connection $\nabla$ on the vector bundle $TM\to M,$ associates a linear isomorphism ${\nabla}_{\gamma}\colon T_{s(\gamma)}M\to T_{t(\gamma)}M$ with a path $\gamma,$ known as parallel transport. Now on a Lie groupoid $\mbbX,$ a connection has to associate a map $T_{s(g)}X_0\to T_{t(g)}X_0$ for any $g\in X_1.$ 

Recall given a Lie group $G$ and an element $g\in G$ we have the adjoint action  by ${\rm Ad}_g\colon G\to G, h\mapsto ghg^{-1}.$ The differential of this map at $h=e$ provides the adjoint representation of $G$ on its Lie algebra $L(G)\simeq T_e(G),$ $$\bigl(d{\rm Ad}_g\bigr)_{e}\colon L(G)\to L(G).$$

Now, consider a Lie groupoid $\mbbX$ and its associated Lie algebroid $A=1^*({\text{ker}} \, ds) \rightarrow X_0$ as constructed in \cref{example:associated Lie algebroid}.  To give a representation of $\mathbb{X}$ on $A \rightarrow X_0$, one has to give a linear isomorphism $\lambda_g\colon A_{s(g)} \rightarrow A_{t(g)}$, for each $g \in X_1$.  Suppose $g \in X_1$ and $\alpha \in A_{s(g)}= {\text{ker}} \, (ds)_{1_{s(g)}}$.  
However, the left action defines a map $t^{-1}(s(g)) \rightarrow t^{-1}(t(g)), h\mapsto gh,$ and the differential of which gives a map on $T_{1_{s(g)}}t^{-1}(s(g))={\text{ker}} \, (dt)_{1_{s(g)}}.$ But $u$ may not lie in $T_{1_{s(g)}}t^{-1}(s(g)).$ Thus, to define an action, we need some unique way of projecting an arbitrary vector in $T_g X_1$ onto ${\rm ker} (ds_g)\subset T_g X_1.$

\begin{defn}[Connection on a Lie groupoid]\label{defn: connection on a lie groupoid}
    An Ehresmann connection on $\mbbX$ is a smooth subbundle $\mathbb{H} \subseteq TX_1$ which complements ker $ds$, that is $TX_1=\mathbb{H} \oplus {\text{ker}} \, ds$ and has the property that $\mathbb{H}_{1_x}=(d1)_x(T_xX_0)$, for every $x \in X_0$.
\end{defn}

It is evident that the associated Lie algebroid naturally fits into a short, exact sequence of vector bundles:
\begin{equation}\label{sequence diagram:short exact sequence with right translations}
\begin{tikzcd}
    0 \arrow[r] & t^*A \arrow[r, "r"] & TX_1 \arrow[r, "ds"] & s^*TX_0 \arrow[r] & 0.
\end{tikzcd}
\end{equation}
The map $r\colon t^*A \rightarrow TX_1$ is given by the right translations: 
\begin{equation}\label{equation:small r map}
    r\left(g,t(g),1_{t(g)},\alpha\right)=\left(g,(dR_g)_{1_{t(g)}}(\alpha)\right),
\end{equation}
 where $R_g\colon s^{-1}(t(g)) \rightarrow s^{-1}(s(g))$ is as in \cref{Equation:Action of fixed arrow}. Observe that the inverse map of $R_g$ is given by $R_{g^{-1}}\colon s^{-1}(s(g)) \rightarrow s^{-1}(t(g))$, $R_{g^{-1}}(h)=hg^{-1}$, and
thus, 
\begin{equation}\label{eqn:Inverse of dR_g}
\begin{split}(dR_g)_{1_{t(g)}}^{-1}={} &(dR_{g^{-1}})_{g}.  
\end{split}
\end{equation}

A more explicit description of  $dR_g$ is given by 
\begin{equation}\label{eqn:description of right translations at tangent level}
    \begin{split}
        (dR_g)_{1_{t(g)}}(\alpha)={} & \bigl(1_{t(g)}.g,\alpha.0\bigr).
    \end{split}
\end{equation}
Considering this, one can see that the three statements below are equivalent. 
\begin{prop}
The following statements are equivalent. 
\begin{enumerate}
     \item An Ehresmann connection $\mathbb{H}$ on $X$.
     \item A smooth left splitting $\sigma:s^*TX_0 \rightarrow TX_1$ such that $\sigma_{1_x}=(d1)_x$, for every $x \in X_0$.
     \item A smooth right splitting $\psi\colon TX_1 \rightarrow t^*A$  such that $\psi_{1_x}={\rm Id}-d(1 \circ s)_{1_x}$, for every $x \in X_0$.
\end{enumerate}  
\end{prop}
In particular, they are interrelated via the following equations.
\begin{equation}\label{eqn:relation between three definition of connections on lie groupoid}
    \begin{split}
        \mathbb{H}={\rm Im}(\sigma)=\ker(\psi), \,  & r \circ \psi+\sigma \circ ds={\rm Id}.
    \end{split}
\end{equation}
It is known that all three descriptions are equivalent.  
For the sake of completeness and details, we offer a proof here.
 \begin{proof}\label{eqn:proof of equivalence between three definitions of connection on lie groupoid}
 $(1) \Leftrightarrow (2)$, Suppose $\mathbb{H}$ is an Ehresmann connection on $\mathbb{X}$, then $TX_1=\mathbb{H} \oplus {\text{ker}} \,  (ds)$ and $\mathbb{H}_{1_x}=(d1)_x(T_xX_0)$. For any $u\in T_{s(g)}X_0,$ surjective submersion $s$ gives a $v \in T_gX_1$ such that $(ds)_g(v)=u$, for each $g \in X_1,$ then $\sigma\colon s^*TX_0 \rightarrow TX_1$, $\sigma(g, s(g), u))=v_0,$ where $v_0$ is the $\mathbb{H}_g$ component of $v,$ is a left splitting of \cref{sequence diagram:short exact sequence with right translations} satisfying  $(ds)_{1_x} \circ (d1)_x(u)=u,$ 
 and $\sigma$  does not depend on the choice of $v$.    Conversely, given such a left splitting $\sigma$ of \cref{sequence diagram:short exact sequence with right translations}, Im$(\sigma)$ is an Ehrasmann connection.\\  
$(1) \Leftrightarrow (3)$, Given an Ehresmann connection $\mathbb{H}$, define $\psi:TX_1 \rightarrow t^*A$ as $\psi(g,u)=(dR_{g^{-1}})_g(\text{pr}_2(u))$, where $\text{pr}_2(u)$ denotes the ${\text{ker}} \, ds$ component of $u$ with respect to the connection $\mathbb{H}$.  In the other direction, given such a $\psi$, ker $\psi$ serves the Ehrasmann connection on $\mathbb{X} $.
\end{proof}
\begin{rem}
A Lie groupoid admits a connection by Lemma 2.9 in \cite{cran}. 
\end{rem}
\begin{rem}\label{Remark:Sequence with respect to target}
Of course, one may choose to work with a dual notion of  a connection as a subbundle $\tilde{\mathbb{H}}$ complementing  $\ker dt,$ which  leads to the short exact sequence 
\begin{equation}\label{sequence diagram:short exact sequence with left translations}
\begin{tikzcd}
    0 \arrow[r] & s^*A \arrow[r, "l"] & TX_1 \arrow[r,"dt"] & t^*TX_0 \arrow[r] & 0.
\end{tikzcd}
\end{equation}
Here $l\colon s^*A \rightarrow TX_1$ denotes the counterpart of  \cref{equation:small r map} given by:
\begin{align}\label{equation:small l}
    l(g,s(g),1_{s(g)},\alpha)     ={} & \left(g,d (L_g \circ i)_{1_{s(g)}}(\alpha)\right),
\end{align}                 
  for $L_g \circ i\colon s^{-1}(s(g)) \rightarrow t^{-1}(t(g)), L_g(h)=gh^{-1}.$ 
That is, for every $\alpha \in T_{1_{s(g)}}s^{-1}(s(g)=A_{s(g)}$ we have,
\begin{equation}\label{eqn:description of left translations at tangent level}
   \begin{split}
    l_{g}(\alpha)=(dL_g)_{1_{s(g)}}((di)_{1_{s(g)}}(\alpha))= {} & \big(g.1_{s(g)},0.(di)_{1_{s(g)}}(\alpha)\bigr).
\end{split}
\end{equation}
It is obvious that both notions are equivalent and interlinked via the relation: $\tilde{\mathbb{H}}_{g}=(di)_{g^{-1}}\mathbb{H}_{g}$.
\end{rem}

A Lie groupoid $\mathbb{X}$ equipped with an Ehresmann connection will be denoted as $\bigl(\mathbb{X},\mathbb{H}\bigr)$ or $\bigl(\mathbb{X},\sigma\bigr)$ or $\bigl(\mathbb{X},\psi\bigr).$  Various other (non-equivalent) notions of connections also exist in literature. For example in  \cite{MR4592876}, a connection is defined as an Ehresmann connection along with the extra condition of functoriality given in \cref{quasi action of X on TX_0} on the induced linear maps $\lambda_g: T_{s(g)}X_0 \rightarrow T_{t(g)}X_0$, and this connection has been used for defining a Chern-Weil theory for a principal Lie group bundle over a Lie groupoid. A smooth subgroupoid $E \rightrightarrows TX_0$ of $TX_1 \rightrightarrows TX_0$ has been called a connection in \cite{BEHREND2005583}, and the author studies cofoliation on a stack. In \cite{MR2238946}, another notion of connection, namely \'etalification of a Lie groupoid, has been used for introducing pseudo-\'etale groupoids.  In \cite{doi:10.1142/S0219199721500929}, a connection on a Lie groupoid $\mbbX$ is a connection on the tangent bundle $TX_1 \rightarrow X_1$ along with compatibility conditions, which allows authors to relate with a connection on the tangent algebroid.
The definition introduced in \cite{cran} is perhaps the most general among all the others, and its existence is guaranteed. Throughout this paper, we will stick to that definition.  In some specific examples, we mention other notions of connections to relate to other works. 
Fixing an Ehresmann connection immediately provides a quasi action of $\mathbb{X}$ on both $A$ and $TX_0$ as follows
\begin{defn}\label{defn:quasi action of X on both A and TX_0}
    Suppose $\mathbb{H}$ is an Ehresmann connection on $\mbbX$. Then there is a natural quasi action of $\mathbb{X}$ on both $A$ and $TX_0,$ for respective anchor maps $A\to X_0, (x, 1_x, \alpha)\mapsto x$  and $TX_0\to X_0, (x, v)\mapsto x,$  given by
    \begin{align}
        \lambda_g^0(\alpha)= {} & -\psi_g(l_g(\alpha)) \label{quasi action of X on A},\qquad  \alpha \in A_{s(g)},\\
        \lambda_g^1(v)={} & (dt)_g(\sigma_g(v)) \label{quasi action of X on TX_0}, \qquad v \in T_{s(g)}X_0,
        \end{align}
    where $\sigma$ and $\psi$ denote the respective left and right splittings of \cref{sequence diagram:short exact sequence with right translations} associated to $\mathbb{H}.$
\end{defn}
\begin{rem}\label{remark:ordinaruadjointaction}
Observe that for a Lie groupoid of the form $G \rightrightarrows \star$, for a Lie group $G$, the above quasi action reduces to the adjoint action of $G$ on its associated Lie algebra $L({G})$.    
\end{rem}  

Given a connection $\sigma\colon s^*TX_0\to TX_1$ one can measure the deviation  $\lambda_{g h}^1-\lambda_g^1 \lambda_h^1,$ when $g, h$ are composable. For that, suppose  $g,h \in X_1$ are composable and $v \in T_{s(h)}X_0$.  Taking differential for the source consistency of the functor $s(g  h)=s(h)$, we see the vector $\sigma_{gh}(v)-\underbrace{(dm)_{g,h}\Bigl(\bigl(g,\sigma_g(\lambda_h^1(v)\bigr),\bigl(h,\sigma_h(v)\bigr)\Bigr)}_{\bigl(gh, 
 \sigma_g(\lambda_h^1(v))\sigma_h(v)\bigr)\,{\rm by\,notation\, in\, \cref{notation:differential of composition map}}}$ belongs to ${\text{ker}} \,(ds)_{gh}=\text{Im}\, r_{gh}$.  That means there exists a vector  $K_{\sigma}^{bas}(g,h)(v) \in (t^*A)_{gh}=A_{t(gh)}=A_{t(g)}$ satisfying
\begin{equation}\label{eqn:measure of failure of closedness of connection}
\begin{split}
    \sigma_{gh}(v)-\bigl(gh, 
 \sigma_g(\lambda_h^1(v)).\sigma_h(v)\bigr)    
    = & (dR_{gh})_{1_{t(gh)}}(K_{\sigma}^{bas}(g,h)(v)).
\end{split}
\end{equation}
In other words,  $K_{\sigma}^{bas}(g,h)(v)=-\psi_{gh}(gh, \sigma_g(\lambda_h^1(v)).\sigma_h(v))$.   
\begin{defn}[Definition 2.12 in \cite{cran}]
    Given a Lie groupoid with connection $(\mbbX, \mathbb H)$, the associated $\mathit{basic}$ $\mathit{curvature}$ is defined as 
    \begin{align}\label{eqn:basic curvature}
K_{\sigma}^{bas}(g,h)(v)=-\psi_{gh}(gh, \sigma_g(\lambda_h^1(v)).\sigma_h(v)).
    \end{align}
A connection with identically vanishing basic curvature is called  \textit{Cartan}. 
\end{defn}
 Particularly, the connection defined in \cite{MR4592876} is Cartan.  It is evident from the following proposition, due to  Abad, that the basic curvature measures the failure of the functoriality condition (mentioned in \cref{eqn:functoriality of representation}) for the (quasi) actions of \cref{defn:quasi action of X on both A and TX_0}.
\begin{prop}[Proposition 2.15 in \cite{cran}]\label{properties of Ehresmann connection}
   Let $(\mbbX, \mathbb H)$   be a Lie groupoid with connection.  Then
  \begin{enumerate}
      \item The anchor of the associated Lie algebroid $\rho\colon A=1^*{{\text{ker}}} \, ds \rightarrow TX_0, (x, 1_x, \alpha)\mapsto (x, dt_{1_x}(\alpha))$ is equivariant. 
      \begin{align}
      \rho(\lambda_g^0(\alpha))={} & \lambda_g^1(\rho(\alpha)),\label{eqn: rho comp lambda=lambda comp rho}
      \end{align} for every $\alpha \in A_{s(g)}$ and $g \in X_1.$
     \item For all $x \xrightarrow[]{h} y \xrightarrow[]{g} z, \, v \in T_xX_0, \, \alpha \in A_x $, we have
     \begin{align}
         \lambda_{gh}^1(v)- \lambda_g^1\lambda_h^1(v)={} & (dt)K_{\sigma}^{bas}(g,h)(v), \label{eqn:error of quasi action on TX_0}\\
         \lambda_{gh}^0(\alpha)-\lambda_g^0\lambda_h^0(\alpha)={} & K_{\sigma}^{bas}(g,h)(dt)(\alpha). \label{eqn:error of quasi action on A}
     \end{align}
     \item For all $x \xrightarrow[]{k} y \xleftarrow[]{h} z \xrightarrow{g} w$, \,  the basic curvature $K_{\sigma}^{bas}$ satisfies following compatibility condition:
     \begin{equation}\label{eqn: compatability conditions of basic curvature}
         \lambda_g^0K_{\sigma}^{bas}(h,k)-K_{\sigma}^{bas}(gh,k)+K_{\sigma}^{bas}(g,hk)-K_{\sigma}^{bas}(g,h)\lambda_k^1=0.
     \end{equation}
  \end{enumerate}
\end{prop}
\begin{rem}
    Here, we make a note that we believe there is a minor sign error in the corresponding equations of \cref{eqn:error of quasi action on TX_0}, \cref{eqn:error of quasi action on A} in Proposition 2.15  \cite{cran}.
\end{rem}
 Using \cref{characterization of representation upto homotopy} and \cref{properties of Ehresmann connection} one deduces:
\begin{prop}[Proposition 3.14 in \cite{cran}]\label{prop:adjoint representation}
    Suppose  $(\mathbb{X}, \mathbb{H})$ is a Lie groupoid  with a connection $\mathbb{H}$, then
    \begin{equation}
            R_0=\rho,  R_1=\lambda , R_2=K_{\sigma}^{bas}
    \end{equation} 
    gives a  unital representation up-to  homotopy of $\mbbX$  on the graded bundle $A \oplus TX_0.$
\end{prop}
\begin{rem}

In \cite{cran}, the proof of unitality of the action above is not provided. One can see unitality as an outcome of  \cref{lem:crucial lemma proving unitality of quasi action}.
\end{rem}
We deduce the following consequence of the unitality here, which comes in quite handy in a later section.  
\begin{prop} \label{prop:measure of failure of g and unit of s(g)}
With notations as above, for any $\alpha \in A_x$ and $v \in T_xX_0$, we have 
\begin{equation}
        \begin{split}
            \sigma_{g}(dt)_{1_{s(g)}}(\alpha)={} & \Bigl(g1_{s(g)}, \sigma_g(\lambda_{1_{s(g)}}^1(dt)_{1_{s(g)}}(\alpha)\bigr).\bigl(\sigma_{1_{s(g)}}dt_{1_{s(g)}}(\alpha)\bigr)\Bigr)\\
            =& \Bigl(\bigl(g,\sigma_g(dt)_{1_{s(g)}}(\alpha). (d(1 \circ t))_{1_{s(g)}}(\alpha)\bigr)\Bigr).
        \end{split}
    \end{equation}
\end{prop}
\begin{proof}
Straightforward. 
\end{proof}
Lie groupoid connections naturally appear in many different contexts. Here, we give examples of a few.  We refer to  \cite{BEHREND2005583, cran, MR4592876} for other examples.
\begin{example}\label{example:Connection on an etale groupoid}
Since the source map of an \'etale groupoid $\mbbX$ is a local diffeomorphism, it has a unique connection given by $H_g=0$ for each $g \in X_1$.  In particular, the Lie groupoid $\left[M \rightrightarrows M\right]$ (see \cref{example:manifold as groupoid}) has only the trivial connection.
\end{example}
\begin{example}\label{example:Connection on pair groupoid}
    Let $[M \times M \rightrightarrows M]$ be the pair groupoid as seen in \cref{example:pair groupoid}.  For this groupoid,  ${\text{ker}} \, ds$ is  the submanifold $\{\bigl((m_2,v),(m_1,0)\bigr) \mid m_2,m_1 \in M, \, v \in T_{m_2}M\} \subset TM \times TM.$ Thus, the distribution $\mathbb{H}$ on $M \times M$ defined by $H_{(m_2,m_1)}={0} \times T_{m_1}M$ is a connection on this pair groupoid.  Furthermore, the corresponding left splitting $\sigma$ of \cref{sequence diagram:short exact sequence with left translations} is given as: $\sigma_{(m_2,m_1)}:(\text{pr}_1)^*(TM) \rightarrow TM \times TM$, $\sigma(m_2,m_1,m_1,u)=(m_2,0,m_1,u).$ Notice that the corresponding quasi representation $\lambda$ of $\mathbb{H}$ given by $\lambda_{(m_2,m_1)}:T_{s(m_2,m_1)}M \rightarrow T_{t(m_2,m_1)}M$ is trivial.  For,
    \begin{align}\nonumber
    \lambda_{(m_2,m_1)}(u)={} & (dt)_{(m_2,m_1)}\sigma_{(m_1,m_2)}(u) \notag \\
                          ={} & (d\text{pr}_1)_{(m_2,m_1)}\bigl((m_1,0),(m_2,u)\bigr)=0. \notag \\
    \end{align}
    Moreover, this is a Cartan connection.
   % Hence, the vector 
    %\begin{align*}
    %(dm)_{(m_3,m_2),(m_2,m_1)}& \Bigl(\bigl(m_3,m_2),\sigma_{(m_3,m_2)}(\lambda_{(m_2,m_1)})(u)\bigr),\bigl((m_2,m_1),\sigma_{(m_2,m_1)}(u)\bigr)\Bigr) \notag \\
   % ={} & (dm)_{(m_3,m_2),(m_2,m_1)}\Bigl(\bigl((m_3,m_2),(0,0)\bigr),\bigl((m_2,m_1),(0,u)\bigr)\Bigr) \\
    %={} & (m_3,0,m_1,u) \\
    %={} & \sigma_{(m_3,m_1)}(u)
    %\end{align*}
    %where $m$ denotes the groupoid multiplication.  As a result, we get that the above connection is a Cartan connection.
    \end{example}
\begin{example} \label{example:connection on action groupoid}
For the action groupoid $[P \times G \rightrightarrows P]$ of \cref{example:action groupoid} ${\text{ker}} \, ds$ is given by the subbundle $\bigl\{(0, v) \mid v \in T_g G\bigr\} \subseteq TP \times TG.$ The smooth distribution $TP \times {0}$ is a Cartan connection on $[P \times G \rightrightarrows P].$
\end{example}
\begin{example}\label{connection on vector bundle seen as a groupoid}
Consider the Lie groupoid $[E \rightrightarrows M]$ associated to a smooth vector bundle  $p\colon E \rightarrow M$ as seen in \cref{example:vector bundle as lie groupoid}.  Then a connection $\mathbb{H}\subset TE$  on the vector bundle $p\colon E \rightarrow M$ splits $TE=\mathbb{H} \oplus {\text{ker}} \, dp,$ and thus defines a Cartan connection on the Lie groupoid 
$[E \rightrightarrows M].$
\end{example}
\begin{example}
    Given a connection $\mathcal{H}\subset TP$   on a principal Lie group $G$-bundle $\pi:P \rightarrow M,$
 \begin{align}\label{eqn:connection on atiyah groupoid}
        \Bar{\mathbb{H}}_{[p,q]} = {} & \dfrac{{\mathcal H}_p \oplus T_qP}{T_{p,q}{\mathcal O}_{(p,q)}} \subset T_{[p,q]}\dfrac{P \times P}{G},
    \end{align}    
where ${\mathcal O}_{(p, q)}=(p,q).G$ is the orbit of $(p,q) \in P \times P$ under the diagonal action $(p, q)\cdot g=(pg, qg),$ defines a Cartan connection on 
the Atiyah groupoid $[\frac{P \times P}{G} \rightrightarrows M]$ (\cref{{example:Atiyah groupoid action}}). 
\end{example}

%%%%%%%%%%%%%%%%%%%%%SECTION%%%%%%%%%%%%%%%%%%%%%%%%%%%%%%%%%%%%%

\section{Diffeological spaces}\label{section:Diffeological spaces}
   Although smooth manifolds and smooth maps have a rich history and classical status, they are often unsuitable in certain circumstances. For instance, the fibre product of two smooth manifolds fails to be a smooth manifold, and likewise, for the orbit space of the action of a Lie group on a smooth manifold. However, it is sometimes necessary to have some notion of ``smoothness'' on such spaces. The ``Diffelogical spaces provide the required framework''. One can consult \cite{iglesias2013diffeology} for a more detailed exposition. This paper will only focus on the absolute necessity and leave out the technical details. 
\begin{defn}\label{defn:diffeological space}
    A {\textit {diffeological space}} $(X, D_X)$ is a set $X$ with  a collection of maps $\{U \rightarrow X\}:=D_{X}$, where each $U$ is an open subset of $\mathbb{R}^n$ for some $n$ satisfying the following three axioms:
    \begin{enumerate}
        \item (covering) Every constant map is in $D_X$.
        \item (smooth compatability) If $p:U \rightarrow X$ is in $D_X$ and $f:V \rightarrow U$ is a smooth map, where $V \subseteq \mathbb{R}^m$ is open, then $p \circ f\colon V \rightarrow X$ is also in $D_X$.
        \item (local) If $U=\bigcup\limits_{i \in I}$ is an open cover for $U$ and $p:U \rightarrow X$ is s map such that $p\restriction_U:U_i \rightarrow X$ is in $D_X$, then $p$ lies in $D_X$.
    \end{enumerate}
  The elements of $D_X$ are called \textit{plots}.
A map $f\colon X \rightarrow Y$ is called \textit{a smooth map between diffeological spaces} $(X, D_X)\rightarrow (Y, D_Y)$ if 
for any plot $p: U \rightarrow X$, the composite map $f \circ p: U \rightarrow Y$ is in $D_Y$.  We denote the category of diffeological spaces and smooth maps by  ${\mathfrak  {Diffeo}}$.
\end{defn}

\begin{example}\label{example:manifold as a diffeological space}
    Any $n$-dimensional smooth manifold $M$ is a diffeological space with a plot given by a usual smooth map $U\to M,$ for an open subset $U\subset {\mathbb R}^m.$ Also, a smooth map between two such diffeological manifolds is the same as a smooth map between smooth manifolds.  Thus, the category of smooth manifolds embeds into the category ${\mathfrak  {Diffeo}}$ in a natural way: $\mathfrak{Man}\subset \mathfrak{Diffeo}.$
\end{example}
\begin{example}[Pullback diffeology]\label{example:pullback diffeology}
    Suppose $f: Y \rightarrow X$ is a set map and $D_X$  a diffeology on $X$.  Then, one can pullback the diffeology $D_X$ to give a diffeological structure for $X$.  A map $p: U \rightarrow Y$ is a plot if and only if the composite map $f \circ p: U \rightarrow X$ lies in $D_X$.  This gives a diffeological structure on $Y$, called {\textit pullback diffeology} and denoted as $f^*D_X$.  Note that $f\colon (Y,f^*D_X) \rightarrow (X, D_X)$ is a smooth map. 
\end{example}
\begin{example}[Subspace diffeology]\label{example:subspace diffeology}
 A subset $Y\subset X$ of a diffeological space  $(X, D_X)$  inherits a natural diffeological structure by the pullback diffeology along the inclusion map $i: Y \rightarrow X;$ that is, the plots of subspace diffeology are $f: U \rightarrow X \in D_X$ such that $f(U) \subseteq Y$. 
\end{example}
\begin{example}[Product diffeology]\label{example:product diffeology}
   Let $(X,D_X)$ and $(Y,D_Y)$ be two diffeological spaces.  Then, the product set $X \times Y$ has a natural diffeology on it given by the product of diffeologies $D_X \times D_Y=\{(f,g) \mid f \in D_X \, \text{and} \, g \in D_Y\}$. 
\end{example}
\begin{example}[Fibre product diffeology]\label{example:fibre product diffeology}
    One of the biggest advantages of working with the category of diffeological spaces and smooth maps is that ${\mathfrak  {Diffeo}}$  admits fibre products.   Suppose $f:(X, D_X) \rightarrow (Z, D_Z)$ and $g:(Y, D_Y) \rightarrow (Z, D_Z)$ are smooth maps between diffeological spaces, then the fibre product set $X \times_{Z} Y$ has a natural diffeology inherited as a subspace diffeology of the product diffeology. Explicitly,  $X \times_{Z} Y$ is a diffeological space with the diffeology $D_X \times_{D_Z} D_Y=\{(f_1,g_1) \mid f \circ f_1=g \circ g_1, f_1 \in D_X,g_1 \in D_Y\}.$
\end{example}
\begin{example}[Pushforward diffelogy]\label{example:pushforward diffeology}
    Let $(X, D_X)$ be a diffeological space and $f:X \rightarrow X'$ be a set map.  There is a natural way of obtaining a diffeological structure on $X'$ in the following manner.  A map $p: U \rightarrow X'$ is a plot if and only if, for every $u \in U$, there exists an open neighbourhood $V$ of $u$ and a plot $q: V \rightarrow X$ in $D_X$ such that $p\restriction_V=f \circ q$.  Typically, the pushforward diffeology along a map $f:X \rightarrow X'$ is denoted by $f_*(D_X)$.
\end{example}
\begin{example}[Quotient diffeology]\label{example:quotient diffeology}
    Let $(X, D_X)$ be a diffeological space and $\sim$ be an equivalence relation on $X$.  Then pushforward along the canonical quotient map $q_X\colon X \rightarrow \dfrac{X}{\sim}$ induces a diffeology on  $\dfrac{X}{\sim},$ called the \textit{quotient diffeology}.   Hence, a map $f:U \rightarrow \dfrac{X}{\sim}$ is a plot if and only if  for every $u \in U$, there exists an open neighbourhood $V$ of $u$ and a plot $g:V \rightarrow X$ such that $f \restriction_V=q_X \circ g.$
  
  \end{example}

    A smooth, surjective  map $f:(X,D_X) \rightarrow (Y,D_Y)$ is called a \textit{subduction} if 
    the diffeology $D_Y$ coincides with the pushforward diffeology $f_{*}(D_X)$.
    In particular, the canonical quotient map $q_X:X \rightarrow \dfrac{X}{\sim}$ in \cref{example:quotient diffeology} is a subduction.
\begin{defn}[Definition 8.3 in \cite{iglesias2013diffeology}]\label{defn:diffeological groupoid}
    A {\textit {diffeological groupoid}} is a groupoid $X_1 \rightrightarrows X_0$, internal to the category of diffeological spaces.
\end{defn}
Explicitly, the sets $X_1, X_0$ are diffeological spaces along with the structure maps $s, t, m, I, u$ all smooth in a diffeological groupoid $\mbbX$.  Recently, the abelianization of Lie and diffeological groupoids has been investigated in \cite{MR4828414}. 
\begin{rem}\label{rem:s,t of diffeological groupoid are subductions}
    Observe that a diffeological groupoid's source and target maps automatically become subductions.
\end{rem}
Generalization of the examples of Lie groupoids given between \cref{example:manifold as groupoid}--\cref{example:action groupoid}
to diffeological groupoids is straightforward. 
%\begin{example}[trivial diffeological %groupoid]\label{example:trivial diffeological groupoid}
%  Suppose $(X, D_X)$ is a diffeological space, then the groupoid $X \rightrightarrows X$ is a diffeological groupoid with source and target maps are identity maps.  This groupoid has only identity arrows.
%\end{example}
%\begin{example}\label{example:diffeological action groupoid}
 %   Let $(X,D_X)$- a diffeological space and $G$ be a diffeological group(A group object in the category of diffeological spaces) acting on the space $X$.  Then, $X \times G \rightrightarrows X$ is a diffeological groupoid with source and target map given by $s(x,g)=x$ and $t(x,g)=x.g$ respectively. It is called as \textit{diffeological action groupoid}.
%\end{example}
\begin{defn}[Definition 4.5 in \cite{MR3467758}]\label{defn:definition of diffelogical vector space}
    Let $(X, D_X)$ be a diffeological space.  A diffeological vector space over $(X,D_X)$ is a diffeological space $(V,D_V)$ along with a smooth map $\pi:(V,D_V) \rightarrow (X,D_X)$ satisfying following conditions:
    \begin{itemize}
        \item $\pi^{-1}(x)$ is a vector space for each $x \in X$.
        \item The vector space addition map $+:(V \times_{X} V,D_{V \times_{X} V}) \rightarrow (V,D_V)$ is a smooth map.
        \item scalar multiplication map $*:(R \times V) \rightarrow V$ is a smooth map. 
        \item zero section of $\pi:V \rightarrow X$ is a smooth map.
    \end{itemize}
\end{defn}
 The notion of diffeological vector space over a diffeological space has been studied in \cite{MR3467758}, and the same object with the name \textit{diffeological vector pseudo bundle} has appeared in \cite{PERVOVA2016269}. Any ordinary smooth vector bundle is an example of a diffeological vector space, and some other nontrivial examples can be found in  \cite{PERVOVA2016269}. Also, in \cref{section:Atiyah sequence}, we will see that associated with a Lie groupoid principal bundle is a sequence of diffeological groupoids with vector space structure.

%%%%%%%%%%%%%%%%%SECTION%%%%%%%%%%%%%%%%%%%%%%%

\section{Atiyah sequence of a principal Lie groupoid bundle}\label{section:Atiyah sequence}
    In  \cref{section:Principal Lie Groupoid Bundle over manifolds}, we have defined a principal Lie groupoid bundle. This paper's primary goal would be to introduce a connection structure on a Lie groupoid bundle. Another of our objectives is to investigate the relation between our connection on a Lie groupoid bundle with a particular kind of sequence, analogous to the so-called \textit{Atiyah sequence} on a classical principal bundle. 
    
     For that, we briefly go through the setup of classical principal bundles. 
Let $G$ be a Lie group and $\pi\colon P \rightarrow M$ a principal $G$-bundle.  For each $p \in P$, there exists a diffeomorphism $\delta_p:  G   \rightarrow  \, \pi^{-1}(\pi(p)),
   g  \mapsto pg,$
and the differential of that map at the identity $e;$ $d({\delta_p})_e\colon L(G)\to T_p P$ defines a map of vector bundles $d\delta\colon P\times L(G)\to TP, (p, X)\mapsto (p, d(\delta_p)_e(X)),$ called \textit{fundamental vector field map}. In turn, we have a natural short exact sequence of vector bundles over $P$:
  $
 \begin{tikzcd}
   0 \arrow[r] & P \times L(G) \arrow[r,"d\delta"] & TP \arrow[r,"d\pi"] & \pi^*TM \arrow[r] & 0,   
  \end{tikzcd}.
 $
Quotient by the natural $G$ actions on each of these bundles reduces the above short exact sequence of vector bundles 
 to a short exact sequence of bundles over $M$ \cite{MR86359}.
\begin{equation}\label{sequence diagram:Atiyah sequence in classical setup}
  \begin{tikzcd}
  0 \arrow[r] &  {{\rm Ad}(P)} \arrow[r,"\Bar{d\delta}"]  & {{\rm At}(P)} \arrow[r,"\Bar{d\pi}"] & TM \arrow[r] & 0, 
  \end{tikzcd}
  \end{equation}
where $\frac{P\times L(G)}{G}:={\rm Ad}(P),$  $\frac{TP}{G}:={\rm At}(P)$ are  called the \textit{adjoint bundle} and the \textit{Atiyah bundle} repectively, while the sequence \cref{sequence diagram:Atiyah sequence in classical setup} is called the \textit{Atiyah sequence.} Interested reader can look at \cite{mackenzie1987lie} for detailed proof and construction. A \textit{connection on the principal $G$-bundle}  is a splitting of the Atiyah sequence.  One can show this description of a connection is equivalent to the following more ``popular'' definitions: 
\begin{enumerate}
        \item A smooth subbundle $\mathcal{H}$ of $TP$ such that $TP=\mathcal{H} \oplus {\text{ker}} \, d\pi$ and $\mathcal{H}_{pg}=\mathcal{H}_p.g$ for every $p \in P$ and $g \in G,$ where $\mathcal{H}_p.g$ denotes the restriction of induced action of $G$ on $TP$ to $\mathcal{H}$.\label{defn:connection as subbundle in classical setup}
        \item An $L(G)$-valued $1$-form $\omega$ on $P$ satisfying; \label{defn:connection as 1-form in classical setup}
        \begin{itemize}
            \item $\omega_p((\delta_p)_e(X))=X$, for every $X \in L(G)$ and $p\in P.$
            \item $(R_g)^*\omega={\rm Ad}_{g^{-1}}\omega$, for every $g \in G,$ for every $g\in G.$\
        \end{itemize}
    \end{enumerate}
The subbundle $\mathcal{H}$ is usually called a \textit{horizontal distribution}, and the vectors lying inside $\mathcal{H}$ are called  \textit{horizontal vectors}.

Now, focusing back on our object of interest, we consider a principal $\mathbb{X}$-bundle $\pi\colon P \rightarrow M$ for a Lie groupoid $\mbbX.$  That is, we have an anchor map $a\colon P \rightarrow X_0$ and an action map $\mu\colon P \times_{X_0} X_1 \rightarrow P$ satisfying the conditions in \cref{defn:torsor}. Let $A$ be the associated Lie algebroid of $\mbbX$ defined in \cref{example:associated Lie algebroid}.
For each $p \in P$, let $\delta_p\colon s^{-1}(a(p)) \rightarrow \pi^{-1}(\pi(p))$ be the diffeomorphism $\delta_p(h)=ph^{-1}$.  Differential at the unit arrow $1_{a(p)}$ produces an isomorphism  $$T_{1_{a(p)}}s^{-1}(a(p))=A_{a(p)} \xrightarrow[]{(d\delta_p)_{1_{a(p)}}} T_p\pi^{-1}(\pi(p))={\text{ker}} (d\pi)_p.$$   Now varying $p$ over $P$, we get the \textit{fundamental vector field map} $d\delta\colon P \times_{X_0} A \rightarrow TP$ given by 
    \begin{equation}\label{eqn:fundamental vector field map}
    \begin{split}
        d\delta(p,\alpha)={} & \bigl(p,(d\delta_p)_{1_{a(p)}}(\alpha)\bigr).
    \end{split}
    \end{equation}
As a consequence, we obtain the following short exact sequence of graded vector bundles over $P.$
\begin{equation}\label{sequence diagram:short exact sequence with fundamental vector field}
\begin{tikzcd}
    0 \arrow[r] & P \times_{X_0} A \oplus a^*TX_0 \arrow[r, "{(d\delta,{\rm Id}})"] & TP\oplus a^*(TX_0) \arrow[r, "d\pi"] & \pi^*TM  \arrow[r] & 0
\end{tikzcd}
\end{equation}
We aim to turn this sequence in \cref{sequence diagram:short exact sequence with fundamental vector field} of graded vector bundles over $P$ to a sequence over $M,$ which involves quotienting with respect to an action of $\mbbX$ on each fibre space. And precisely here, things get interesting! 

The first point to take notice of here is that $\mbbX$ does not have a natural (adjoint) action on its Lie algebroid $A$ and, as discussed in \cref{properties of Ehresmann connection} and \cref{prop:adjoint representation}, one can only define the adjoint representation (a quasi action) as a representation up to homotopy by introducing a Lie groupoid connection $\mathbb{H}$ on $\mbbX.$ Moreover, notice that for a given $g\in X_1,$ the right action $R_g$ is defined on $s^{-1}(a(p)),$ not on the entire $P,$ thus the action of $\mbbX$ on $TP$ is also not obvious. Curiously, a connection $\mathbb{H}$ on $\mbbX$ resolves both the issues and produces a sequence of (diffeological) groupoids over $M$ with graded vector space structure on their fibres. The rest of this section will be entirely devoted to this task. To give a glimpse  of  the main result of this section, we  note that eventually, we shall arrive at the following exact sequence, 
 \begin{equation}
 \begin{tikzcd}\label{sequence diagram:atiyah sequence of categories}
     0 \arrow[r]  & {{\mathcal A\mathfrak d}(P)} \arrow[d] \arrow[r, "\Bar{d\delta}"] & {{\mathcal A}{\mathfrak {t}}(P)} \arrow[d] \arrow[r,"\Bar{d\pi}"] & {\mathcal A\mathfrak {ct}}(\pi^*TM) \arrow[r] \arrow[d] & 0.\\
     0 \arrow[r] & M \arrow[r] & M \arrow[r] & M \arrow[r] & 0,
 \end{tikzcd}    
 \end{equation}
 where $M$ is treated as a discrete category, and $\AdP,$ $\AtP$ and $\ActTM$ are diffeological groupoids over $M$ with graded vector space structure on their fibres, and arrows are smooth functors.

Before proceeding further, we state a useful identity involving the fundamental vector field map \cref{eqn:fundamental vector field map}.  Differential at $1_{a(p)}$ implies
\begin{align}\label{eqn:description of fundamental vector field map}
    (d\delta_p)_{1_{a(p)}}(\alpha) ={} & d\mu_{(p, 1_{a(p)})}[(di)_{1_{a(p)}}\alpha] \notag   \\
     ={} & \bigl(p1_{a(p)},0.(di)_{1_{a(p)}}(\alpha)\bigr), \, \text{for every} \, \alpha \in A_{a(p)}, 
\end{align}
where $i$ is the inverse map, and in the last equality, we have followed the notation in \cref{defn:right action of lie groupoid}.

We begin with the following definition to provide a formal description of \cref{sequence diagram:atiyah sequence of categories}.

\begin{defn} (Equivariant $\mathbb{X}$ vector bundle)\label{defn:quasi equivariant X-bundle over a X-space}
    Let $P$ be a manifold equipped with a right action of $\mbbX$ given by an anchor map $a\colon P \rightarrow X_0$ and an action map $\mu\colon P \times_{X_0} X_1 \rightarrow P$.  A \textit{quasi equivariant $\mathbb{X}$-(vector) bundle over $P$} is a vector bundle $\pi\colon E \rightarrow P$ equipped with a right quasi action $\tau\colon E \times_{X_0} X_1 \rightarrow E$ via the anchor map $a \circ \pi:E \to P$ such that, 
    \begin{enumerate}
        \item $\tau_g:E_p \rightarrow E_{\mu(p, g)}$ is a linear map for every $(p,g)$ satisfying $a(p)=t(g),$
        \item  $\tau_{1_{a(p)}}={\rm Id}_{E_p}$, for every $p \in P$.
    \end{enumerate}

     A \textit{morphism} between a pair of quasi $\mbbX$ vector bundles  $\pi\colon E \rightarrow P$ and  $\pi':E' \rightarrow P$ is a map of vector bundles $\phi\colon E\to E'$ such that the following diagram is commutative for each $(p,g)$ satisfying $a(p)=t(g):$
 \begin{equation}
     \begin{tikzcd}
         E_p \arrow[r,"\Phi_p"] \arrow[d,"\tau_g"] & E'_p \arrow[d,"\tau'_g"] \\
         E_{pg} \arrow[r, "\Phi_{pg}"] & E'_{pg}
     \end{tikzcd}
 \end{equation}   
\end{defn}

A quasi equivariant $\mathbb{X}$-bundle will be  represented by the  diagram,
    \begin{tikzcd}
        E \arrow[d, "\pi"] & X_1 \arrow[d, shift left] \arrow[d, shift right] \\
        P \arrow[r,"a"] & X_0.
    \end{tikzcd}
Let us see some examples of such bundles.
\begin{example}\label{example:Associated Lie algebroid as quasi equivariant X-bundle over X_0}
   In \cref{quasi action of X on A} we saw that an Ehresmann connection on $\mbbX$ induces a left quasi action of $\mbbX$ on its associated Lie algebroid $A$ Which can be turned into a right quasi action by inverting the arrows; that is, $\tau_{g}(\alpha)=\lambda_{g^{-1}}(\alpha)$, for every $\alpha \in A_{t(g)}.$
   Whereas by \cref{example:Action of groupoid on its object space} the space $X_0$ has a  natural right action $X_0 \times_{X_0} X_1 \rightarrow X_0,\,  (x,g)\mapsto s(g)$ of $\mathbb{X}$ for the anchor map ${\rm id}_{X_0}\colon X_0 \rightarrow X_0.$  Thus, it is not hard to see that $A$ is a quasi equivariant $\mathbb{X}$-bundle over $X_0$ with right quasi action 
\end{example}

Our first aim is to turn the short exact sequence in \cref{sequence diagram:short exact sequence with fundamental vector field} into a short exact sequence of quasi equivariant $\mathbb{X}$-bundles upto homotopy which we introduce now.
\begin{defn}(Equivariant bundle upto homotopy)\label{defn: equivariant bundle upto homotopy}
   Let $P$ be a manifold equipped with a right action of $\mbbX$ given by an anchor map $a:P \rightarrow X_0$ and an action map $\mu:P \times_{X_0} X_1 \rightarrow P.$  A \textit{(degree $2$) quasi equivariant $\mathbb{X}$-bundle upto homotopy} is a $2$-degree graded vector bundle $\pi:E=E_0 \oplus E_1 \rightarrow P,$ where $E_0\to P, E_1\to P$ both are equivariant bundles with respective  quasi actions $\tau^0:E_0 \times_{X_0} X_1 \rightarrow E_0$ and $\tau^1:E_1 \times_{X_0} X_1 \rightarrow E^1,$  along with: 
   \begin{enumerate}
       \item a smooth vector bundle map $\xi:E_0 \rightarrow E_1,$
       \item a smooth map $\Upsilon:E^1 \times_{X_0} (X_1 \times_{X_0} X_1) \rightarrow E^0 \times_{X_0} X_1$(called \textit{homotopy map}), such that for a fixed $(g, h)\in X_1\times_{X0}X_1,$ restriction of $\Upsilon$ to  each fiber $E^1_p$ gives a linear map $\Upsilon(g,h)\colon E^1_p \rightarrow E^0_{pgh}.$      
   \end{enumerate}
Further the following compatibility conditions have to hold for appropriate  $p \in P$ and $g,h,k \in X_1.$
\[
\begin{tikzpicture}
    \matrix(a)[matrix of math nodes,
row sep=6em, column sep=4em,
text height=1.5ex, text depth=0.25ex]
{ (E^0)_p & (E^1)_{p}  \\
   (E^0)_{pgh} & (E^1)_{pgh}  \\};
   \path[->](a-1-1) edge node [font=\tiny,above]{$(\xi)_p$} (a-1-2)
            (a-2-1) edge node[font=\tiny, below]{$(\xi)_{pgh}$} (a-2-2);
         \path[->, transform canvas={xshift=-3mm}]
          (a-1-1) edge node[font=\tiny, left]{$\tau^0_{gh}$} (a-2-1)
          (a-1-2) edge node[font=\tiny, left]{$\tau^1_{gh}$} (a-2-2);
          \path[->,transform canvas={xshift=-1mm}]
          (a-1-1) edge node[font=\tiny, right]{$\tau^0_h\tau^0_g$} (a-2-1)
          (a-1-2) edge node[font=\tiny, right]{$\tau^1_h\tau^1_g$} (a-2-2);
          \path[->,dashed]
          (a-1-2) edge node[font=\tiny,fill=white]{$\Upsilon(g,h)$} (a-2-1);
\end{tikzpicture}
\]
    \begin{equation}\label{eqn: comptability conditions of equivariant bundle upto homotopy}
    \begin{split}
         &\xi_{pg} \circ \tau_g^0={} \tau_g^1 \circ \xi_p, \\
        &\tau_{gh}^0-\tau_h^0 \circ \tau_g^0 ={}  \Upsilon(g,h) \circ \xi_p, \\
         &\tau_{gh}^1-\tau_h^1 \circ \tau_g^1={}  \xi_{pgh} \circ \Upsilon(g,h), \\
        &\tau_k^0 \circ \Upsilon(g,h)-\Upsilon(gh,k)+\Upsilon(g,hk)-\Upsilon(k,h) \circ \tau_g^1={}  0.
\end{split}
\end{equation}
\end{defn}

A \textit{morphism} between two such quasi equivariant $\mathbb{X}$-bundles upto homotopy $\bigl(E=E_0 \oplus E_1,\xi,(\tau^0,\tau^1),\Upsilon\bigr)$ and $\bigl(F=F_0 \oplus F_1,\xi',(\tau'^0,\tau'^1),\Upsilon'\bigr)$ consists of two vector bundle maps $\Phi_0:E_0 \rightarrow F_0, \, \Phi_1:E_1 \rightarrow F_1$ such that the following conditons hold for appropriate  $p \in P$ and $g,h \in X_1.$
\begin{equation}
    \begin{split}
        & \xi' \circ \Phi_0={}  \Phi_1 \circ \xi, \\
       & (\tau'_g)^0 \circ (\Phi_0)_p={} (\Phi_0)_{pg} \circ \tau_g^0, \, \hspace{0.5cm} (\tau'_g)^1 \circ (\Phi_1)_p ={} (\Phi_1)_{pg} \circ \tau_g^1, \\
       & \Upsilon'(g,h) \circ (\Phi_1)_p={}    (\Phi_0)_{pgh} \circ \Upsilon(g,h).
     \end{split}
\end{equation}
Denote the category of such bundles over $P$  by $\mathcal {Q-EQUIV}(P, \mathbb{X}).$
\begin{rem}
    Every quasi equivariant $\mathbb{X}$-bundle $\pi:E \rightarrow P$ is trivially a quasi equivariant $\mathbb{X}$-bundle upto homotopy.  It is a straightforward generalization to define $n$-degree quasi equivariant $\mathbb{X}$-bundles upto homotopy.  For our purpose, $2$-degree suffices.
\end{rem}
\begin{prop}\label{prop: fibred category related to quasi equivariant bundle}
A quasi-equivariant $\mbbX$-bundle upto a homotopy $\bigl(\pi:E=E^0\oplus E^1 \rightarrow P, \xi, \tau, \Upsilon\bigr)$ defines a fibered category $F$ over $\mathbb{X}.$ The fiber $F(x)$  over  each $x\in X_0$ is a diffeological groupoid.

\end{prop}
\begin{proof}
    For each $x \in X_0$ consider a groupoid $F(x)$ whose objects are $\bigl((v^0\oplus v^1),\gamma)\bigr),$ where $v_0\oplus v_1 \in (a \circ \pi)^{-1}(x)$ and $\gamma \in t^{-1}(x),$ and an arrow is given as $(u^0 \oplus u^1, \gamma, v_0 \oplus v_1, \beta)$ with source and target $ (u^0 \oplus u^1, \gamma),$ and $\bigl((u^0+v^0) \oplus (u^1+v^1),\gamma\beta\bigr)$ respectively.  The composition is given as $\bigl((u^0+v^0)\oplus (u^1+v^1), \gamma\beta,w^0 \oplus w^1,\zeta\bigr) \circ (u^0 \oplus u^1,\gamma,v^0\oplus v^1,\beta)=\bigl(u^0\oplus u^1,\gamma,(w^0+v^0) \oplus (w^1+v^1), \beta\zeta\bigr).$ The category $F(x)$ has an obvious smooth diffeology.

An arrow $x \xrightarrow{g} y \in X_1$ defines a functor $F(g)\colon F(y) \rightarrow F(x)$ sending an object $\bigl(u^0\oplus u^1,\gamma\bigr)$ to $\Bigl(\tau_g^0(u^0)\oplus \tau_g^1(u^1)\bigr),g^{-1}\gamma\Bigr)$ and an arrow  $(u^0 \oplus u^1, \gamma, v^0 \oplus v^1,\beta) $ to $\bigl(\tau_g^0(u^0)\oplus \tau_g^1(u^1),g^{-1}\gamma,\tau_g^0(v^0)\oplus \tau_g^1(v^1),\beta\bigr).$
    The homotopy and compatability relations in \cref{eqn: comptability conditions of equivariant bundle upto homotopy} produce a natural transformation between $F(gh)$ and $F(h) \circ F(g)$ by sending an object $(v^0 \oplus v^1, \gamma)$ to an arrow $\bigl(\tau_{gh}^0(u^0) \oplus \tau_{gh}^1(u^1),h^{-1}g^{-1}\gamma,\bigl(-\Upsilon \circ \xi(u^0)\oplus -\xi \circ \Upsilon(u^1)\bigr),1_{s(\gamma)}\bigr).$  Thus, we have a pseudo functor on $\mathbb{X}$.  The construction of the corresponding fibred category is standard. 
\end{proof}
\begin{rem}\label{rem: relation between action of stacky groupoid and bundle upto homotopy}
    In \cite{MR4139032}, authors have formulated the notion of action of a stacky Lie groupoid $\mathcal{G} \rightrightarrows M$ on a fibered category $\chi$ (definition 3.13 in \cite{MR4139032}), by relaxing functoriality and unitality conditions of a stacky groupoid action by means of a $2$-isomorphism between fibered categories, and defined the ``stacky quotient''. Our notion of quasi equivariant bundles and the quotients described in the following sections rely on the linear homotopy between vector spaces.    However, as \cref{prop: fibred category related to quasi equivariant bundle}  indicates, the two descriptions are closely related.  
    
\end{rem}
Fix a connection $\mathbb{H}\subset TX_1$ on the Lie groupoid $\mbbX.$ We  show that $(\mathbb{X}, \mathbb{H})$ 
allows us to realize   \cref{sequence diagram:short exact sequence with fundamental vector field} as a sequence in the category  $\mathcal {Q-EQUIV}(P, \mathbb{X});$ That is each vector bundle in  \cref{sequence diagram:short exact sequence with fundamental vector field} is quasi equivariant $\mathbb{X}$-bundle upto homotopy and maps are  morphisms of quasi equivariant $\mathbb{X}$-bundles upto homotopy. The theory developed in \cref{subsection:Connection of a Lie groupoid}, more specifically the results for representation up to homotopy on the associated Lie algebroid expressed in \cref{properties of Ehresmann connection} and \cref{prop:adjoint representation} going to be pivotal for this section. 

Throughout this section, $\lambda^0, \, \lambda^1$ denote the respective quasi actions of $\mathbb{X}$ on its associated Lie algebroid $A$ and $TX_0$ (see \cref{quasi action of X on A}, \cref{quasi action of X on TX_0}).

\subsection{Construction of diffeological groupoid $\AdP$}\label{Subsection:Groupoid AdP}
\subsubsection {Making $P \times_{X_0} A \oplus a^*TX_0 \xrightarrow[]{{\rm pr}_1} P$ into a quasi equivariant$\mathbb{X}$-bundle upto homotopy}\label{subsubsection:Making P*A into a quasi equivariant X-bundle}

Define quasi actions of $\mathbb{X}$ on $P \times_{X_0} A$ and $a^*TX_0$ respectively by: 
\begin{equation}\label{eqn:Action on P times A}
\begin{split}
&\tau^0\colon \bigl(P \times_{X_0} A\bigr) \times_{X_0} X_1 \rightarrow P \times_{X_0} A\\
&\bigl((p,\alpha), g\bigr)\mapsto (pg,\lambda_{g^{-1}}^0(\alpha)).\\
&\tau^1 \colon a^*TX_0 \times_{X_0} X_1\to a^*TX_0\\
& (p,a(p),v,g)\mapsto \bigl((pg,a(pg),\lambda_{g^{-1}}^1(v)\bigr).
\end{split}
\end{equation}
For each $g \in X_1$, associated linear maps  $\tau_g^0\colon A_{a(p)} \rightarrow A_{a(pg)}, \, \tau_g^1:T_{a(p)}X_0 \rightarrow T_{a(pg)}X_0$ on the fibres are given by $\tau_g^0(\alpha)=\lambda_{g^{-1}}^0(\alpha)=-(\psi)_{g^{-1}}l_{g^{-1}}(\alpha)$, and $\tau_g^1(v)=\lambda_{g^{-1}}^1(v)=(dt)_{g^{-1}}\sigma_{g^{-1}}(v)$ where $a(p)=t(g).$
Define $\xi:P \times_{X_0} A \rightarrow a^*TX_0$ and the homotopy map $\Upsilon:a^*TX_0 \times_{X_0} \bigl(X_1 \times_{X_0} X_1) \rightarrow (P \times_{X_0} A) \times_{X_0} X_1 $, respectively by  
\begin{align}
   & (p,\alpha)  \mapsto \bigl(p,a(p),(dt)_{1_{a(p)}}(\alpha)\bigr),\label{eqn: xi map for adjoint bundle} \\
   & (p,a(p),v,g,h)  \mapsto  (pgh,a(pgh),K_{\sigma}^{bas}(h^{-1},g^{-1})(v),gh). \label{eqn: Upsilon map for adjoint bundle}
\end{align}

\begin{prop}\label{prop:P times A equivariant}
For the quasi actions in \cref{eqn:Action on P times A}, and  the maps $\xi, \Upsilon$ in \cref{eqn: xi map for adjoint bundle} and \cref{eqn: Upsilon map for adjoint bundle} we have  an equivariant $\mathbb{X}$-bundle upto homotopy $P\times_{X_0}A \oplus a^*TX_0 \to P,\,  (p, \alpha, v)\mapsto p.$ 
\end{prop}
\begin{proof}
    Using \cref{{eqn: rho comp lambda=lambda comp rho},{eqn:error of quasi action on A},{eqn:error of quasi action on TX_0},{eqn: compatability conditions of basic curvature}}, one verifies 
    \begin{equation}\label{eqn:error of quasi action on fibre product of P and A}
  \begin{split}
   \tau_{gh}^0(\alpha)-\tau_h^0\tau_g^0(\alpha)={} & K_{\sigma}^{bas}(h^{-1},g^{-1})(dt)_{1_{s(g^{-1})}}(\alpha),     \hspace{0.5 cm}
 {\rm and}\\
    \tau_{gh}^1(v)-\tau_h^1\tau_g^1(v)={} & (dt)K_{\sigma}^{bas}(h^{-1},g^{-1})(v),
\end{split}
\end{equation}
     confirming the compatibility conditions of \cref{eqn: comptability conditions of equivariant bundle upto homotopy}. 
\end{proof}

\begin{rem}\label{Remark:Adjoint full action Cartan}
Note that if we have chosen the connection that appeared in either of \cite{MR4592876}, \cite{BEHREND2005583}, i.e., A Cartan connection, then the right-hand sides of \cref{eqn:error of quasi action on fibre product of P and A} would have identically vanished. 
\end{rem}

    \subsubsection{Groupoid $\AdP$}  Observe that,   because of the error term appearing in \cref{eqn:error of quasi action on fibre product of P and A}, the usual orbit space $\dfrac{P \times_{X_0} A \oplus a^*TX_0}{\mathbb{X}}$ is not well defined.  Hence, we first eliminate errors by taking the quotient of the subspace generated by those errors on each fibre of $P\times_{X_0} A \oplus a^*TX_0.$
Define a relation on the manifold $P \times_{X_0} A \oplus a^*TX_0$ as follows: $(p,\alpha,v) \sim (p',\alpha',v')$ if $p'=p$ and $\alpha-\alpha'$  and $v-v'$ can be written as a finite sum of errors of respective quasi action. 
\begin{equation}\label{eqn: equivalance relation on P times A plus TX_0}
\begin{split}
    i.e.\hspace{1cm} \alpha-\alpha '={} & \sum_{i=1}^{n}\lambda_{h_i^{-1}g_i^{-1}}^0(\alpha_i)-\lambda_{h_i^{-1}}^0\lambda_{g_i^{-1}}^0(\alpha_i), \\
    \hspace{1cm} v-v'={} & \sum_{i=1}^{n}\lambda_{h_i^{-1}g_i^{-1}}^1(v_i)-\lambda_{h_i^{-1}}^1\lambda_{g_i^{-1}}^1(v_i).
\end{split}
\end{equation}
for some $(p_i,\alpha_i,v_i) \in P \times_{X_0} A \oplus a^*TX_0$ with $p=p_ig_ih_i,$ for every $1\leq i \leq n.$
This relation can be easily seen to be an equivalence relation.    Now, we will construct a diffeological groupoid $\AdP.$ Define
\begin{equation} \label{Eqn:Ad(P) object morphism}
\begin{split}   
&{\rm {Obj}}\bigl(\AdP\bigr)=\dfrac{P \times_{X_0}A \oplus a^*TX_0}{\sim}=\Bigl\{[p,\alpha,v]  \, \bigl| \, \alpha \in A_{a(p)}, v \in T_{a(p)}X_0\Bigr\}\\
&{\rm {Mor}}\bigl(\AdP\bigr)=\dfrac{P \times_{X_0} A \oplus a^*TX_0}{\sim} \times_{X_0} X_1=\biggl\{\Bigl([p, \alpha, v], g\Bigr) \, \Bigl| \, a(p)=t(g),\alpha \in A_{a(p)}, v \in T_{a(p)}X_0\biggr\}, 
\end{split}
\end{equation}
with the  source, target maps  
\begin{align}\label{eqn:structural maps of Ad(P)}
%\begin{split}
\tilde{s}\Bigl([p, \alpha, v], g\Bigr)={}  [p, \alpha],  & \, \tilde{t}\Bigl([p, \alpha, v],g\Bigr)={}  [pg,\lambda_{g^{-1}}^0(\alpha), \lambda_{g^{-1}}^1(v)],
%\end{split}
\end{align} 
and the other structure maps are defined naturally.
One verifies that the structure maps are well-defined.   Thus we have a groupoid $\AdP$.  
\subsubsection{Diffeology on $\AdP$}\label{Subsection:Diffeology on  AdP}
We impose a diffeology on the object and morphism sets of $\AdP$ such that the structure maps are smooth.

We give the quotient diffeology on $\dfrac{P \times_{X_0} A \oplus a^*TX_0}{\sim}$ (\cref{example:quotient diffeology}) induced from the smooth manifold structure on $P \times_{X_0} A \oplus a^*TX_0.$ The map $\Bar{\text{pr}}_1:\dfrac{P \times_{X_0} A \oplus a^*TX_0}{\sim} \rightarrow P$ which sends $[p, \alpha, v]$ to $p$ is a smooth map. This can be seen by considering a plot $f: U \rightarrow \dfrac{P \times_{X_0} A \oplus a^*TX_0}{\sim}$ on the quotient diffeology.  Then, for each $u \in U,$ there exists a neighbourhood $V$ of $u$ and a smooth map $(f_1,f_2)\colon V \rightarrow P \times_{X_0} A \oplus a^*TX_0$ such that $f\restriction_V=q \circ (f_1,f_2).$  That is, we have the following commutative diagram
\begin{equation}
    \begin{tikzcd}
        & V \subseteq U \arrow[d,shift right=4, "{f\restriction_V}"] \arrow[dl, "{(f_1,f_2)}"'] \\
        P \times_{X_0} A \arrow[r,"q"'] \arrow[dr,"\text{pr}_1"'] & \dfrac{P \times_{X_0} A}{\sim} \arrow[d,"\Bar{\text{pr}}_1"]  \\
        & P
    \end{tikzcd}
\end{equation}
Thus, we have the smooth map $\Bar{\text{pr}}_1 \circ f\restriction_V=\Bar{\text{pr}}_1 \circ q \circ (f_1,f_2)=\text{pr}_1 \circ (f_1,f_2)$, and the map $\Bar{\text{pr}}_1 \circ f:U \rightarrow P$ is a smooth map between manifolds $U$ and $P.$ In other words  $\Bar{\text{pr}}_1 \circ f$ is a plot in the diffeology of $P$. Hence, by the  fibre product diffeology (see \cref{example:fibre product diffeology}),  $\dfrac{P \times_{X_0} A \oplus a^*TX_0}{\sim}\times_{X_0}X_1$ is a difeological space:
\begin{equation}
\begin{tikzcd}
    \left(\dfrac{P \times_{X_0} A \oplus a^*TX_0}{\sim}\right) \times_{X_0} X_1  \arrow[r, "\text{pr}_2"] \arrow[d, "\tilde{s}"] & X_1 \arrow[d, "t"] \\
    \dfrac{P \times_{X_0} A \oplus a^*TX_0}{\sim} \arrow[r, "a \circ \Bar{\text{pr}}_1"] & X_0
\end{tikzcd}
\end{equation}
Verifying the smoothness of structure maps is more or less straightforward, albeit tedious for some. Here, we only give detailed proof for the target map.

To prove the target map $\tilde{t}\colon \dfrac{P \times_{X_0}A \oplus a^*TX_0}{\sim} \times_{X_0} X_1 \longrightarrow \dfrac{P \times_{X_0}A \oplus a^*TX_0}{\sim}$, \, $\tilde{t}\Bigl([p, \alpha, v],\gamma\Bigr)=[p\gamma, \lambda_{\gamma^{-1}}^0(\alpha), \lambda_{\gamma^{-1}}^1(v)]$ is smooth, consider a plot $(f,g):U \rightarrow \left(\dfrac{P \times_{X_0} A \oplus a^*TX_0}{\sim}\right) \times_{X_0} X_1  $. We aim to prove that the composite map $\tilde{t} \circ (f, g): U \rightarrow \dfrac{P \times_{X_0} A \oplus a^*TX_0}{\sim}$ is a plot in the quotient diffeology. Let $u \in U$.  Since $f: U \rightarrow \dfrac{P \times_{X_0} A \oplus a^*TX_0}{\sim}$ is a plot in the quotient diffeology, there exists a neighbourhood $V$ of $u$ and a smooth map $(f_1,f_2): V \rightarrow P \times_{X_0} A \oplus a^*TX_0$ (which is a plot in the fibre product diffeology of $P \times_{X_0} A \oplus a^*TX_0$) such that $f\restriction_V=q \circ (f_1,f_2)$.   Hence, we have the following commutative diagram.
\[
\begin{tikzcd}
    V \arrow[dr,"{(f,g) \restriction_V}"] \arrow[d, bend right=20,"{\left(f_1,f_2\right)}"'] \arrow[ddr, "f \restriction_V"'] \arrow[drr, bend left=30,"g \restriction_V"] & & \\
    P \times_{X_0} A \oplus a^*TX_0 \arrow[dr, bend right=20,"q"']  & \dfrac{P \times_{X_0} A \oplus a^*TX_0}{\sim} \times_{X_0} X_1 \arrow[d] \arrow[r, "\text{pr}_2"]& X_1 \arrow[d,"t"'] \\
    & \dfrac{P \times_{X_0} A \oplus a^*TX_0}{\sim} \arrow[r,"a"] & X_0
\end{tikzcd}
\]
Since the combined action map  $\tilde{\mu}:(P \times_{X_0} A \oplus a^*TX_0) \times_{X_0} X_1 \rightarrow P \times_{X_0} A \oplus a^*TX_0$ defined by $\tilde{\mu}((p, \alpha, v),g)=(pg, \lambda_{g^{-1}}^0(\alpha), \lambda_{g^{-1}}^1(v))$ is  smooth,   the composition  $\tilde{\mu} \circ \bigl((f_1,f_2), g\restriction_V\bigr):V \rightarrow P \times_{X_0} A \oplus a^*TX_0$ is smooth as well.  Hence, $\tilde{\mu} \circ \bigl(f_1,f_2,g\restriction_V\bigr)$ is a plot in the diffeology of $P \times_{X_0} A \oplus a^*TX_0.$  Moreover, it is easy to see that, $\tilde{t} \circ (f,g) \restriction_V=q \circ \tilde{\mu} \circ (f_1,f_2,g) \restriction_V$. We conclude that the target map is smooth.  

The projection $\Bar{\text{pr}_1}\colon \dfrac{P \times_{X_0} A \oplus a^*TX_0}{\sim} \rightarrow P$ has a natural diffeological graded vector space  structure (\cref{defn:definition of diffelogical vector space}) over $P$ with the vector space structure on a fibre given by 
\begin{equation}\label{eqn:Vector space fibre over P object AdP}
\begin{split}
&[p, \alpha, v]+[p, \alpha', v']=[p, \alpha+\alpha', v+v'],\\
&c[p, \alpha, v]=[p, c\alpha, cv], \forall c\in \mathbb{R}.    
\end{split}
\end{equation}
\subsubsection{The adjoint bundle $\AdP$ over $M$} \label{subsubsection:adjonit bundle AdP over M} The surjective submerison map $\pi\colon P \rightarrow M$  induces a smooth map between $\dfrac{P \times_{X_0} A \oplus a^*TX_0}{\sim}$ and $M$, defined by $\Bar{\pi}[p, \alpha, v]=\pi(p)$. Considering $M$ as the discrete category over $M,$ we have a natural smooth functor $\Bar{\pi}_{\AdP}\colon\AdP \rightarrow M,$
\begin{equation}\label{eqn:functor AdP to M}
 \begin{split}
 {\rm Obj}&(\AdP)\to M\\
 &[p, \alpha, v] \mapsto \pi(p),\\
 {\rm Mor}&(\AdP)\to M \\
  & \bigl([p, \alpha, v],g\bigr) \mapsto \pi(p).
 \end{split}   
\end{equation}
  The condition $\pi(pg)=\pi(p)$ ensures that $\Bar{\pi}$ is indeed a (smooth) functor, and surjectivity of $\pi$ implies surjectivity of $\Bar{\pi}_{\AdP}.$
 
Moreover, one can define a graded vector space of connected components of $\Bar{\pi}^{-1}_{\AdP}(m)$, for each $m \in M$ as follows: A connected component of an object $[p, \alpha, v]$ of subcategory $\Bar{\pi}^{-1}_{\AdP}(m)$ will be denoted as $\langle[p,\alpha,v]\rangle$. Let ${\mathfrak {Conn}}(\Bar{\pi}^{-1}_{\AdP}(m))$ be the set of connected components of $\Bar{\pi}^{-1}_{\AdP}(m).$ Suppose $\langle[p, \alpha, v]\rangle,\langle[q,\beta,v']\rangle$ are two connected components such that $\Bar{\pi}_{\AdP}([p,\alpha,v])=\Bar{\pi}_{\AdP}([q,\beta,v])$, then $\pi(p)=\pi(q).$  Hence, there exists a unique $g \in X_1$ such that $q=pg.$  Now define 
\begin{equation}\label{eqn:vector space structure on connected components of fibre of Ad(P)}
    \begin{split}
        \langle[p,\alpha,v]\rangle + \langle[q,\beta,v']\rangle ={} & \langle[pg,\lambda_{g^{-1}}^0(\alpha)+\beta,\lambda_{g^{-1}}^1(v)+v']\rangle.
    \end{split}
\end{equation}
One can prove that the above addition is well-defined.  That is, the addition does not depend on the choice of a representative from the connected components.  The zero element is given by the connected component $\langle[p,0,0]\rangle$, while the negative of the $\langle[p,\alpha,v]\rangle$ is given by $\langle[p,-\alpha,-v]\rangle.$
\begin{prop}\label{prop:AdP category over M}
  We have a diffeological groupoid $\AdP$ as defined in \cref{Eqn:Ad(P) object morphism}   and a smooth, surjective functor $\Bar{\pi}_{\AdP}\colon\AdP \rightarrow M$ given by \cref{eqn:functor AdP to M} such that  ${\mathfrak {Conn}}(\Bar{\pi}^{-1}_{\AdP}(m))$ is a graded vector space for each $m\in M,$ defined by \cref{eqn:vector space structure on connected components of fibre of Ad(P)}.
\end{prop}

\subsection{Construction of diffeological groupoid $\AtP$}
\label{Subsection:Groupoid AtP} 
\subsubsection{Making $TP \oplus a^*(TX_0) \rightarrow P$ into a quasi equivariant $\mathbb{X}$-bundle upto homotopy}
We have the natural map  $\xi={}da:TP \rightarrow a^*(TX_0)$ as the differential of the anchor map.
Let $\sigma$ be the right splitting of \cref{sequence diagram:short exact sequence with right translations} associated to the connection $\mathbb{H}$ and $\mu\colon P \times_{a, X_0, t} X_1 \rightarrow P$ is the action map. We define quasi actions of $\mathbb{X}$ on both $TP$ and $a^*TX_0$ as follows:
  \begin{equation}\label{eqn:quasi action on TP}
  \begin{split}
  &\tau^0\colon TP \times_{X_0} X_1 \rightarrow  TP, \hspace{0.5cm} (p,u) \mapsto  \bigl(pg,u.(di)_{g^{-1}}\sigma_{g^{-1}}(da)_p(u)\bigr),\\
  &\tau^1\colon a^*TX_0 \times_{X_0} X_1 \rightarrow  a^*TX_0 , \hspace{0.5cm} (p,a(p),v,g) \mapsto \bigl(pg,a(pg),\lambda_{g^{-1}}^1(v)\bigr),
  \end{split}
  \end{equation}
  where on the right-hand sides of the above equation, we adopted the convention in \cref{defn:right action of lie groupoid}.
  That is, for each $g \in X_1$, associated linear maps $\tau_g^0\colon T_p P \rightarrow T_{pg}P, \, \tau_g^1:T_{a(p)}X_0 \rightarrow T_{a(pg)}X_0$ are given by 
  \begin{align}\label{eqn:quasi action of X on TP}
  \tau_g^0(u) ={} & u.(di)_{g^{-1}}\sigma_{g^{-1}}(da)_p(u),\\
  \tau_g^1(v)={} & \lambda_{g^{-1}}^1(v)=(dt)_{g^{-1}}\sigma_{g^{-1}}(v),\label{eqn: quasi action of X on a^*TX_0}
  \end{align}
for $a(p)=t(g)$ and $u \in T_pP$.  Observe that the vector $\bigl(u,(di)_{g^{-1}}\sigma_{g^{-1}}(da)_p(u)\bigr)$ lies inside the tangent space of $P \times_{X_0} X_1$ at the point $(p,g)$, making our action  well defined.  

Again, we compute the deviation of these quasi actions from the functoriality condition in \cref{eqn:functoriality of representation}.
On one hand, we have,  
\begin{equation}\label{eqn:Action of gh on TP}
\begin{split}
\tau_{gh}^0(u)={} & u.(di)_{h^{-1}g^{-1}}\sigma_{h^{-1}g^{-1}}(da)_p(u),
\end{split}
\end{equation}
and 
on the other,
\begin{align*}
\tau_h^0 \circ \tau_g^0(u)={} & \tau_h^0\bigl(u.(di)_{g^{-1}}\sigma_{g^{-1}}(da)_p(u)\bigr) \notag
=\bigl(u.(di)_{g^{-1}}\sigma_{g^{-1}}(da)_p(u)\bigr).(di)_{h^{-1}}\sigma_{h^{-1}} \underbrace{(da)_{pg}\bigl( u.(di)_{g^{-1}}\sigma_{g^{-1}}(da)_p(u)\bigr)} \notag.
\end{align*}
Let us examine the vector in the under-braced expression in the above equation separately.  Taking the differential of the  condition (2) of \cref{defn:right action of lie groupoid} 
at $(p, g),$ we obtain  the following expression
\begin{equation}\label{eqn:reference example for relation of action with anchor}
\begin{split}
(da)_{pg}\bigl( u.(di)_{g^{-1}}\sigma_{g^{-1}}(da)_p(u)\bigr)
={} & d(s \circ {\text{pr}}_2)_{(p,g)} \bigl(u.(di)_{g^{-1}}\sigma_{g^{-1}}(da)_p(u)\bigr)  \\
={} & (ds)_g(di)_{g^{-1}}\sigma_{g^{-1}}(da)_p(u) 
\end{split}
\end{equation}
and thus,
\begin{equation}\label{eqn:action of g followed of action of h on TP}
\begin{split}
    \tau_h^0\circ \tau_g^0(u)={} & \bigl(u.(di)_{g^{-1}}\sigma_{g^{-1}}(da)_p(u)\bigr).(di)_{h^{-1}}\sigma_{h^{-1}}(ds)_g(di)_{g^{-1}}\sigma_{g^{-1}}(da)_p(u). 
    \end{split}
\end{equation}
It should be kept in mind that we have adopted the conventions of \cref{notation:differential of composition map} and   \cref{defn:right action of lie groupoid}, and $u \in T_pP, \,(di)_{g^{-1}}\sigma_{g^{-1}}(da)_p(u) \in T_gX_1, \, (di)_{h^{-1}}\sigma_{h^{-1}}(ds)_{g^{-1}}(di)_{g^{-1}}\sigma_{g^{-1}}(da)_p(u) \in T_hX_1.$
Likewise, taking the differential at the point $(p, g, h)$ for the condition (1) of \cref{defn:right action of lie groupoid}, we obtain $ d\mu_{(pg,h)} \circ \Bigl((d\mu)_{(p,g)} \times {\rm Id}_{T_hX_1}\Bigr)\Bigl(u, v, w\Bigr) ={}  d\mu_{(p,gh)} \circ \Bigl({\rm Id}_{T_pP} \times (dm)_{(g,h)}\Bigr)\Bigl(u, v, w\Bigr),$ that is,  
\begin{equation}\label{diagram:multiplicative action of torsor at tangent space level}
   \begin{split}
    (u.v).w={} & u.(v.w).
    \end{split}
\end{equation}
Inserting  \cref{diagram:multiplicative action of torsor at tangent space level} in \cref{eqn:action of g followed of action of h on TP}, we get
\begin{equation}\label{eqn:action of g followed of action of h before inverting g and h}
    \begin{split}
        \tau_h^0 \circ \tau_g^0(u) ={} & u.\biggl((di)_{g^{-1}}\sigma_{g^{-1}}(da)_p(u).(di)_{h^{-1}}\sigma_{h^{-1}}(ds)_g(di)_{g^{-1}}\sigma_{g^{-1}}(da)_p(u)\biggr).
    \end{split}
\end{equation}
Moreover, differentiating the relation between the multiplication and inverse, $(g_1 \circ g_2)^{-1}=(g_2)^{-1} \circ (g_1)^{-1}$, we obtain
\begin{equation}\label{diagram:relation of dm and di at the tangent level}
\begin{split}
(di)(v).(di)(u)={} & (di)(u.v),
\end{split}
\end{equation}
and hence \cref{eqn:action of g followed of action of h before inverting g and h} can be expressed as
\begin{equation}\label{eqn:final eqn of action of g followed of action of h}
\begin{split}
\tau_h^0 \circ \tau_g^0(u) ={} & u.\Bigl((di)_{h^{-1}g^{-1}}\bigl(\sigma_{h^{-1}}  \underbrace{(ds)_g  (di)_{g^{-1}}}_{d(s \circ i)=dt}(da)_p(u).\sigma_{g^{-1}}  (da)_p(u)\bigr)\Bigr)  \\
={} & u.\Bigl((di)_{h^{-1}g^{-1}}\bigl(\sigma_{h^{-1}}(dt)_{g^{-1}}(da)_p(u).\sigma_{g^{-1}}  (da)_p(u)\bigr)\Bigr).
\end{split}
\end{equation}
Finally, the error of this quasi action is obtained by subtracting \cref{eqn:final eqn of action of g followed of action of h} from \cref{eqn:Action of gh on TP}
\begin{equation}\label{eqn:error of induced quasi action on TP}
\begin{split}
&\tau_{gh}^0\left(u\right)-\tau_h^0\tau_g^0\left(u\right)\\
={} & 0.(di)_{h^{-1}g^{-1}}\Bigl(\sigma_{h^{-1}g^{-1}}(da)_p(u)-\bigl(\sigma_{h^{-1}}\underbrace{(dt)_{g^{-1}}\sigma_{g^{-1}}}_{\lambda=dt \circ \sigma \, by\, \cref{quasi action of X on TX_0}}(da)_p(u).\sigma_{g^{-1}}(da)_p(u)\bigr)\Bigr) \\
={} & 0.(di)_{h^{-1}g^{-1}}\Bigl(\underbrace{\sigma_{h^{-1}g^{-1}}(da)_p(u)-\bigl(\sigma_{h^{-1}}\lambda_{g^{-1}}(da)_p(u).\sigma_{g^{-1}}(da)_p(u)}_{=(dR_{h^{-1}g^{-1}})_{1_{t(h^{-1})}}K_{\sigma}^{bas}(h^{-1}g^{-1})(da)_p(u) \,  by \, \cref{eqn:measure of failure of closedness of connection}}\bigr)\Bigr) \\
={} & 0.\underbrace{(di)_{h^{-1}g^{-1}}(dR_{h^{-1}g^{-1}})_{1_{t(h^{-1})}}}K_{\sigma}^{bas}(h^{-1}, g^{-1})(da)_p(u).
\end{split}
\end{equation}
Differentiating the following relation 
\begin{equation}\label{diagram:i comp R=L comp i}
\begin{split}
    i \circ R_{h^{-1}g^{-1}}={} & L_{gh} \circ i
\end{split}
\end{equation}
the underbraced term given in the last line of \cref{eqn:error of induced quasi action on TP} turns into the following: 
\begin{equation}\label{eqn:simplified version of error of quasi action on TP}
    \begin{split}
        \tau_{gh}^0(u)-\tau_h^0\tau_g^0(u)={} & 0.\underbrace{(dL_{gh})_{1_{a(pgh)}}(di)_{1_{a(pgh)}}}_{=l_{gh} \, by \, \cref{eqn:description of left translations at tangent level}}K_{\sigma}^{bas}(h^{-1},g^{-1})(da)_p(u) \\
        ={} & 0.l_{gh}\bigl(K_{\sigma}^{bas}(h^{-1},g^{-1})(da)_p(u)\bigr).
    \end{split}
\end{equation}
Error of the other quasi action on $a^*TX_0,$ as follows:
\begin{equation}\label{eqn: error of quasi action on TX_0=dt comp basic curvature}
    \begin{split}
        \tau_{gh}^1(v)-\tau_h^1\circ\tau_g^1(v)={} & (dt)K_{\sigma}^{bas}(h^{-1},g^{-1})(v).
    \end{split}
\end{equation}
Hence, from \cref{eqn:simplified version of error of quasi action on TP}, we get that the map 
\begin{equation}\label{eqn: Upsilon map for atiyah bundle}
\begin{split}            
& \Upsilon(h,g):T_{a(p)}X_0 \rightarrow {} {} T_{pgh}P, \hspace{0.2 cm} u \mapsto 0.l_{gh}K_{\sigma}^{bas}(h^{-1},g^{-1})(u),
\end{split}
\end{equation}
provides a homotopy between linear maps $\tau_{gh}$ and $\tau_h \circ \tau_g.$
Despite the messy calculations and untidy expressions, the important piece of information  to retain in our mind is that the error term in \cref{eqn:simplified version of error of quasi action on TP}
again determined by the homotopy term  of \cref{prop:adjoint representation}.

\begin{rem}\label{Rem:error of quasi action on TP lies insie vertical}
Applying $(d\pi)_{pgh}$ on both sides of \cref{eqn:simplified version of error of quasi action on TP} 
we see that $\bigl(\tau_{gh}(u)-\tau_h\tau_g(u)\bigr)\in \ker({d\pi}_{pgh}).$  Altogether, we arrive at the following proposition
\end{rem}
\begin{prop}\label{prop: TP plus a^*TX_0 is a equivariant bundle upto homotopy}
Let $\bigl(\mbbX,\mathbb{H}\bigr)$  be a Lie groupoid equipped with Ehresmann connection and $\pi:P \rightarrow M$ be a principal $\mathbb{X}$-bundle.  Then, the map $\xi=da:TP \rightarrow a^*TX_0,$ quasi actions $\tau^0,\tau^1$(defined as in \cref{eqn:quasi action on TP} and the homotopy map $\Upsilon$(defined as in \cref{eqn: Upsilon map for atiyah bundle}) makes the graded vector bundle $TP \oplus a^*TX_0 \rightarrow P$ into an equivariant $\mathbb{X}$-bundle upto homotopy.
\end{prop}
\begin{proof}
    Most of the calculations are lengthy and tedious, and carried out in \cref{subsection: Calculations for atiyah bundle proposition} in the Appendix.  Suppose $(p,g) \in P \times_{X_0} X_1$, then by \cref{eqn: computation of da comp Upsilon} $(da)_{pg}\bigl(\tau_g^0(u))={} \tau_g^1(da)_p(u)$.  For a triple $(p,g,h)$ satisfying $a(p)=t(g)$ and $s(g)=t(h),$ we have by \cref{{eqn: computation of deviation of tau^1 for atiyah bundle},{eqn: computation of da comp Upsilon for atiyah bundle}}$, \, \tau_{gh}^0(u)-\tau_h^0\tau_g^0(u)={} \Upsilon(g,h)(da)_p(u)$ and $(da)_{pg}\Upsilon(g,h)(v)={} \tau_{gh}^1(v)-\tau_h^1\tau_g^1(v).$ 
    
    Finally, for a tuple $(p,g,h,k)$ satisfying $a(p)=t(g), \, s(g)=t(h), \, s(h)=t(k)$ the following compatibility condition
    \begin{equation}
        \tau_k^0\Upsilon(g,h)-\Upsilon(gh,k)+\Upsilon(g,hk)-\Upsilon(h,k)\tau_g^1={} 0,
    \end{equation} is verified using \cref{{eqn: first term for homotopy map for atiyah case},{eqn: second term of homotopy map relation for atiyah case},{eqn: third term of homotopy map relation for atiyah case},{eqn: fourth term of homotopy map relation for atiyah case}}.
\end{proof}

\subsubsection{Groupoid $\AtP$} \label{subsubsection:adjonit bundle AtP over M}
As with the adjoint bundle, we define the following relation on $TP \oplus a^*TX_0$ to nullify the errors.  Define $(p, u, v) \sim (p', u', v)$ if $p'=p$ and $u-v$ and $u'-v'$ can be written as the sum of errors of quasi action. i.e.  $$u-v=\sum_{i=1}^{n}\tau_{g_ih_i}(u_i)-\tau_{h_i}\tau_{g_i}(u_i), \hspace{0.5cm} u'-v'=\sum_{i=1}^n\tau_{g_ih_i}^1(v_i)-\tau_{h_i}^1\tau_{g_i}(v_i) $$ for some $u_i \in T_{p_i}P, \, v_i \in T_{a(p_i)}X_0$ such that $p=p_ig_ih_i$, for every $1\leq i \leq n$. Now, we construct the  category $\AtP$: 
Define
\begin{equation} \label{Eqn:At(P) object morphism}
\begin{split}   
&{\rm {Obj}}\bigl(\AtP\bigr)=\dfrac{TP}{\sim}=\Bigl\{[p, u, v]  \, \bigl| \, u \in T_p P, \, v \in T_{a(p)}X_0 \Bigr\}\\
&{\rm {Mor}}\bigl(\AtP\bigr)=\dfrac{TP}{\sim} \times_{X_0} X_1=\biggl\{\Bigl([p, u, v], g\Bigr) \, \Bigl| \, a(p)=t(g),u \in T_p P, \, v \in T_{a(p)}X_0  \biggr\}, 
\end{split}
\end{equation}
with the structure maps defined as we did for $\AdP$.
Verification of the welldefinedness of various structure maps is similar to that of $\AdP.$

\subsubsection{Diffeology on $\AtP$}
Diffeology on $\AtP$ is very similar to that of $\AdP.$
\subsubsection{The Atiyah bundle $\AtP$ over $M$}\label{subsubsection:Atiyah bundle over M} 
The surjective submersion $\pi:P \rightarrow M$ induces a smooth functor $\Bar{\pi}_{\AtP}:\AtP \rightarrow M$ defined as 
\begin{equation}\label{eqn:functor from AtP to ActTM}
    \begin{split}
        & {\rm{Obj}}\bigl(\AtP\big)   \rightarrow M ,\hspace{0.3cm} [p, u, v]   \mapsto \pi(p), \\
        & {\rm{Mor}}\big(\AtP\bigr)  \rightarrow M , \hspace{0.3cm} \bigl([p, u, v],g\bigr)  \mapsto \pi(p).
    \end{split}
\end{equation}
As with the case of $\AdP$, one can define a graded vector space structure on connected components $\Bar{\pi}^{-1}_{\AtP}(m),$ for each $m \in M.$ Let $\langle[p,u]\rangle$ denote the connected component of $[p,u]$ in the category  $\Bar{\pi}^{-1}_{\AtP}(m)$.  Suppose, $\langle[p, u, v]\rangle,\langle[q, u', v']\rangle$ are two connected components such that $\Bar{\pi}_{\AtP}\bigl([p, u, v]\bigr)=\Bar{\pi}_{\AtP}\bigl([q, u', v']\bigr)$, then $\pi(p)=\pi(q).$  That means, there exists a unique $g \in X_1$ such that $q=pg.$  Now, define 
\begin{equation}\label{eqn:vector space stucture on connected components of Atiyah bundle}
    \begin{split}
        \langle[p, u, v]\rangle+\langle[q, u', v']\rangle={} & \langle[pg, \tau_g^0(u)+u', \tau_g^1(v)+v']\rangle.
    \end{split}
\end{equation}
It is not difficult to verify that the above addition is well-defined. Hence, we get the following proposition.
\begin{prop}\label{prop:AtP category over M}
  We have a diffeological groupoid $\AtP$ as defined in \cref{Eqn:At(P) object morphism}   and a smooth, surjective functor $\Bar{\pi}_{\AtP}\colon\AtP \rightarrow M$  such that the connected components of each fibre $\Bar{\pi}^{-1}_{\AtP}(m)$ form a graded vector space.
  \end{prop}
\begin{proof}
Virtually identical to that of \cref{prop:AdP category over M}.
\end{proof}
\subsection{Construction of diffeological groupoid $\ActTM$}

\subsubsection{Making $\pi^*TM \xrightarrow[]{{\rm pr}_1} P$ into a quasi equivariant $\mathbb{X}$-bundle}
We do not require a connection here because there is a natural action of $\mathbb{X}$ on $\pi^*TM$ given as  
\begin{equation}
    \begin{split}
        \tau\colon \pi^*TM \times_{X_0} X_1 & \rightarrow \pi^*TM \notag \\
        (p,\pi(p),u, g) & \mapsto (pg, \pi(pg), u).
    \end{split}
\end{equation}
For each $g \in X_1$, the associated linear map on fibres $\tau_g\colon T_{\pi(p)}M \rightarrow T_{\pi(p)} M$ is the identity map for each $p \in P$ satisfying $a(p)=t(g).$
We have the associated groupoid  $\ActTM)$
\begin{equation}\label{eqn:Obj and Mor of ActTM}
\begin{split}
{\rm{Obj}}\bigl(\ActTM\bigr)={} & \pi^*TM ={}  \bigl\{(p,\pi(p),u) \mid p \in P, \, u \in T_{\pi(p)}M\bigr\},\\
{\rm{Mor}}\bigl(\ActTM\bigr)={} & \pi^*TM \times_{X_0} X_1 ={}  \bigl\{(p,\pi(p),u,g) \mid p \in P, \, u \in T_{\pi(p)}, \, a(p)=t(g) \bigr\},
\end{split}
\end{equation}
with the obvious structure maps.
\subsubsection{Diffeology on $\ActTM$}
 Both the sets, Obj$(\ActTM)$ and Mor$(\ActTM)$, are manifolds with smooth structure maps.  Hence, $\ActTM$ is a Lie groupoid.
\subsubsection{The bundle $\ActTM$ over $M$}\label{subsubsection:The bundle ActTM over M}
 Likewise with  $\AdP$ and $\AtP$, the surjective submersion $\pi:P \rightarrow M$ induces a smooth functor $\Bar{\pi}:\ActTM \rightarrow M$,
\begin{equation}\label{eqn:functor form ActTM to M}
\begin{split}
    &{\rm{Obj}}\bigl(\ActTM\bigr)  \rightarrow M, \hspace{0.3cm} (p,\pi(p),u)  \mapsto \pi(p), \\
    &{\rm{Mor}}\bigl(\ActTM\bigr)  \rightarrow M, \hspace{0.3cm} (p,\pi(p),u,g)  \mapsto \pi(p).
    \end{split}
\end{equation}
Notice that, for each $m \in M$, the connected components of the corresponding fibre $\Bar{\pi}^{-1}(m)$ are equal to $T_mM$ and thus inherit a vector space structure.
\begin{prop} \label{prop:ActM category over M}
    We have a Lie (and, thus, diffeological) groupoid  $\ActTM$ as defined in \cref{eqn:Obj and Mor of ActTM} and a smooth surjective functor $\Bar{\pi}:\ActTM \rightarrow M$ such that the connected components of each fibre $\Bar{\pi}^{-1}(m)$ is a vector space. 
\end{prop}

It is obvious that the map  $d\pi$ in  \cref{sequence diagram:short exact sequence with fundamental vector field} is a map of quasi $\mbbX$-equivariant bundle upto homotopy. The same is true for $(d\delta,{\rm Id}) $ as well. However, it demands substantial work 
 to prove the same for $d\delta,$ and the proof will be carried out in \cref{thm:equvariance of fundamental vector field map}. In conclusion:
\begin{prop}\label{prop:Sequence Q-EQUIV}
The sequence in \cref{sequence diagram:short exact sequence with fundamental vector field}  is a short exact sequence in   the category $\mathcal {Q-EQUIV}(P, \mathbb{X}).$
\end{prop}

 Next, we will investigate the behaviour of the fundamental vector field map $d\delta$ defined in \cref{eqn:fundamental vector field map} with respect to the quasi actions defined on both $P \times_{X_0} A$ and $TP.$ In particular we show that the $d\delta$ map behaves well under the quotients in \cref{Subsection:Groupoid AdP} and \cref{Subsection:Groupoid AtP}  and extends to a `linear' functor $\bar d\delta\colon \AdP \to \AtP$.
  
\subsection{The functor $\bar d\delta$}\label{subsection:Functor ddelta}

\begin{prop}\label{prop:fundamental vector field map preserve error of quasi action}
    Let $d\delta:P \times_{X_0} A \rightarrow TP$ be the fundamental vector field map defined as in \cref{eqn:description of fundamental vector field map}, $\mathbb{H}$  a connection on $\mathbb{X}$, and $\lambda,\tau$ the quasi actions defined as in \cref{quasi action of X on A}, \cref{eqn:quasi action of X on TP} respectively.  Then, \begin{equation}\label{eq:ddelta preserves error}
(d\delta)_{pgh}\bigl(\lambda_{h^{-1}g^{-1}}^0(\alpha)-\lambda_{h^{-1}}^0\lambda_{g^{-1}}^0(\alpha)\bigr)=\tau_{gh}^0\bigl((d\delta_{p})_{1_{a(p)}}(\alpha)\bigr)-\tau_h^0\tau_g^0\bigl((d\delta_p)_{1_{a(p)}}(\alpha)\bigr),    
    \end{equation} 
for every $p \in P$ and $g, h \in X_1$ satisfying $a(p)=t(g), s(g)=t(h)$ and $\alpha \in A_{s(g^{-1})}=\ker (ds)_{1_{s(g^{-1})}}=T_{1_{s(g^{-1})}}s^{-1}(s(g^{-1})).$
\end{prop}
\begin{proof}
Recall that, from \cref{eqn:description of fundamental vector field map}
\begin{equation}
\begin{split}
    (d\delta_p)_{1_{a(p)}}(\alpha) ={} &  \bigl(p.1_{a(p)},0.(di)_{1_{a(p)}}(\alpha)\bigr).
\end{split}
\end{equation}
Hence,  
\begin{equation}\label{eqn:computation of fundamental vector field on an error of A} 
\begin{split}
(d\delta_{pgh})_{1_{a(pgh)}}  \Bigl(\lambda_{h^{-1}g^{-1}}^0(\alpha)-\lambda&_{h^{-1}}^0\lambda_{g^{-1}}^0(\alpha)\Bigr)\\
={} & \Bigl((pgh)1_{a(pgh)},\underbrace{0}_{=0.0}.(di)_{1_{a(pgh)}}\bigl(\lambda_{h^{-1}g^{-1}}^0(\alpha)-\lambda_{h^{-1}}^0\lambda_{g^{-1}}^0(\alpha)\bigr)\Bigr) \\
={} & \biggl((p(gh))1_{a(pgh)},\underbrace{\Bigl(0.0\Bigr).\Bigl((di)_{1_{a(pgh)}}\bigl(\lambda_{h^{-1}g^{-1}}^0(\alpha)-\lambda_{h^{-1}}^0\lambda_{g^{-1}}^0(\alpha)\bigr)\Bigr)}\biggr)\\
&[{\rm The \, underbraced \, term \,}=0.\Bigl(0.(di)_{1_{a(pgh)}}\bigl(\lambda_{h^{-1}g^{-1}}^0(\alpha)-\lambda_{h^{-1}}^0\lambda_{g^{-1}}^0(\alpha)\bigr)\Bigr) \\
& {{\rm by \, using \, \cref{diagram:multiplicative action of torsor at tangent space level} \, for \, the \, triple} \, (p,gh,1_{a(pgh)})}] \\
={} & \biggl(p(gh1_{\underbrace{a(pgh)}_{s(gh)}}, 0.\Bigl(0.(di)_{1_{\underbrace{a(pgh)}_{s(gh)}}}\bigl(\underbrace{\lambda_{h^{-1}g^{-1}}^0(\alpha)-\lambda_{h^{-1}}^0\lambda_{g^{-1}}^0(\alpha)}_{K_{\sigma}^{bas}(h^{-1}, g^{-1}(dt)_{1_{t(g)}}(\alpha)\, \rm by \, \cref{eqn:error of quasi action on A}}\bigr)\Bigr)\biggr) \\
={} & \Bigl(p(gh1_{s(gh)}), 0.\bigl(\underbrace{0.(di)_{1_{s(gh)}}} K_{\sigma}^{bas}(h^{-1}, g^{-1})(dt)_{1_{t(g)}}(\alpha)\bigr)\Bigr) \\
& [{\rm The \, underbraced \, term \,}=(dL_{gh})_{1_{s(gh)}}(di)_{1_{s(gh)}} \, \rm by \, \cref{eqn:description of left translations at tangent level}]\\
={} & \Bigl(pgh,0.\bigl(\underbrace{(dL_{gh})_{1_{s(gh)}}(di)_{1_{s(gh)}}}_{=l_{gh} \, \rm by \, \cref{eqn:description of left translations at tangent level}}K_{\sigma}^{bas}(h^{-1},g^{-1})(dt)_{1_{t(g)}}(\alpha)\bigr)\Bigr)\\
={} & \bigl(pgh,0.l_{gh}K_{\sigma}^{bas}(h^{-1},g^{-1}))(dt)_{1_{t(g)}}(\alpha)\bigr).
\end{split}
\end{equation}
On the other hand, substituting $(d\delta_p)_{1_{a(p)}}(\alpha)$ in \cref{eqn:simplified version of error of quasi action on TP}, we get
\begin{equation}\label{eqn:preservaton of error before computing a comp delta}
\begin{split}
\tau_{gh}^0\Bigl((d\delta_p)_{1_{a(p)}}&(\alpha)\Bigr)-\tau_h^0\tau_g^0\Bigl((d\delta_p)_{1_{a(p)}}(\alpha)\Bigr)  \\
   ={} & \bigl(pgh, 0.l_{gh}K_{\sigma}^{bas}(h^{-1},g^{-1})\underbrace{(da)_p(d\delta_p)_{1_{a(p)}}(\alpha)}\bigr).
\end{split}
\end{equation}
The vector in the under-braced  term can be computed as follows:
\begin{equation}\label{eqn:computation of a comp delta}
\begin{split}
(da)_p(d\delta_p)_{1_{a(p)}}(\alpha) = {} & \underbrace{(da)_p\bigl(p1_{a(p)}, 0.(di)_{1_{a(p)}}(\alpha))}_{}\\
&[ {\rm The\,\, underbraced\,\, term}=(ds)_{1_{a(p)}}(di)_{1_{a(p)}}(\alpha), {\rm \, by \, taking \,\, differential}\\ 
& {\rm of \, condition (2)\, in\, \cref{defn:right action of lie groupoid}\, 
\, for \, the \, tuple} \, \bigl(0,(di)_{1_{a(p)}}(\alpha)\bigr)]\\
={} & \underbrace{(ds)_{1_{a(p)}} (di)_{1_{a(p)}}}_{ds \circ di=dt}(\alpha) \\
={} & (dt)_{1_{a(p)}}(\alpha).
\end{split}
\end{equation}
So, inserting \cref{eqn:computation of a comp delta} in \cref{eqn:preservaton of error before computing a comp delta} we get that,
\begin{equation}\label{eqn:computation of tau on error coming from A}
\begin{split}
\tau_{gh}^0\Bigl(p,(d\delta_p)_{1_{a(p)}}(\alpha)\Bigr) & -\tau_h^0\tau_g^0\Bigl(p,(d\delta_p)_{1_{a(p)}}(\alpha)\Bigr) \\
={} &  \Bigl(pgh,0.l_{gh}K_{\sigma}^{bas}(h^{-1},g^{-1})(dt)_{1_{a(p)}}(\alpha)\Bigr).
\end{split}
\end{equation}
Comparing \cref{eqn:computation of fundamental vector field on an error of A} and \cref{eqn:computation of tau on error coming from A}, we draw the necessary conclusion.
\end{proof}

The significance of the above result lies in the following corollary.
\begin{cor}\label{Cor:ddelta for objects}
We have a  smooth map  ${\bar d}\delta\colon {\rm Obj}(\AdP)\to {\rm Obj}(\AtP)$ given by $[p, \alpha, v]\mapsto [p,(d\delta_p)_{1_{a(p)}}(\alpha), v].$
\end{cor}
\begin{proof}
 Well-definedness of the map directly follows from \cref{eq:ddelta preserves error} and the equivalence relations defined respectively in \cref{Subsection:Groupoid AdP} and \cref{Subsection:Groupoid AtP}.   The smoothness of $\bar d\delta$ between diffeological spaces is a straightforward verification. 
\end{proof}

In fact, the map $\bar d\delta$ extends to a functor between the categories ${\bar d}\delta\colon \AdP\to \AtP.$ We prove this in the following proposition. 
Also, recall that $\psi$ is the right splitting of \cref{sequence diagram:short exact sequence with left translations} defined as 
\begin{equation}\label{eqn:description of left splitting of Ehresmann connection}
 \psi(g,u)=\bigl(1_{t(g)},(dR_{g^{-1}})_g({\rm{pr}}_2)_g(u)\bigr),
\end{equation}
where $({\rm pr}_2)_g(u)$ denotes the ${\rm ker} \, ds$ component of $u$ with respect to the splitting $TX_1=\mathbb{H} \oplus {\rm ker} \, ds.$
\begin{prop}[Equivariance of fundamental vector field]\label{thm:equvariance of fundamental vector field map}
 Let $\mathbb{H}$ be an Ehresmann connection on $\mbbX$ and $\lambda, \tau$  the quasi actions of $\mathbb{X}$ respectively on $A$ and $TP$ as in \cref{quasi action of X on A}, \cref{eqn:quasi action of X on TP}.  Then, the fundamental vector field  in \cref{eqn:fundamental vector field map} is an $\mathbb{X}$-equivariant map:
 \begin{equation}\label{eqn:X-equivariancy of ddelta}
 (d\delta_{pg})_{1_{a(pg)}}\bigl(\lambda_{g^{-1}}^0(\alpha)\bigr)=\tau_g^0\bigl((d\delta_p)_{1_{a(p)}}(\alpha)\bigr),
 \end{equation}
 for every $p \in P$ and $g \in X_1$  satisfying $a(p)=t(g)$ and $\alpha \in A_{t(g)}.$
\end{prop}
\begin{proof}
We will evaluate the LHS and RHS separately and prove that their difference is zero.
\begin{equation}  \label{eqn:LHS before computing dL comp di comp dR}
\begin{split}
(d\delta_{pg})_{1_{a(pg)}}\bigl(\lambda_{g^{-1}}^0(\alpha)\bigr) = {} & \Bigl((pg)\underbrace{1_{a(pg)}}_{1_{s(g)}},\underbrace{0}_{=0.0}.(di)_{1_{a(pg)}}\bigl(\lambda_{g^{-1}}^0(\alpha)\bigr)\Bigr)\\
={} & \biggl(\bigl(p(g)\bigr)1_{s(g)},\underbrace{\Bigl(0.0\Bigr).\Bigl((di)_{1_{s(g)}}\bigl(\lambda_{g^{-1}}(\alpha)\bigr)\Bigr)}\biggr) \\
& [{\rm The \, undebraced\,\, term}=0.\Big(0.(di)_{1_{s(g)}}\bigl(\lambda_{g^{-1}}^0(\alpha)\bigr)\Bigr), {\rm \, by \, \cref {diagram:multiplicative action of torsor at tangent space level}}\\ 
& {\rm \, for \, the \, tuple} \, (p,g,1_{s(g)})]\\
={} & \biggl(p(g1_{s(g)}),0.\Bigl(\underbrace{0.(di)_{1_{s(g)}}}\bigl(\lambda_{g^{-1}}^0(\alpha)\bigr)\Bigr)\biggr) \\
& [{\rm The \, underbraced \, term} =(dL_g)_{1_{s(g)}}(di)_{1_{s(g)}} \, {\rm by \, \cref{eqn:description of left translations at tangent level}}\\
={} & \biggl(pg,0.\Bigl((dL_g)_{1_{s(g)}}(di)_{1_{s(g)}}\bigl(\underbrace{\lambda_{g^{-1}}}^0(\alpha)\bigr)\Bigr)\biggr) \\
& [{\rm The \, underbraced \, term \, }=-\psi_{g^{-1}}l_{g^{-1}} \, {\rm by \, \cref{quasi action of X on A}}] \\
   ={} & \biggl(pg,0.\Bigl((dL_g)_{1_{s(g)}}(di)_{1_{s(g)}}-\bigl(\underbrace{\psi_{g^{-1}}}l_{g^{-1}}(\alpha)\bigr)\Bigr)\biggr) \\
& [{\rm The \, underbraced \, term \, }=(dR_g)_{g^{-1}}({\rm{pr}}_2)_{g^{-1}} \, {\rm{by \, \cref{eqn:description of left splitting of Ehresmann connection}}}] \\
  ={} & \biggl(pg,0.\Bigl(\underbrace{(dL_g)_{1_{s(g)}}(di)_{1_{s(g)}}\bigl(-(dR_g)_{g^{-1}}}({\rm{pr}}_2)_{g^{-1}}l_{g^{-1}}(\alpha)\bigr)\Bigr)\biggr). 
\end{split}
\end{equation}
The underbraced expression in the last line can be further simplified by differentiating the identity 
$L_g \circ  i \circ R_{g^{-1}} =i.$
We re-express  \cref{eqn:LHS before computing dL comp di comp dR}  as
\begin{equation}\label{eqn:LHS after computing dL comp di comp dR}
    \begin{split}(d\delta_{pg})_{1_{a(pg)}}\bigl(\lambda_{g^{-1}}^0(\alpha)\bigr) = {} & \Bigl(pg, 0.\bigl(-(di)_{g^{-1}}({\rm{pr}}_2)_{g^{-1}}l_{g^{-1}}(\alpha)\bigr)\Bigr).
    \end{split}
\end{equation}
Using  \cref{eqn:quasi action of X on TP}, one derives for the right-hand side of the above equation
\begin{equation}\label{eqn:RHS of equivariance theorem}
\begin{split}
    \tau_g^0\bigl((d\delta_p)_{1_{a(p)}}(\alpha)\bigr) ={} & \Bigl(pg,(d\delta_p)_{1_{a(p)}}(\alpha).(di)_{g^{-1}}\sigma_{g^{-1}}\underbrace{(da)_p(d\delta_p)_{1_{a(p)}}}_{=(dt)_{1_{a(p)}} \, \rm by \, \cref{eqn:computation of a comp delta}}(\alpha)\Bigr) \\
={} & \bigl(pg,\underbrace{(d\delta_p)_{1_{a(p)}}(\alpha)}.(di)_{g^{-1}}\sigma_{g^{-1}}(dt)_{1_{a(p)}}(\alpha)\bigr)\\
&[{\rm The \, underbraced \, term \,}=0.(di)_{1_{a(p)}}(\alpha) \, {\rm by \, \cref{eqn:description of fundamental vector field map}}] \\
={} & \Bigl((p1_{a(p)})g,\bigl(\underbrace{0.(di)_{1_{a(p)}}(\alpha)\bigr).\bigl((di)_{g^{-1}}\sigma_{g^{-1}}(dt)_{1_{a(p)}}(\alpha)\bigr)}\Bigr) \\
& [{\rm The \, underbraced \, term \,}=0.\bigl((di)_{1_{a(p)}}(\alpha).(di)_{g^{-1}}\sigma_{g^{-1}}(dt)_{1_{a(p)}}(\alpha)\bigr) \\
& {{\rm by \, using \, \cref{diagram:multiplicative action of torsor at tangent space level} \, for \, the \, tuple} \, (p,1_{a(p)},g)}] \\
={} & \Bigl(p(1_{a(p)}g),0.\bigl(\underbrace{(di)_{1_{a(p)}}(\alpha).(di)_{g^{-1}}\sigma_{g^{-1}}(dt)_{1_{a(p)}}(\alpha)}\bigr)\Bigr) \\
& [{\rm The \, underbrace \, term \,}=(di)_{g^{-1}}\bigl(\sigma_{g^{-1}}(dt)_{1_{a(p)}}(\alpha).\alpha\bigr) \\
& {{\rm by \, using \, \cref{diagram:relation of dm and di at the tangent level} \, for \, the \, tuple} \, (g^{-1},1_{a(p)})}] \\
={} & \biggl(pg,0.\Bigl((di)_{g^{-1}}\bigl(\sigma_{g^{-1}}(dt)_{1_{a(p)}}(\alpha).\alpha\bigr)\Bigr)\biggr). 
\end{split}
\end{equation}
Now, LHS (\cref{eqn:LHS after computing dL comp di comp dR})-RHS (\cref{eqn:RHS of equivariance theorem}) gives,
\begin{equation}\label{eqn:LHS-RHS}
\begin{split}
    (d\delta_{pg})&_{1_{a(pg)}}\bigl(\lambda_{g^{-1}}^0(\alpha)\bigr)-\tau_g^0\bigl((d\delta_p)_{1_{a(p)}}(\alpha)\bigr) \\
    = {} & \Biggl(pg,0.\biggl(-(di)_{g^{-1}}\Bigl(\underbrace{({\rm{pr}}_2)_{g^{-1}}l_{g^{-1}}(\alpha)+\bigl(\sigma_{g^{-1}}(dt)_{1_{a(p)}}(\alpha).\alpha\bigr)}\Bigr)\biggr)\Biggr).
\end{split}
\end{equation}
The rest of the proof will be dedicated to showing that the under-braced term on the right-hand side of \cref{eqn:LHS-RHS} vanishes. 

Recall from \cref{eqn:relation between three definition of connections on lie groupoid}, we have that $r \circ \psi +\sigma \circ ds={\rm Id}$.   That is, for each $(g,x) \in TX_1$, one has $r_{g} \circ \psi_g(x)+\sigma_g \circ (ds)_g(x)=x.$  In particular, we obtain the following
\begin{equation}
    \begin{split}
        l_{g^{-1}}(\alpha) = {} & r_{g^{-1}} \circ \underbrace{\psi_{g^{-1}}}\bigl(l_{g^{-1}}(\alpha)\bigr)+\sigma_{g^{-1}} \circ (ds)_{g^{-1}}\bigl(l_{g^{-1}}(\alpha)\bigr) \\
        &[{\rm The \, underbraced \, term \,}=(dR_{g})_{g^{-1}}({\rm{pr}}_2)_{g^{-1}} \, {\rm by \, \cref{eqn:description of left splitting of Ehresmann connection}}] \\
        ={} & \underbrace{r_{g^{-1}}} \circ (dR_g)_{g^{-1}}({\rm{pr}}_2)_{g^{-1}}\bigl(l_{g^{-1}}(\alpha)\bigr)+\sigma_{g^{-1}} \circ (ds)_{g^{-1}}\bigl(l_{g^{-1}}(\alpha)\bigr) \\
        & [{\rm The \, underbraced \, term \,}=(dR_{g^{-1}})_{1_{t(g^{-1})}} \, {\rm by \, \cref{equation:small r map}}] \\
        ={} & \underbrace{(dR_{g^{-1}})_{1_{t(g^{-1})}} \circ (dR_g)_{g^{-1}}}_{\rm Id \, by \, \cref{eqn:Inverse of dR_g}}\bigl(({\rm{pr}}_2)_{g^{-1}}\bigr)\bigl(l_{g^{-1}}(\alpha)\bigr)+\sigma_{g^{-1}} \circ (ds)_{g^{-1}}\bigl(\underbrace{l_{g^{-1}}(\alpha)}\bigr) \\
        & [{\rm The \, underbraced \, term \,}=(dL_{g^{-1}})_{1_{s(g^{-1})}}(di)_{1_{s(g^{-1})}}(\alpha) \, {\rm by \, \cref{eqn:description of left translations at tangent level}}] \\
        ={} & ({\rm{pr}}_2)_{g^{-1}}\bigl(l_{g^{-1}}(\alpha)\bigr)+\sigma_{g^{-1}} \circ \underbrace{(ds)_{g^{-1}}(dL_{g^{-1}})_{1_{s(g^{-1})}}(di)_{1_{s(g^{-1})}}(\alpha)}_{=dt \, {\rm by \, differentials \, of \, } \, s \circ L_g \circ i=t} \\
        ={} & ({\rm{pr}}_2)_{g^{-1}}\bigl(l_{g^{-1}}(\alpha)\bigr) + \sigma_{g^{-1}}(dt)_{1_{s(g^{-1})}}(\alpha).
    \end{split}
\end{equation}
Hence, the under-braced term in \cref{eqn:LHS-RHS} can be expressed as:
\begin{equation}
    \begin{split}
        ({\rm{pr}}_2)_{g^{-1}}l_{g^{-1}}(\alpha)&+\bigl(\sigma_{g^{-1}}(dt)_{1_{a(p)}}(\alpha).\alpha\bigr) \\
        ={} & \underbrace{l_{g^{-1}}(\alpha)}-\sigma_{g^{-1}}(dt)_{1_{s(g^{-1})}}(\alpha)+\bigl(\sigma_{g^{-1}}(dt)_{1_{a(p)}}(\alpha).\alpha\bigr) \\
        & [{\rm The \, underbraced \, term \,}=\bigl(0.(di)_{1_{s(g^{-1})}}(\alpha)\bigr) \, {\rm by \, \cref{eqn:description of left translations at tangent level}}] \\
        ={} & \bigl(0.(di)_{1_{s(g^{-1})}}(\alpha)\bigr)-\underbrace{\sigma_{g^{-1}}(dt)_{1_{s(g^{-1})}}(\alpha)}  +\bigl(\sigma_{g^{-1}}(dt)_{1_{a(p)}}(\alpha).\alpha\bigr) \\
        & [{\rm The \, underbraced \, term \,}=\bigl(\sigma_{g^{-1}}(dt)_{1_{s(g^{-1})}}(\alpha).(d(1 \circ t))_{1_{s(g^{-1})}}(\alpha)\bigr) \, {\rm by \, \cref{prop:measure of failure of g and unit of s(g)}}] \\
        ={} & \bigl(0.(di)_{1_{s(g^{-1})}}(\alpha)\bigr)-\bigl(\sigma_{g^{-1}}(dt)_{1_{s(g^{-1})}}(\alpha).(d1 \circ t)_{1_{s(g^{-1})}}\alpha\bigr)+\bigl(\sigma_{g^{-1}}(dt)_{1_{\underbrace{a(p)}_{s(g^{-1})}}}(\alpha).\alpha\bigr)\\
        ={} & \bigl(0.(di)_{1_{s(g^{-1})}}(\alpha)-d(1 \circ t)_{1_{s(g^{-1})}}(\alpha)+\alpha\bigr).
    \end{split}
\end{equation}
So, to prove the theorem, it suffices to prove that the vector $(d1 \circ t)_{1_{s(g^{-1})}}(\alpha)-(di)_{1_{s(g^{-1})}}(\alpha)-\alpha$ is zero which follows from \cref{lem:crucial lemma proving unitality of quasi action}. 
\end{proof}
\begin{lem}\label{lem:crucial lemma proving unitality of quasi action}
    Let $\mathbb{X}: X_1 \rightrightarrows X_0$ be a Lie groupoid.  Then, $d(1 \circ t)_{1_x}(\alpha)-(di)_{1_x}(\alpha)-\alpha=0$, for every $\alpha \in \ker \, (ds)_{1_x}$ and $x \in X_0.$
\end{lem}
\begin{proof}
    First of all $s^{-1}(x)$ is a closed submanifold of $X_1$. In \cite{moerdijk2003introduction}, author has proved that the distribution defined by $E_g={\text{ker}} \, (ds)_g \, \cap \, {\text{ker}} \, (dt)_g$ defines a foliation for the submanifold $s^{-1}(x)$.  Moreover, leaves of the above foliation are nothing but the connected components of fibres of the restricted target map $t_{\mkern 1mu \vrule height 2ex\mkern 2mus^{-1}(x)}\colon s^{-1}(x) \rightarrow X_1$. Since $1_x \in s^{-1}(x)$, there exists a submanifold $M$ $\bigl($a connected component of $t_{\mkern 1mu \vrule height 2ex\mkern 2mus^{-1}(x)}^{-1}(x)=s^{-1}(x) \cap t^{-1}(x)\bigr)$ such that $T_{1_x}M={\text{ker}} \, (ds)_{1_x} \, \cap \, {\text{ker}} (dt)_{1_x}$.  Since, $M$ is a connected component of $s^{-1}(x) \, \cap \, t^{-1}(x)$, we also have that $T_{1_x}M=T_{1_x}\bigl(s^{-1}(x) \, \cap \, t^{-1}(x)\bigr)$.  In the same theorem in \cite{moerdijk2003introduction}, author has shown that the loop space $s^{-1}(x) \, \cap \, t^{-1}(x)$ concentrated at $x$ is indeed a Lie group.   The result follows after verification that the vector $d(1 \circ t)_{1_x}(\alpha)-(di)_{1_x}(\alpha)-\alpha$ lies inside the tangent space of $s^{-1}(x) \, \cap \, t^{-1}(x),$ and invariant under the map $(di)_{1_x}$. 
\end{proof}

\begin{prop}\label{prop: morphism between P times A and TP bundles upto homotopy}
    The following map $\bigl(d\delta,{\rm Id}\bigr):P \times_{X_0} A \oplus a^*TX_0 \rightarrow TP \oplus a^*TX_0$ given by $(p, \alpha, v) \mapsto (p, (d\delta_p)_{1_{a(p)}}(\alpha_, v)$ is a morphism of quasi equivairant bundles upto homotopy.
\end{prop}
\begin{proof}
    Verification is fairly straightforward.
\end{proof}
We make a note of the following properties of the inverse $d\delta$
map for our immediate use. Recall that, the fundamental vector field map $d\delta\colon P \times_{X_0} A \rightarrow TP$  in \cref{eqn:fundamental vector field map} is an isomorphism between vector bundles $P \times_{X_0} A$ and ${\rm{ker}} \, d\pi.$  Then, the inverse  $(d\delta)^{-1}\colon {\rm{ker}} \, d\pi \rightarrow P \times_{X_0} A$ is 
    \begin{equation}\label{eqn:inverse of fundamental vector field map}
        \begin{split}
            (d\delta)^{-1}(p,u)={} & (p,(d\delta_p)_{1_{a(p)}})^{-1}(u)).
        \end{split}
    \end{equation}

\begin{cor}\label{cor:preservation of error by inverse of fundamental vector field map}
As per the notations in  \cref{prop:fundamental vector field map preserve error of quasi action} and \cref{thm:equvariance of fundamental vector field map}, 
\vskip 0.1 cm 
%\begin{itemize}
    %    \item [] 
   (i) for every $u \in T_pP$ and $(p,g,h)$ satisfying $a(p)=t(g),s(g)=t(h),$   
   \begin{equation}\label{eqn:preservation of error by inverse of fundamental vector field map}
         (d\delta_{pgh})^{-1}\bigl(\tau_{gh}^0(u)-\tau_h^0\tau_g^0(u)\bigr)=\lambda_{h^{-1}g^{-1}}^0(d\delta_p)_{1_{a(p)}}^{-1}(u)-\lambda_{h^{-1}}^0\lambda_{g^{-1}}^0(d\delta_p)_{1_{a(p)}}^{-1}(u).
    \end{equation}
    %\item [] 
    \vskip 0.05 cm
    (ii)  We have    
    $(d\delta)^{-1}\colon {\rm{ker}} \, d\pi \rightarrow P \times_{X_0} A$  is a morphism in $\mathcal {Q-EQUIV}(P, \mathbb{X})$;  That is, for every $u \in ({\rm{ker}} \, d\pi)_p$ and $(p,g)$ satisfying $a(p)=t(g)$, we have 
 \begin{equation}\label{eqn:equivariance of inverse of fundamental vector field map}
        \begin{split}
            (d\delta)^{-1}\bigl(\tau_g^0(u)\bigr)={} & \lambda_{g^{-1}}^0(d\delta_p)_{1_{a(p)}}^{-1}(u).
        \end{split}
    \end{equation}
%\end{itemize}
    
\end{cor}
\begin{proof}
    Immediate from \cref{prop:fundamental vector field map preserve error of quasi action} and \cref{thm:equvariance of fundamental vector field map}. 
\end{proof}

From \cref{thm:equvariance of fundamental vector field map} the smooth map $\Bar{d\delta}:{\rm Obj\Bigl(\AdP\Bigr)} \rightarrow {\rm Obj\Bigl(\AtP\Bigr)},  \Big([p, \alpha, v]\Bigr) \mapsto \Bigl([p,  (d\delta_p)_{1_{a(p)}}\alpha, v]\Bigr)$ extends to a smooth functor $\Bar{d\delta}:\AdP \rightarrow \AtP$ by 
\begin{equation}\label{eqn:functor form AdP to AtP}
    \begin{split}
        {\rm{Mor}}\bigl(\AdP\bigr) & \rightarrow {\rm{Mor}}\bigl(\AtP\bigr) \\
        \bigl([p, \alpha, v],g\bigr) & \rightarrow \bigl([p, (d\delta_p)_{1_{a(p)}}, v], g\bigr).
    \end{split}
\end{equation}
Recall we have seen in \cref{subsubsection:adjonit bundle AdP over M} and \cref{subsubsection:adjonit bundle AtP over M} that there are smooth surjective functors $\AdP \xrightarrow{\Bar{\pi}_{\rm Ad}} M$ and $\AtP \xrightarrow{\Bar{\pi}_{\rm At}} M$, and connected components of the fibres of both of these maps form vector spaces. Let $\langle[p, \alpha, v]\rangle$ be a   connected component of ${\Bar{\pi}}_{\rm Ad}^{-1}(m).$ Define
\begin{equation}\label{eqn:ddelta on connected component}
\begin{split}
{\bar d}\delta\colon & {\mathfrak {Conn}}(\Bar{\pi}_{\rm Ad}^{-1}(m))\to {\mathfrak {Conn}}(\Bar{\pi}_{\rm At}^{-1}(m))\\
& \, \, \qquad \langle[p, \alpha, v]\rangle \mapsto \langle[p, (d\delta_p)_{1_{a(p)}}(\alpha), v]\rangle.
\end{split}
\end{equation}
Clearly, the map is well defined by the functionality of ${\bar d}\delta$ and \cref{thm:equvariance of fundamental vector field map}. Curiously,  $\Bar{d\delta}$ in \cref{eqn:ddelta on connected component} is in fact linear.

To see this, consider two connected components  $\langle[p, \alpha, v]\rangle,\langle [q, \beta, v']\rangle$ of ${\Bar{\pi}}_{\rm Ad}^{-1}(m)$, then from \cref{eqn:vector space structure on connected components of fibre of Ad(P)}, we have that $q=pg$ for some unique $g \in X_1$, and the addition is defined by $$\langle[p, \alpha, v]\rangle+\langle[q, \beta, v']\rangle=\langle[pg, \lambda_{g^{-1}}^0(\alpha)+\beta, \lambda_{g^{-1}}^1(v)+v']\rangle.$$  Hence,
\begin{equation}\label{eqn:linearity of functor from AdP to AtP}
    \begin{split}
        \Bar{d\delta}\bigl(\langle[p, \alpha, v]\rangle+\langle[q, \beta, v']\rangle\bigr)={} & \Bar{d\delta}\bigl(\langle[pg, \lambda_{g^{-1}}^0(\alpha)+\beta, \lambda_{g^{-1}}^1(v)+v']\rangle\bigr)\\
        ={} & \langle[pg,\underbrace{(d\delta_{pg})_{1_{a(pg)}}\lambda_{g^{-1}}^0(\alpha)}+(d\delta_{pg})_{1_{a(pg)}}(\beta),\lambda_{g^{-1}}^1(v)+v']\rangle \\
        & [{\rm The \, underbraced \, term \,}=\tau_g(d\delta_p)_{1_{a(p)}}(\alpha) \, {\rm by \, \cref{thm:equvariance of fundamental vector field map}}] \\
        ={} & \langle[pg,\tau_g(d\delta_p)_{1_{a(p)}}(\alpha)]+(d\delta_{pg})_{1_{a(pg)}}(\beta), \lambda_{g^{-1}}^1(v)+v' ]\rangle \\
        ={} & \langle[p,(d\delta_p)_{1_{a(p)}}(\alpha), v]\rangle+\langle[q,(d\delta_{pg})_{1_{a(pg)}}(\beta), v']\rangle\\
        ={} & {\bar d}\delta \langle[p,\alpha, v]\rangle+{\bar d}\delta\langle[q, \beta, v']\rangle.
    \end{split}
\end{equation}
\begin{prop}\label{prop:functor bardelta}
    The functor $\Bar{d\delta}:\AdP \rightarrow \AtP$ is a smooth, full, faithful functor with an injective object level map.  The induced map \cref{eqn:ddelta on connected component}   on the connected components of each fibre: ${\mathfrak {Conn}}(\Bar{\pi}_{\rm Ad}^{-1}(m))\to {\mathfrak {Conn}}(\Bar{\pi}_{\rm At}^{-1}(m))$  is linear.    
\end{prop}
\begin{proof}
    We can see the map in \cref{eqn:functor form AdP to AtP} is injective by supposing  ${\bar d}\delta ([p, \alpha, v])={\bar d}\delta ([p, \beta], v');$ i.e.    
     $$[p,(d\delta_p)_{1_{a(p)}}(\alpha), v]=[p, (d\delta_p)_{1_{a(p)}}(\beta), v'].$$ Then,
     \begin{equation}\label{eqn:difference of ddelta of alpha and ddelta of beta}
     \begin{split}
     (d\delta_p)_{1_{a(p)}}(\alpha)-(d\delta_p)_{1_{a(p)}}(\beta)={} & \sum_{i=1}^{n}\tau_{g_ih_i}^0(u_i)-\tau_{h_i}^0\tau_{g_i}^0(u_i), \\
     v-v'={} & \sum_{i=1}^n\tau_{g_ih_i}^1(v_i)-\tau_{h_i}^1\tau_{g_i}^1(v_i)
     \end{split}
     \end{equation}
     for some $u_i \in T_{p_i}P$ and $v_i \in T_{a(p_i)}X_0$ such that $p=p_ig_ih_i$, for every $1\leq i \leq n$. Hence,
    \begin{equation}
        \begin{split}
            \alpha-\beta={}& (d\delta_p)_{1_{a(p)}}^{-1}\bigl(\underbrace{(d\delta_p)_{1_{a(p)}}(\alpha)-(d\delta_p)_{1_{a(p)}}(\beta)}\bigr) \\
            & [{\rm The \, underbraced \, term \,}=\sum_{i=1}^{n}\tau_{g_ih_i}(u_i)-\tau_{h_i}\tau_{g_i}(u_i) \, {\rm by \, \cref{eqn:difference of ddelta of alpha and ddelta of beta}}] \\
            ={} & \sum_{i=1}^{n}\underbrace{(d\delta_{p_ig_ih_i})_{1_{a(p_ig_ih_i)}}^{-1}\bigl(\tau_{g_ih_i}(u_i)-\tau_{h_i}\tau_{g_i}(u_i)\bigr)}\\
            & [{\rm The \, underbraced \, term \,}=\lambda_{h_i^{-1}g_i^{-1}}(d\delta_{p_i})_{1_{a(p_i)}}^{-1}(u_i)-\lambda_{h_i^{-1}}\lambda_{g_i^{-1}}(d\delta_{p_i})_{1_{a(p_i)}}(u_i) \, {\rm by \, \cref{cor:preservation of error by inverse of fundamental vector field map}}] \\
            ={} & \sum_{i=1}^{n}\lambda_{h_i^{-1}g_i^{-1}}^0(d\delta_p)_{1_{a(p)}}^{-1}(u_i)-\lambda_{h_i}^0\lambda_{g_i}^0(d\delta_p)_{1_{a(p)}}^{-1}(u_i). 
        \end{split}
    \end{equation}
    Thus, we conclude that $[p, \alpha, v]=[p, \beta, v'].$ 
 Now if for the objects  $[p, \alpha, v], [q, \beta, v']$ in $\AdP$, there is  an  arrow  in $\AtP$ between $[p,(d\delta_p)_{1_{a(p)}}(\alpha), v]$ and $[q, (d\delta_q)_{1_{a(q)}}(\beta), v'],$  then, by definition of ${\rm{Mor}}\bigl(\AtP\bigr),$ there exists a unique $g \in X_1$ such that $q=pg$ and $[q,(d\delta_q)_{1_{a(q)}}(\beta), v']=[pg, \tau_g^0(d\delta_p)_{1_{a(p)}}(\alpha), \tau_{g}^1(v)].$ 
    So,
    \begin{equation}
        \begin{split}
            \Bar{d\delta}\bigl([pg, \lambda_{g^{-1}}^0(\alpha), \lambda_{g^{-1}}^1(v)]\bigr)={} & [pg, \underbrace{(d\delta_{pg})_{1_{a(pg)}}\bigl(\lambda_{g^{-1}}^0(\alpha)}, \lambda_{g^{-1}}^1(v)\bigr)] \\
            & [{\rm The \, underbraced \, term \,}=\tau_g^0(d\delta_p)_{1_{a(p)}}(\alpha) \, {\rm by \, \cref{thm:equvariance of fundamental vector field map}}] \\
            ={} & [pg,\tau_g^0(d\delta_p)_{1_{a(p)}}(\alpha), \lambda_{g^{-1}}^1(v)]=[q, (d\delta_q)_{1_{a(q)}}(\beta), v']= \Bar{d\delta}\bigl([q, \beta, v']\bigr).
            \end{split}
            \end{equation}
Hence, by Injectivity of $\Bar{d\delta}$ on object level, we get that $[q, \beta, v']=[pg, \lambda_{g^{-1}}(\alpha), \lambda_{g^{-1}}^1(v)]$. Eventually, $\bigl([p, \alpha, v],g\bigr)$ is the unique arrow from $[p, \alpha, v]$ and $[pg, \lambda_{g^{-1}}(\alpha), \lambda_{g^{-1}}^1(v)]=[q, \beta, v']$ whose image is the unique arrow from $[p,(d\delta_p)_{1_{a(p)}}(\alpha), v]$ to $[q, (d\delta_q)_{1_{a(q)}}(\beta), v']$ in $\AtP$. Thus, the functor $\Bar{d\delta}$ is  full.
Now we note that the cardinality of each of the hom sets in either of the categories $\AdP$ and $\AtP$ is at most $1,$ which implies the functor is faithful.
\end{proof}
One can make a similar consideration for the differential of the  surjective submersion $\pi\colon P \rightarrow M$ to induce a  smooth functor $\Bar{d\pi}:\AtP \rightarrow \ActTM$ over $M$ as follows:
\begin{equation}\label{eqn:functor form AtP to ActTM}
    \begin{split}
       & {\rm{Obj}}\bigl(\AtP\bigr)  \rightarrow {\rm{Obj}}\bigl(\ActTM\bigr) , \hspace{0.3cm} [p, u, v]  \rightarrow (p,\pi(p),(d\pi)_p(u)), \\
       & {\rm{Mor}}\bigl(\AtP\bigr)  \rightarrow {,\rm{Mor}}\bigl(\ActTM\bigr), \hspace{0.3cm} \bigl([p, u, v],g\bigr)  \mapsto (p,\pi(p),(d\pi)_p(u)), g).
    \end{split}
\end{equation}
Verification that \cref{eqn:functor form AtP to ActTM} is a well-defined functor is more or less routine.  It is also apparent from the discussion in  \cref{subsubsection:The bundle ActTM over M}, that the connected components of each fibre $\Bar{\pi}^{-1}_{\ActTM}(m)$ of $\ActTM$ can be naturally identified with  $T_mM.$  With the vector space structure  defined in \cref{eqn:vector space stucture on connected components of Atiyah bundle} for ${\mathfrak{Conn}}\bigl(\Bar{\pi}^{-1}_{\AtP}(m)\bigr),$ define
\begin{equation}\label{eqn:dpi restricted to a connected component}
    \begin{split}
        \Bar{d\pi}:{\mathfrak{Conn}}\bigl(\Bar{\pi}^{-1}_{\AtP}(m)\bigr) & \rightarrow {\mathfrak{Conn}}\bigl(\Bar{\pi}^{-1}_{\ActTM}(m)\bigr) \\
        \langle[p, u, v]\rangle & \rightarrow  {d\pi}_p(u). 
    \end{split}
\end{equation}
We skip the verification of the linearity of the map $\Bar{d\pi}$ in \cref{eqn:dpi restricted to a connected component}.  
%For, suppose $\langle[p,u]\rangle,\langle[q,v]\rangle \in {\mathfrak{Conn}}\bigl(\Bar{\pi}^{-1}_{\AtP}(m)\bigr)$, then there exists a unique $g \in X_1$ such that $q=pg$.  Hence, 
%\begin{equation}
%    \begin{split}
        %\Bar{d\pi}\bigl(\langle[p,u]\rangle+\langle[q,v]\rangle\bigr)={} & \Bar{d\pi}\bigl(\langle(pg,\tau_g(u)+v)\rangle\bigr) \\
    %    ={} & \langle(pg,\pi(pg),(d\pi)_{pg}(\tau_g(u))+(d\pi)_{pg}(v))\rangle \\
     %   ={} & \langle(pg,\pi(p),(d\pi)_p(u)+(d\pi_{pg})(v))\rangle \\
      %  ={} & \langle(p,\pi(p),%(d\pi)_p(u)\rangle+\langle(pg,\pi(pg),(d\pi)_{pg}(v)\rangle 
   % \end{split}
%\end{equation}

\begin{prop} \label{prop:functor barpi} The functor $\Bar{d\pi}:\AtP \rightarrow \ActTM$ is a smooth, essentially surjective and faithful functor.  The induced map \cref{eqn:dpi restricted to a connected component}   on the connected components of each fibre: ${\mathfrak {Conn}}(\Bar{\pi}_{\rm At}^{-1}(m))\to {\mathfrak {Conn}}(\Bar{\pi}_{\ActTM}^{-1}(m))$  is linear.    
\end{prop}
\begin{proof}
A direct consequence of surjectivity of $d\pi.$ Faithfulness of $\Bar{d\pi}$ is the consequence of the cardinality of hom sets. 
\end{proof}
A direct verification confirms  the induced sequence of linear maps 
$$
 \begin{tikzcd}
        0 \arrow[r] & \mathfrak{Conn}\bigl(\Bar{\pi}^{-1}_{\AdP}(m)\bigr) \arrow[r,  " 
 \Bar{d\delta}"] & \mathfrak{Conn}\bigl(\Bar{\pi}^{-1}_{\AtP}(m)\bigr) \arrow[r, " 
 \Bar{d\pi}"] & \mathfrak{Conn}\bigl(\Bar{\pi}^{-1}_{\ActTM}(m)\bigr) \arrow[r] & 0
    \end{tikzcd}
    $$
is short exact. Combining all the results of this section, we make the final observation of this section as a theorem.
\begin{thm}\label{Thm:Atiyah sequence for Lie groupoid bundle existence}
Let $P\to M$ be a Lie groupoid  $\mbbX$-bundle, and $\mathbb{H}$  a connection on the Lie groupoid $\mathbb{X}$. Then  we have a sequence of smooth functors with $\bar d\delta$ full, faithful, injective on the objects and $\bar d\pi$ essentially surjective and faithful
\begin{equation}\label{sequence diagram:ATIYAH sequence}
\begin{tikzcd}
    0 \arrow[r] & \AdP \arrow[r, " 
 \Bar{d\delta}"] & \AtP \arrow[r,"\Bar{d\pi}"] &\ActTM \arrow[r] & 0,
\end{tikzcd}
\end{equation}
inducing a  sequence of linear maps 
\begin{equation}\label{sequence diagram:short exact sequence of connected components}
    \begin{tikzcd}
        0 \arrow[r] & \mathfrak{Conn}\bigl(\Bar{\pi}^{-1}_{\AdP}(m)\bigr) \arrow[r, , " 
 \Bar{d\delta}"] & \mathfrak{Conn}\bigl(\Bar{\pi}^{-1}_{\AtP}(m)\bigr) \arrow[r,  " 
 \Bar{d\pi}"] & \mathfrak{Conn}\bigl(\Bar{\pi}^{-1}_{\ActTM}(m)\bigr) \arrow[r] & 0
    \end{tikzcd}
\end{equation}
for each $m\in M.$  
The sequence is exact in the sense that \eqref{sequence diagram:short exact sequence of connected components} is exact for each $m\in M.$
\end{thm}
\begin{defn}\label{Def:Atiyah sequence for Lie groupoid bundle}
The short exact sequence in \cref{sequence diagram:ATIYAH sequence} associated to a Lie groupoid $\mbbX$-bundle $\pi\colon P \rightarrow M$ and for a choice of connection $\mathbb{H}\subset TX_1$
will be called \textit{Atiyah sequence of the $\mbbX$-bundle for the connection $\mathbb{H}$}.
\end{defn}
\begin{rem}
Suppose a groupoid $\mbbX$ acts on a set $S.$ Then one can view the quotient under this groupoid action as having the action groupoid $[X_1\times S\rightrightarrows S].$
The construction of grupoids $\AdP$ and 
$\AtP$ respectively discussed in
\cref{subsubsection:adjonit bundle AdP over M} and \cref{subsubsection:adjonit bundle AtP over M} can then be interpreted as some kind of ``stacky quotient'' in the following sense.  Notice that if $[p,\alpha,v],[p',\alpha',v'] \in {\rm Obj}\bigl(\AdP\bigr)$ are connected by a morphism in $\AdP$, then by  \cref{{eqn:error of quasi action on A},{eqn:error of quasi action on TX_0},{eqn: equivalance relation on P times A plus TX_0}} there exists a unique $g \in X_1$ such that $p'=pg$ and 
    $\alpha'-\lambda_{g^{-1}}^0(\alpha)=\sum_{i=1}^nK_{\sigma}^{bas}(h_i^{-1},g_i^{-1})(dt)(\alpha_i)$ and $v'-\lambda_{g^{-1}}^1(v)=\sum_{i=1}^n(dt)K_{\sigma}^{bas}(h_i^{-1},g_i^{-1})(v_i).$ for some $\alpha_i \in A_{a(p_i)}$ and $v_i \in T_{a(p_i)}X_0.$  Same is true for the Atiyah bundle $\AtP$ also.
\end{rem}

\begin{example}
   In this example, we offer an explicit construction of the Atiyah sequence  \cref{sequence diagram:ATIYAH sequence}
   for a  unit bundle $X_1 \xrightarrow{t} X_0$  of \cref{example:unit bundle}. Let $A=1^*{\rm{ker}} \, ds \xrightarrow{\rm{pr}} X_0$ be the associated Lie algebroid of $\mathbb{X}$ (\cref{example:associated Lie algebroid}).  To begin with, we fix an Ehresmann connection $\mathbb{H}$ on $\mbbX.$  Then, it is easy to see that the quasi equivariant $\mathbb{X}$-bundle of \cref{subsubsection:Making P*A into a quasi equivariant X-bundle} reduces to the equivariant bundle $X_1 \times_{s, X_0,\rm{pr}} A \oplus TX_0=s^*A \oplus s^*TX_0$ with the respective quasi actions:
    \begin{equation}
        \begin{split}
            & \tau^0:s^*A \times_{X_0} X_1  \rightarrow s^*A, \hspace{0.3cm} \bigl(h,s(h),1_{s(h)},\alpha, g\bigr)  \mapsto  \biggl(hg,s(hg),1_{s(hg)},\lambda_{g^{-1}}^0(\alpha)\biggr), \\
            & \tau^1:s^*TX_0 \times_{X_0} X_1  \rightarrow s^*TX_0,\hspace{0.3cm} \bigl(h, s(h), v, g\bigr)  \mapsto \biggl(hg,s(hg),\lambda_{g^{-1}}^1(v)\biggr),
        \end{split}
    \end{equation}
    for $\xi=dt:s^*A \rightarrow s^*TX_0$, \, $(h, s(h), 1_{s(h)}, \alpha) \mapsto (h, s(h), (dt)_{1_{s(h)}}(\alpha)$ and $\Upsilon(g,h):T_{s(k)}X_0 \rightarrow A_{s(kgh)}$,\, $v \mapsto K_{\sigma}^{bas}(h^{-1},g^{-1})(v)$.    Likewise  for each $g \in X_1$, quasi actions on the quasi equivariant $\mathbb{X}$-bundle upto homotopy $TX_1 \oplus s^*TX_0$ is given by linear maps $\tau_g^0\colon T_hX_1 \rightarrow T_{hg}X_1$ and $\tau^1_g: T_{s(h)}X_0 \rightarrow T_{s(hg)}X_0$(meaningful  for all $h \in X_1$ satisfying $s(h)=t(g)$) 
    \begin{equation}\label{eqn:quasi action of X on TX_1 for the unit bundle}
        \begin{split}
            u\mapsto {} & u.(di)_{g^{-1}}\sigma_{g^{-1}}(ds)_h(u)\\
            v \mapsto {} & \lambda_{g^{-1}}^1(v)
        \end{split}
    \end{equation}
with $\xi=ds:TX_1 \rightarrow s^*TX_0$  and  $\Upsilon(g,h):T_{s(k)}X_0 \rightarrow T_{kgh}X_1$, \, $v \mapsto 0.l_{gh}K_{\sigma}(h^{-1},g^{-1})(v)$.   The quasi action $t^*TX_0 \times X_1  \rightarrow t^*TX_0, \, (h,t(h), v, g)  \mapsto (hg, t(hg), v)$ defines a quasi equivariant $\mathbb{X}$-bundle $\pi^*TM.$
 Further, the fundamental vector field map is given by, after applying \cref{eqn:description of fundamental vector field map} and \cref{eqn:description of left translations at tangent level},  $d\delta:s^*A \rightarrow TX_1$, $d\delta(h, s(h),1_{s(h)}, \alpha)=\bigl(h, l_h(\alpha)\bigr).$
Hence, one can see that \cref{sequence diagram:short exact sequence with fundamental vector field} is identical to \cref{sequence diagram:short exact sequence with left translations} with an added $s^*TX_0$ along with a quasi-action on each of the terms in it.  In turn  we derive the functors $\Bar{d\delta}:\AdP \rightarrow \AtP, \Bar{dt}:\AtP \rightarrow \ActTM$ respectively given as: $\Bar{d\delta}\Bigl([h, s(h), 1_{s(h)}, \alpha, v]\Bigr)=[h, l_{h}(\alpha), v]$ and $\Bar{dt}\Bigl([h, u, v]\Bigr)=(h,t(h),dt(u))$.
\end{example}

\begin{defn}[quasi-connection as splitting of Atiyah sequence]\label{defn:quasi connection}
    A \textit{quasi-connection} on a principal bundle $\pi:P \rightarrow M$ is a functor $F=(F_0, F_1)\colon \AtP \rightarrow \AdP$ over $M$  satisfying the conditions,
    \begin{enumerate}
            \item $F \circ \Bar{d\delta}= {\rm Id}_{\AtP}$, 
            \item  $F_0:{\rm{Obj}}\bigl(\AtP\bigr) \rightarrow {\rm{Obj}}\bigl(\AdP\bigr)$ is  a  map of diffeological graded vector  bundles,  
           \item $\Bar{F}:{\mathfrak{Conn}}\bigl(\pi_{\AtP}^{-1}(m)\bigr) \rightarrow {\mathfrak{Conn}}\bigl(\pi_{\AdP}^{-1}(m)\bigr)$ is a linear map for each $m \in M,$
    \end{enumerate}
    where $\Bar{F}$  is defined in a same fashion as $\bar{d}\delta$ in  \cref{eqn:ddelta on connected component}.
    
\end{defn}

\subsection{Atiyah sequence for a Cartan connection}\label{subsection:Atiyah sequence for a cartan connection}
The construction of the Atiyah sequence associated with a Lie groupoid $\mbbX$-bundle $(P, M)$ very much relies on a choice of the Ehresmann connection. As a special case here, we consider a Cartan connection; that is, an Ehresmann connection $\mathbb{H}_{\rm Cart}$ with a vanishing basic curvature; $K_{\sigma}^{bas}=0$ in  \cref{eqn:basic curvature}.

According to  \cref{eqn:error of quasi action on fibre product of P and A} and \cref{eqn:error of induced quasi action on TP}, errors in both quasi actions defined on $P \times_{X_0} A \oplus a^*TX_0$ and $TP \oplus a^*TX_0$ are zero for a Cartan connection. Thus, we have full actions of $\mbbX$ on both of them with the functoriality condition of \cref{eqn:functoriality of representation} intact. Moreover,  object spaces of the categories $\AdP$ and $\AtP$ are respectively, the smooth graded vector bundles $P\times_{X_0}A \oplus a^*TX_0$ and $TP \oplus a^*TX_0.$ In turn $\AdP$ and $\AtP$ are Lie groupoids and 
\begin{equation}\label{sequence diagram:ATIYAH sequence in case of cartan connection}
\begin{tikzcd}
    0 \arrow[r] & \AdP \arrow[r, " 
 \Bar{d\delta}"] & \AtP \arrow[r,"\Bar{d\pi}"] &\ActTM \arrow[r] & 0
\end{tikzcd}
\end{equation} 
is a sequence of Lie groupoid morphisms. 
The connected components of a fibre $\pi^{-1}(m)$ containing $(p, \alpha, v)\in P\times_{X_0}A \oplus a^*TX_0$ or $(p, u, v)\in TP \oplus a^*TX_0$ respectively for the maps $\AdP\to M$ or $\AtP\to M$ is exactly the orbit space  of $(p, \alpha, v)  \in P \times_{X_0} A$ or $(p, u, v)\in TP$. 

  As an illustration, let us examine the case of a principal Lie group $G$-bundle $P\to M$ viewed as a Lie groupoid $[G\rightrightarrows *]$-bundle as in \cref{example:ordinary bundle}. $G\to *$ has the unique (trivial) choice for a connection $\mathbb{H}=0.$ Also, the connected components of fibers of both $\AdP$ and $\AtP$ are exactly the fibers of the adjoint bundle ${\rm{Ad}}(P)$ and Atiyah bundle ${\rm{At}}(P)$ that appeared in \cref{sequence diagram:Atiyah sequence in classical setup}. Thus, the short exact sequences in \cref{sequence diagram:short exact sequence of connected components} for all $m\in M$ are glued together to produce the short exact sequence of vector bundles over $M$ in \cref{sequence diagram:Atiyah sequence in classical setup}.

\begin{example}\label{example:Atiyah sequence for ordinary bundle with action groupoid}
Consider the principal $[P \times G \rightrightarrows P]$-bundle $\pi\colon P \rightarrow M$ of \cref{example:ordinary bundle with action groupoid}. The various structure maps of the tangent groupoid $[TP \times TG \rightrightarrows TP]$ are given by
\begin{equation}\label{eqn:differentials of structural maps of action groupoid}
\begin{split}
& ds(p, u, g, \alpha)=(p, u),\\
& dt(p, u, g, \alpha)=(pg, u\cdot \alpha),\\
& \bigl(pg, u. \alpha, h, \beta\bigr)\circ \bigl(p, u, g, \alpha\bigr)= (p, u, gh, \alpha\cdot \beta),\\
& (p, u, g,\alpha)^{-1}=\bigl(pg, u\cdot \alpha, g^{-1}, -g^{-1}(\alpha)g^{-1}\bigr),\\
& \qquad 1_{(p, u)}=(p, u, e, 0).
\end{split}
\end{equation}
Also notice that the associated Lie algebroid of $[P \times G \rightrightarrows P]$ is given by the vector bundle $A=1^*{\rm{ker}} \, ds=\{(p,0,X) \mid 0 \in T_pP, X \in T_eG\}\simeq P \times L(G)$.  Now, taking the  differential of the action of $[P\times G\rightrightarrows P]$ on $P$, given by $pg(p,g)= pgg^{-1}=p$
in \cref{example:ordinary bundle with action groupoid}, we obtain the action of $[TP\times TG\rightrightarrows TP]$ on $TP,$
\begin{equation}\label{eqn:differential of action map of torsor of action groupoid}
    \begin{split}
        (pg, u.\alpha) \bigl(p, u, g, \alpha\bigr)=\bigl(p, v.(-g^{-1}(\alpha)g^{-1})\bigr).
    \end{split}
\end{equation}
We have  the unique Cartan connection $\mathbb{H}_{(p,g)}=T_pP \times {0}$ (\cref{example:connection on action groupoid}) on  $[P \times G \rightrightarrows P].$  Then the associated left, right splittings  of \cref{sequence diagram:short exact sequence with right translations} are 
\begin{equation}\label{eqn:splittings associtated to connection on action groupoid}
\begin{split}
&\sigma_{(p, g)}:T_pP \rightarrow  T_pP \times T_gG, \hspace{0.3cm} u  \mapsto   (u, 0),\\
&\psi_{(p, g)}:T_pP \times T_gG \rightarrow  A_{(pg, e)}, \hspace{0.3cm} (u,\alpha)  \mapsto \underbrace{(dR_{(pg, g^{-1})})_{(p, g)}(0,\alpha)}_{=(0,\alpha) \circ (0,0) \, \rm{by} \, \cref{eqn:description of right translations at tangent level}} 
 =\underbrace{(0,\alpha) \circ (0,0)}_{(0,0.\alpha) \, by \, \cref{eqn:differentials of structural maps of action groupoid}}=(0,0.\alpha).
\end{split}
\end{equation}
Since, $\sigma$ is a Cartan connection, the induced quasi representation of $\mathbb{X}$ on both $A$ and $TX_0$ reduces to a representation given as follows: For each $(p,g) \in P \times G$, the corresponding linear isomorphisms $\lambda_{(p,g)}^0 \colon A_{s(p,g)}=0 \times L(G) \rightarrow A_{t(p,g)}=0 \times L(G), \, \lambda_{(p,g)}^1:T_{s(p,g)}X_0=T_pP \rightarrow T_{t(p,g)}X_0=T_{pg}P$ are given by respectively as $\lambda_{(p, g)}^0(0, X)={}  -\psi_{(p, g)}l_{(p, g)}(0, X)$, and $\lambda_{(p,g)}^1(u)=u.0.$  By applying \cref{eqn:description of left translations at tangent level} on \cref{eqn:differentials of structural maps of action groupoid} and \cref{eqn:splittings associtated to connection on action groupoid} we obtain 
\begin{equation}
    \begin{split}
       &\lambda^1_{(p, g)}(0, X),
       ={}  (0, {\rm ad}_{g^{-1}}(X))\\
      & \lambda^1_{(p,g)}(u)={}  u.0.
    \end{split}
\end{equation}
%Our next step is calculating the quasi-action induced by $\sigma$. Observe that, for each $ g \in X_1,$ the associated linear map $\tau_g\colon T_{pg}P \rightarrow T_pP$ given by,
%\begin{equation}\label{eqn:quasi action of X on TP for the ordinary bundle with action groupoid}
%    \begin{split}
  %    \tau_{(p,g)}(u)={} & \bigl(u.(di)_{(p,g)^{-1}}\underbrace{\sigma_{(p,g)^{-1}}(u)}\bigr) \\
 %     & [{\rm The \, underbraced \, term \, }=(u,0) \, \rm{by \, \cref{eqn:splittings associtated to connection on action groupoid}}] \\
 %     ={} & \bigl(u.\underbrace{(di)_{(pg,g^{-1})}(u,0)}\bigr) \\
 %     & [{\rm The \, underbraced \, term \,}=(u.0,0) \, {\rm by \, \cref{eqn:differentials of structural maps of action groupoid}}] \\
%      ={} & \underbrace{u.(u.0,0)} \\
%      & [{\rm The \, underbraced \, term \,}=u.0 \, {\rm by \, \cref{eqn:differential of action map of torsor of action groupoid}}] \\
 %     ={} & u.0 \\
  %    ={} & u.g^{-1}
 %   \end{split}
%\end{equation}
%where in the last line, $u.g^{-1}$ is the action of $G$ on $T_{pg}P$ for the principal $G$-bundle $P \rightarrow M$(\cref{equation:Actions of G on classical bundles}.)  
 On the other hand, for each $p \in P$, differential of the map $\delta_p\colon s^{-1}(p) \rightarrow \pi^{-1}(\pi(p)), g\mapsto pg,$  produces the fundamental vector field
 $(d\delta)(p, p, 0, X)={}  \bigl(p,(d\delta_p)_{1_{a(p)}}(0, X)\bigr),$   by \cref{eqn:description of fundamental vector field map}, \cref{eqn:differentials of structural maps of action groupoid} and \cref{eqn:differential of action map of torsor of action groupoid} which is same as $(p, 0.X).$
\begin{equation}
    \begin{split}
       (d\delta)(p, p, 0, X) 
        ={}  (p, 0.X).
    \end{split}
\end{equation}
Hence, the \cref{sequence diagram:short exact sequence with fundamental vector field} 
turns into the following sequence of vector bundles:
\begin{equation}\label{eqn: short exact sequence for principal action groupoid bundles}
    \begin{tikzcd}
        0 \arrow[r] & P \times_{P}(P \times L(G) \oplus P \times TP) \arrow[r,"d\delta"] & TP \oplus P \times TP \arrow[r,"d\pi"]  & \pi^*TM \arrow[r] & 0.
    \end{tikzcd}
\end{equation}
Corresponding Atiyah sequence is the sequence of action groupoids of terms involved in \cref{eqn: short exact sequence for principal action groupoid bundles}.
\end{example}

\section{Connections on a principal Lie groupoid bundle}\label{section:connection as 1-form}
  In this section, we introduce the notion of a connection on a principal $\mathbb{X}$-bundle $\pi\colon P \rightarrow M$ for a choice of a connection $\sigma$ on the Lie groupoid $\mbbX.$ A connection, according to our definition, will be a certain type of $1$-form on $P$ that takes values in the associated Lie algebroid $A=1^* {\rm{ker}} \, ds$ (\cref{example:associated Lie algebroid}) of $\mathbb{X}.$
It should be noted that the framework introduced here relies on the Ehresmann connection we choose on the Lie groupoid $\mathbb{X}$. Since two entirely distinct types of connections, one is a connection $\sigma$ or $\mathbb{H}$ on a Lie groupoid (introduced in \cref{subsection:Connection of a Lie groupoid}) and the other is a connection on a principal $\mathbb{X}$-bundle (we are about to introduce here) will be involved here, to avoid any confusion we adopt the convention of calling the former a \textit{Lie groupoid connection} and later a \textit{principal Lie groupoid bundle connection}. Towards the end of this section, we will explicitly work out the dependence of our principal Lie groupoid bundle connection on the choice of a Lie groupoid connection.  
  
  First, we recall the definition of a vector bundle-valued differential form on a smooth manifold. 
\begin{defn}[\cite{tu2017differential}]\label{defn:vector bundle valued differential form}
  For a vector bundle  $E \rightarrow M$, a smooth \textit{$E$-valued differential $k$-form} is defined as a section of inner hom bundle ${\rm Hom}(\bigwedge^k{T^*M} \otimes E)$.  
\end{defn}
Equivalently, a $E$-valued differential $k$-form smoothly assigns an alternating $k$-linear map $T_pM \times \ldots \times T_pM \rightarrow E_p$, for each $p \in M$. The space of such forms is denoted by $\Omega^k(M, E)$.  In particular, one can see that an $E$-valued differential $1$-form is nothing but a bundle map between $TM$ and $E$.
 Note that a usual vector space $V$-valued differential form on $M$ can be viewed as an $M\times V\to M$ vector bundle valued differential form.

As a special case of the above definition, we have:
\begin{defn}\label{defn:Lie algebroid valued differential form}
     Let $p\colon A\to M$ be a Lie algebroid (\cref{defn:lie algebroid}) and $P$  a manifold with a smooth map $a\colon P \rightarrow M.$  Then a bundle map between $TP$ and $a^*A=P \times_{X_0} A$ will be called a \textit{Lie Algebroid valued differential $1$-form} on $P$.
\end{defn}
\begin{example}\label{ExamP:Ordinary Lie algebra valued form}
A common Lie algebra  $\mathcal L$-valued differential form on $M$ is a Lie algebroid $\mathcal L\to *$-valued differential form (see \cref{examp:Ordinary Lie algebra}).

\end{example}

\begin{example}\label{examp:Associated Lie algebra valued form}
Particularly for our interest,  consider the  associated Lie algebroid of a Lie groupoid $\mathbb{X}: X_1 \rightrightarrows X_0$
of \cref{example:associated Lie algebroid}. If $P\to M$ is a principal $\mbbX$-bundle, then we have Lie algebroid $A=1^*\ker ds$-valued   differential $1$-forms given by bundle maps between $TP$ and $a^*A=P \times_{X_0} A.$
\end{example}

%Now if we have a following principal $\mathbb{X}$-bundle, 
%\[
%\begin{tikzcd}
    %& P \arrow[dr, "a"] \arrow[dl] & X_1 \arrow[d, shift left] \arrow[d, shift right] \\
    %M & & X_0
%\end{tikzcd}
%\] 
%then one can pullback the Lie algebroid $A \xrightarrow[]{} X_0$ by the anchor map to get a vector bundle over $P$. 

    %Before that, let us recall how connection has been defined on a principal lie group bundle.  Suppose $G$ is a Lie group and $\pi:P \rightarrow M$ is a principal $G$-bundle. Let  L(G) be the associated lie algebra of $G$. Then, we saw that at the start of section 4, $G$ acts on both $P \times L(G)$ and $TP$ and the fundamental vector field map $d\delta:P \times L(G) \rightarrow TP$ is $G$-equivariant.  Then, a connection form $\omega$ on $P$ is a $E$-valued differential $1$-form, where $E=P \times L(G) \rightarrow P$, satisfying
    %\begin{equation}\label{eqn:connection form in classical setup}
       % \begin{split}
        % \omega_p \circ (d\delta_p)(A)={} & A, \, \text{for every} \, A \in L(G).  \\
        % \omega_{pg}(dR_g)_p(u)={} & ad_{g^{-1}}(\omega_p(u)), \text{for every} \, (p,u) \in TP, \, \text{where} \, R_g \,is \, \text{defined as in} \, \cref{eqn:right translation in torsor in classical setup}  
        %\end{split}
   % \end{equation}
    %So, a connection form on $P$ is a $ L(G)$-valued differential $1$-form $\omega$ such that it is  $G$-equivariant and a section of $d\delta$.% 
   Let  $\pi:P \rightarrow M$ be a $\mbbX$-bundle. Fix an Ehresmann connection $\sigma$ on $\mathbb{X}$.
   Then we saw  in \cref{section:Atiyah sequence} that  there exist quasi actions of $\mathbb{X}$ on both the vector bundles $P \times_{X_0} A$ and $TP.$  
\begin{defn}[Connection $1$-form]\label{defn:connection as 1-form}
    Let $\mbbX$ be a Lie groupoid and $\pi\colon P \rightarrow M$  a principal $\mathbb{X}$-bundle with anchor map $a\colon P \rightarrow X_0.$ Let $A \rightarrow X_0$ be the associated Lie algebroid of $\mathbb{X}$. Fix a connection $\sigma$ on $\mathbb{X}.$  Let $\lambda, \tau$ be the respective quasi actions on $A$ and $TP$ respectively defined  in \cref{quasi action of X on A}, \cref{eqn:quasi action of X on TP} induced by $\sigma.$ Then, a \textit{connection $1$-form} on $P$ is a $A$-valued $1$-form  $\omega\colon TP \rightarrow P \times_{X_0} A$ satisfying the following conditions:
       \begin{eqnarray}
       &&\omega_{pg}(\tau_g(u))=\lambda_{g^{-1}}\bigl(\omega_p(u)\bigr),\nonumber\\
       &&{\rm for\, every}\,\, u \in T_pP\, \, {\rm and}\, \, (p, g)\, \, {\rm such\, that}\, \, a(p)=t(g). \label{eqn:X-equivariance of connection form}\\
       &&\omega_p \circ d\delta_p(\alpha)=\alpha,\, {\rm for\, \,  every} \, \alpha \in A_{a(p)}. \label{eqn:connection form is a section of fundamental vector field map}
   \end{eqnarray}
    We often refer to the pair $(\omega, \sigma)$ as a \textit{connection pair on the Lie groupoid $\mathbb{X}$-bundle} $\pi\colon P\to M.$
\end{defn}
Though  \cref{defn:connection as 1-form} appears to be a straightforward generalization of the classical definition, its remarkable adaptability to the framework of representation up to homotopy should be evident from the following result, which will be used later on to relate a connection $1$-form with the splitting of the Atiyah sequence in \cref{Def:Atiyah sequence for Lie groupoid bundle}. 
\begin{prop}\label{rem:preservation of error by connection form}
      Suppose $\omega: TP \rightarrow P \times_{X_0} A$ is a connection form associated with an Ehresmann connection $\sigma$.  Then,
      \begin{equation}
          \begin{split}
              \omega_{pgh}\bigl(\tau_{gh}(u)-\tau_h\tau_g(u)\bigr)={} & \lambda_{h^{-1}g^{-1}}\omega_p(u)-\lambda_{h^{-1}}\lambda_{g^{-1}}\omega_p(u),
          \end{split}
      \end{equation}
    for every $u \in T_pP$ and $(p,g,h)$ satisfying $a(p)=t(g),s(g)=t(h).$ In other words,
    \begin{equation}
        \begin{split}
            \omega_{pgh}\bigl( 0.l_{gh}K_{\sigma}^{bas}(h^{-1}, g^{-1})(da)_p(u)\bigr)={} & K_{\sigma}^{bas}(h^{-1},g^{-1})(dt)(\omega_p(u)).
        \end{split}
    \end{equation}
\end{prop}
\begin{proof}
    First observe that, by applying \cref{eqn:X-equivariance of connection form} for the triple $(\tau_g(u),pg,h)$, we get that
    \begin{equation}\label{eqn:computation of connection form quasi action of g followed by h}
        \begin{split}
        \omega_{(pg)h}\Bigl(\tau_h\bigl(\tau_g(u)\bigr)\Bigr)={} & \lambda_{h^{-1}}\Bigl(\underbrace{\omega_{pg}\bigl(\tau_g(u)\bigr)}\Bigr) \\
        &[{\rm The \, underbraced \, term \,}=\lambda_{g^{-1}}\omega_p(u) \, {\rm by \, \cref{eqn:X-equivariance of connection form}}] \\
        ={} & \lambda_{h^{-1}}\lambda_{g^{-1}}\omega_p(u).
         \end{split}
    \end{equation}
    Also, by \cref{eqn:X-equivariance of connection form} for the triple $\bigl(u,p,gh\bigr)$ we get that,
    \begin{equation}\label{eqn:computation of connection form quasi action of gh}
        \begin{split}
            \omega_{pgh}\bigl(\tau_{gh}(u)\bigr)={} & \lambda_{h^{-1}g^{-1}}\omega_p(u).
        \end{split}
    \end{equation}
    Hence, the result follows from \cref{eqn:computation of connection form quasi action of g followed by h} and \cref{eqn:computation of connection form quasi action of gh}.
\end{proof}
It would often be convenient to work with the following equivalent definition of a principal Lie groupoid bundle connection. The result obtained in  \cref{thm:equvariance of fundamental vector field map} is pivotal in ensuring this correspondence is almost identical to that of a classical principal bundle.
\begin{thm}\label{thm:chartacterization of connection form in terms of subbundle}
    Let $\pi\colon P \rightarrow M$ be a principal $\mathbb{X}$-bundle and $\sigma$ a connection on $\mathbb{X}$.  Then there is a one-to-one correspondence between the set of connection $1$-$forms$ $\omega$ to the set of all smooth subbundles $\mathcal{H}$ of $TP\to P$ satisfying 
    \begin{align}
        TP={} & \mathcal{H} \oplus {\ker} d\pi,\label{eqn:subbundle splitting ker dpi}\\
        \mathcal{H}_{pg}={} &\mathcal{H}_p.g, \qquad {\rm for}\,\, a(p)=t(g),\label{eqn:stability of connection under quasi action}
        \end{align}
where $\mathcal{H}_p.g$ denotes the restricted  quasi action of $\mathbb{X}$ on $\mathcal{H}$ induced by $\sigma.$
        
\end{thm}
\begin{proof}
    Let $\mathcal{H}\subset TP$ be a subbundle satisfying conditions \cref{eqn:subbundle splitting ker dpi} and \cref{eqn:stability of connection under quasi action}.  Splitting a vector $u$ as $u=u_1+u_2$, $u_1 \in \mathcal{H}_p$ and $u_2 \in (\ker d\pi)_p$, we define $\omega_p(u)=(d\delta_p)_{1_{a(p)}}^{-1}(u_2)$, using the diffeomorphism $d\delta:P \times_{X_0} A \rightarrow TP.$ 
    We verify that $\omega$ is a connection $1$-form by observing that $\mathcal{H}$ and $\ker d\pi$ both are stable under the quasi action.  On the other hand, by \cref{thm:equvariance of fundamental vector field map} we obtain $(d\delta_{pg})\bigl(\lambda_{g^{-1}}(d\delta_p)^{-1}(u_2)\bigr)=\tau_g(d\delta_p)\circ(d\delta_p)^{-1}(u_2)=\tau_g(u_2)$.  The diffeomorphism of $(d\delta)$ implies $\omega_{pg}(\tau_g(u))=\lambda_{g^{-1}}\omega_p(u)$. Verification of  \cref{eqn:connection form is a section of fundamental vector field map} is immediate from the definition of $\omega$.  
    
   Conversely, given a connection $1$-form $\omega$, we define $\mathcal{H}:=\ker \omega$, which is stable under the quasi action by  \cref{eqn:X-equivariance of connection form}.   Any $u \in T_pP$ can be split into  the vectors $u-d\delta_p \circ \omega_p(u)\in \ker \omega$ and   $d\delta_p \circ \omega_p(u)\in \ker d\pi.$  
\end{proof}

\begin{rem}\label{rem:quasi action becoming full action}
    It should be noted that if $\mathcal{H}$ is a connection on a principal $\mathbb{X}$-bundle $\pi\colon P \rightarrow M$, then the restricted quasi action of $\mathbb{X}$ on $\mathcal{H}$ is in fact a full action.  Consider $u \in \mathcal{H}_p$ to see this. Then \cref{Rem:error of quasi action on TP lies insie vertical} implies $\tau_{gh}(u)-\tau_h\tau_g(u)$ contained in  $(\ker d\pi)_{pgh}$.  But, by \cref{eqn:stability of connection under quasi action} both $\tau_{gh}(u)$ and $\tau_h\tau_g(u)$ are inside $\mathcal{H}_{pgh}$.  Hence, we must have $\tau_{gh}(u)=\tau_h\tau_g(u)$, for every $u \in \mathcal{H}_p$ and $(p, g)$ satisfying $a(p)=t(g)$.
\end{rem}

\begin{rem}\label{rem:Casual Moerdijk connection}
Let $\pi\colon P \rightarrow M$ be a principal $\mathbb{X}$-bundle and  $\mathcal{F}$  a foliation on the base manifold $M$.
 One can show that the pullback foliation $(d\pi)^{-1}(\mathcal{F})\subset TP$ on $P$ is stable under the restriction of $\tau$ onto it, and thus it makes sense to consider a splitting $(d\pi)^{-1}(\mathcal{F})={\mathcal H}_{\mathcal F}\oplus \ker d\pi.$ In fact this provides a slightly general version of  \textit{$\mathcal{F}$-partial connection} on  $\pi\colon P \rightarrow M$  in \cite{moerdijk2003introduction}. The relation between our connection on a principal $\mathbb{X}$-bundle $P\to M$ with the foliation and its consequences will be explored in more detail in our upcoming paper.
\end{rem}

Recall that we have seen in \cref{prop: fibred category related to quasi equivariant bundle} that  a quasi equivariant $\mbbX$-bundle upto homotopy $\bigl(\pi:E=E^0\oplus E^1 \rightarrow P, \xi, \tau, \Upsilon\bigr)$ produces a fibered category  $\mathbb{X}^{{\rm op}}\to {\rm Cat}$.  Hence, by \cref{prop:P times A equivariant} and \cref{prop: TP plus a^*TX_0 is a equivariant bundle upto homotopy}, fixing an Ehresmann connection $\sigma$ on $\mathbb{X},$ we obtain a pair of fibered categories $F_{{\rm Ad}(P)}, \, F_{{\rm At}(P)}:\mathbb{X}^{\rm op} \rightarrow {\rm Cat},$ respectively corresponding to the equivariant bundles upto homotopy $P \times_{X_0} A\oplus a^*TX_0$ and $TP \oplus a^*TX_0.$  Then,  the map $(d\delta,{\rm Id}):P \times_{X_0} A \oplus a^*TX_0 \rightarrow TP \oplus a^*TX_0$ defined in \cref{prop: morphism between P times A and TP bundles upto homotopy}, produces a morphism of fibred categories $\eta:F_{{\rm Ad}(P)} \Longrightarrow F_{{\rm At}(P)}$ with a functor $\eta_x:F_{{\rm Ad}(P)}(x) \rightarrow F_{{\rm At}(P)}(x)$, for each $x \in X_0$ defined as $(\alpha \oplus v, \gamma) \mapsto (d\delta(\alpha)\oplus v, \gamma)$ and $(\alpha \oplus v, \gamma, \alpha' \oplus v', \beta) \mapsto (d\delta(\alpha) \oplus v, \gamma,  d\delta(\alpha') \oplus v', \beta). $  Moreover, $\eta_x$ is linear over the vector components.
%\begin{equation}\label{eqn: description of eta}
 %       \begin{split}
  %         & (\eta_x)_0:{\rm Obj}\bigl(F_{{\rm Ad}(P)}(x)\bigr) \rightarrow {\rm Obj}\bigl(G_{{\rm At}(P)}(x)\bigr), \, (\alpha \oplus v, \gamma) \mapsto (d\delta(\alpha)\oplus v, \gamma), \\
   %        & (\eta_x)_1:{\rm Mor}\bigl(F_{{\rm Ad}(P)}(x)\bigr) \rightarrow {\rm Mor}\bigl(G_{{\rm At}(P)}(x)\bigr), \, (\alpha \oplus v, \gamma, \alpha' \oplus v', \beta) \mapsto (d\delta(\alpha) \oplus v, \gamma,  d\delta(\alpha') \oplus v', \beta). 
    %    \end{split}
    %\end{equation}

%\begin{equation}\label{eqn: F and G as pseudo functors}
   % \begin{split}
    %    & {\rm Obj}\bigl(F_{{\rm Ad}(P)}(x)\bigr)={} \Bigl\{(\alpha\oplus v,\gamma) \mid \alpha\oplus v  \in (a \circ {\rm pr})^{-1}(x),\, \gamma \in t^{-1}(x) \Bigr\} \\
     %   & {\rm Mor}\bigl(F_{{\rm Ad}(P)}(x)\bigr)={} \Bigl\{(\alpha\oplus v, \gamma, \alpha'\oplus v', \beta) \mid s(\gamma)=t(\beta) \Bigr\}  \\
      %  &{\rm Obj}\bigl(F_{{\rm At}(P)}(x)\bigr)={} \Bigl\{(u \oplus v, \gamma) \mid u \oplus v \in (a \circ {\rm pr})^{-1}(x)\Bigr\} \\
       % & {\rm Mor}\bigl(F_{{\rm At}(P)}(x)\bigr)={} \Bigl\{(u \oplus v, \gamma, u'\oplus v', \beta) \mid s(\gamma)=t(\beta)\Bigr\}.
    %\end{split}
%\end{equation}
The following proposition characterises a connection on a principal $\mbbX$-bundle $\pi:P \rightarrow M$ as a morphism of fibred categories from $F_{{\rm At}(P)}$ to $F_{{\rm Ad}(P)}.$  
\begin{prop}\label{prop: natural transformation between F and G coming from ddelta and omega}
  Let $\pi:P \rightarrow M$ be a principal $\mbbX$-bundle and $\sigma$  an Ehresmann connection on $\mbbX.$  Then,  a connection form $\omega:TP \rightarrow P \times_{X_0} A$ of \cref{defn:connection as 1-form} produce a morphism of fibred categories $\Theta:F_{{\rm At}(P)} \Longrightarrow F_{{\rm Ad}(P)}$ satisfying $\Theta \circ \eta={} {\rm Id}_{F_{{\rm Ad}(P)}},$ with a diffeologically smooth functor $\Theta_x:F_{{\rm At}(P)}(x) \rightarrow F_{{\rm Ad}(P)}(x)$ linear on appropriate components, for each $x \in X_0.$  Conversely, if $\Theta:F_{{\rm At}(P)} \Longrightarrow F_{{\rm Ad}(P)}$ is a morphism of fibred categories satisfying linearity for the diffeologically smooth functor $\Theta_x:F_{{\rm At}(P)}(x) \rightarrow F_{{\rm Ad}(P)}(x)$  and $\Theta \circ \eta={} {\rm Id}_{F_{Ad}(P)}$, then there is a unique connection form $\omega:TP \rightarrow P \times_{X_0} A.$
  \end{prop}
\begin{proof}
     For each $x \in X_0,$  the functor $\Theta_x:F_{{\rm At}(P)}(x) \rightarrow F_{{\rm Ad}(P)}(x)$, given as $ (u\oplus v, \gamma) \mapsto (\omega(u) \oplus v, \gamma),$ and $(u \oplus v, \gamma, u' \oplus v', \beta) \mapsto (\omega(u) \oplus v, \gamma, \omega(u') \oplus v', \beta).$  Verification is straightforward that $\Theta$ is a natural transformation of pseudo functors from $F_{{\rm At}(P)}$ to $F_{{\rm Ad}(P)}.$   
     
     Conversely, given  $\Theta:F_{{\rm At}(P)} \Longrightarrow F_{{\rm Ad}(P)}$ satisfying $\Theta \circ \eta={} {\rm Id}_{{\rm Ad}(P)}$, we associate a connection form  $\omega:TP \rightarrow P \times_{X_0} A,$ $\omega_p(u)={}{\rm pr}_1 \circ \bigl((\Theta_{a(p)})^1(u\oplus 0, 1_{a(p)})\bigr)$, for every $u \in T_pP$, where ${\rm pr}_1:P \times_{X_0} A \oplus a^*TX_0 \rightarrow P \times_{X_0} A$ is the first projection map and $(\Theta_{a(p)})^1$ is the vector component map of $\Theta_{a(p)}.$
     
\end{proof}

The principal Lie groupoid bundle connection provides a general framework for various natural examples.
\begin{example}
    Viewing an ordinary principal $G$-bundle  $\pi\colon P\to M$ as a Lie groupoid $[G \rightrightarrows \star]$-bundle as in \cref{example:ordinary bundle}, we note that $[G \rightrightarrows \star]$ admits only the trivial connection $\mathbb{H}=0,$ and thus  quasi-equivariant bundles $P \times_{\star} A$ and $TP$  reduce to the adjoint bundle ${\rm Ad}(P)$ and the Atiyah bundle ${\rm At}(P)$ respectively.  Hence, a connection on the $G$-bundle $\pi\colon P\to M$  is same as a connection  on the principal $[G \rightrightarrows \star]$-bundle $\pi\colon P \rightarrow M.$
\end{example}
\begin{example}    Let $\mathbb{H}$  be a Cartan connection on a Lie groupoid $\mbbX;$ that means it is an Ehresmann connection (\cref{defn: connection on a lie groupoid}) with a left splitting $\sigma$ of \cref{sequence diagram:short exact sequence with right translations} satisfying     $\sigma_{gh}(v)={}  \bigl(gh,\sigma_g(\lambda_h(v).\sigma_h(v)\bigr), {\rm{for \, every}} \, g,h \in X_1,v \in T_{s(h)}X_0.$ 
 Note that $\mathbb{H}=\text{Im}(\sigma)$.  Now, consider the unit principal Lie groupoid bundle of \cref{example:unit bundle}.  
  One sees that the sequence in \cref{sequence diagram:short exact sequence with fundamental vector field} is the same as in  \cref{sequence diagram:short exact sequence with left translations}. Now, recall by \cref{Remark:Sequence with respect to target} the subbundle defined by $\tilde{\mathbb{H}}:=(di)(\mathbb{H})$ complements  $\ker dt$:
$TX_1=\tilde{\mathbb{H}}\oplus \ker dt$
  and the corresponding left splitting $\tilde \sigma$ of \cref{sequence diagram:short exact sequence with left translations} is related to $\tilde{\mathbb{H}}$ by ${\rm Im}(\tilde \sigma)=\tilde{\mathbb{H}}.$

  To see that  $\tilde{\mathbb{H}}$ is in fact a connection on the principal $\bigl(\mathbb{X}, \sigma\bigr)$-bundle $t:X_1 \rightarrow X_0$,  we take $(g, h, u)$ satifying $s(h)=t(g)$ and  $u \in \tilde{\mathbb{H}}_h$. Then $u=\tilde{\sigma}_h(v)=(di)_{h^{-1}}\sigma_{h^{-1}}(v)$, for some $v \in T_{t(h)}X_0$. Hence, the induced quasi-action gives
     \begin{equation}
     \begin{split}
         \tau_g(u)={} & \bigl(hg,u.(di)_{g^{-1}}\sigma_{g^{-1}}(ds)_h(u)\bigr) \\
         ={} & \bigl(hg,\tilde{\sigma}_{hg}(v)\bigr),
         \end{split}
         \end{equation}
         implying $\tilde{\mathbb{H}}_h g=\tilde{\mathbb{H}}_{hg}.$  If we have started with an Ehresmann connection $\mathbb{H}$, then a connection on the unit bundle $t:X_1 \rightarrow X_0$ is a smooth subbundle $\mathcal{H}$ satisfying 
         \begin{equation}
             \begin{split}
                 & TP={}  \mathcal{H} \oplus {\rm{ker}} \, dt,  \\
                 & \tau_g(u)={} u.(di)_{g^{-1}}\sigma_{g^{-1}}(ds)_{h}(u) \in \mathcal{H}_{hg},
             \end{split}
         \end{equation}
         for every $(h,u) \in TX_1$ such that $s(h)=t(g).$
\end{example}
The next example shows that our notion of connection behaves well under pullback.
\begin{example}
    Suppose $\pi:P \rightarrow M$ is a principal $\bigl(\mbbX, \sigma\bigr)$-bundle equipped with a connection $\mathcal{H}$.  Let $f: N \rightarrow M$ be a smooth map.  Then by \cref{example:pullback principal bundle},  ${\text{pr}}_1:N \times_{M} P \rightarrow N$ is a principal $\bigl(\mathbb{X},\sigma\bigr)$ bundle.  
     The subbundle  $$\tilde{\mathcal{H}}=\bigl\{(n, u, p, v) \big | u \in  T_{n}N, v \in \mathcal{H}_p, f(n)=\pi(p),(df)_n(u)=(d\pi)_p\bigr\}\subset T(N\times _{M} P)$$ is a connection  on the pullback principal bundle.  
     To see this, observe that 
    $\ker {d\text{pr}}_1=\{(n, 0),(p, v) \mid f(n)=\pi(p), v \in {{\text{ker}}} \, (d\pi)_p\},$  and  any vector $(u,v) \in T_{(n,p)}(N \times_{M} P)$ can be expressed as  $(u, v)=(u, v_1)+(0, v_2)$, where $v_1 \in \mathcal{H}_p$ and $v_2 \in {{\text{ker}}} (d\pi)_p$.  Moreover, suppose $(u,v) \in \tilde{\mathcal{H}}_{(n,p)} \cap {\text{ker}} \, (d\text{pr}_1)_{(n,p)}$, then $u=0$ and $v \in {\text{ker}} \, (d\pi)_p \cap \mathcal{H}_p$.  Since $\mathcal{H}$ is a connection on $\pi:P \rightarrow M$, we also have  $v=0$.  Thus, $$T(N \times_{M} P)=\tilde{\mathcal{H}} \oplus {\text{ker}} \, d\text{pr}_1.$$ Now  $\mathbb{X}$ acts on $N \times_{X_0} P$ by $(n, p) g=(n, pg)$ for the anchor map $a \circ \text{pr}_2$.
   For  $(u,v) \in \tilde{\mathcal{H}}_{(n, p)}$ and $g \in X_1$ satisfying $a(p)=t(g)$, by definition
\begin{align}
    \tau_g(u,v)={} & \bigl(u,v\bigr).\bigl((di)_{g^{-1}}\sigma_{g^{-1}}(d(a \circ \text{pr}_2))_{(n,p)}(u,v)\bigr), \notag \\
    ={} & \bigl(u\bigr).\bigl(v.(di)_{g^{-1}}\sigma_{g^{-1}}(da)_p(v)\bigr). 
\end{align}
Since, $\mathcal{H}$ is a connection, it follows that $\bigl(v.(di)_{g^{-1}}\sigma_{g^{-1}}(da)_p(v)\bigr) \in \mathcal{H}_{pg}$ and in turn $\tau_g(u, v) \in \tilde{\mathcal{H}}_{(n, pg)}.$
\end{example}
\begin{example}\label{example:Connection on torsor of action groupoid}
Connections on a principal $G$-bundle $\pi:P \rightarrow M$ are in one-to-one correspondence with the connections on the principal Lie groupoid bundle  $[P \times G \rightrightarrows P]$ bundle $\pi:P \rightarrow M$  in \cref{example:ordinary bundle with action groupoid}.    For a connection $\mathcal{H}$ on the principal $G$-bundle $\pi:P \rightarrow M,$  the $G$ invariant splitting $TP ={}  \mathcal{H} \oplus {\rm{ker}} \, d\pi,$ defines a connection on the $[P\times G\rightrightarrows P]$-bundle $\pi\colon P\to M$  associated to the canonical connection ${\mathbb H}=TG\subset T(P\times G)$ on $[P\times G\rightrightarrows P]$ in \cref{example:connection on action groupoid}.
Converse is straightforward. 
\end{example}

   %%%%%%%%SUBSECTION%%%%%%%%%%%%%%%

\subsection{Principal Lie groupoid bundle connections for different choices of Ehresmann connections}\label{subsec:Relation between connection forms induced by different choices of Ehresmann connections}

The entire set-up of connections on a principal Lie groupoid $\mbbX$-bundle $\pi\colon P\to M$ of the preceding sections depended on a choice of an Ehresmann connection on the Lie groupoid $\mathbb{X}$ to induce quasi action of $\mathbb{X}$ on both $P \times_{X_0} A$ and $TP$.   The natural question is how different choices of Ehresmann connections affect connections on a Lie groupoid bundle. To be more precise, suppose $\pi\colon P\to M$ a principal $\mbbX$-bundle, and $(\omega, \sigma)$ and $(\tilde \omega, \tilde \sigma)$  connection pairs for the Lie groupoid connections $\sigma$ and $\tilde{\sigma}$ on $\mathbb{X}$ respectively. We will investigate the relation between $\omega$ and $\tilde{\omega}.$

Let $A=1^* \ker ds\to X_0$  be the Lie algebroid of $\mbbX,$ and $\tau, \lambda$ be the quasi-actions of $\mathbb{X}$ on $TP$ and $A,$ respectively induced by choice of the Lie groupoid connection $\sigma$ (see \cref{quasi action of X on A}, \cref{eqn:quasi action of X on TP}).  And, $\tilde{\tau}, \tilde{\lambda}$ be the likewise for the Lie groupoid connection $\tilde{\sigma}.$ Albeit in a different context, in \cite{cran}, the authors address a similar question relating representations up to homotopy (summarized in \cref{subsection:Connection of a Lie groupoid})  for different choices of the Lie groupoid connections.

  Consider a connection pair $(\omega, \sigma)$ on the Lie groupoid $\mbbX$-bundle $\pi\colon P\to M$.  Let $TP=\mathcal{H}\oplus \ker d\pi$ be the corresponding splitting of the tangent bundle $TP$ (\cref{thm:chartacterization of connection form in terms of subbundle}). Given a vector $u\in T_pP,$ split $u$ as $u=u_1+u_2,$  for $u_1 \in \mathcal{H}_p$  and $u_2 \in (\ker d\pi)_p.$  The connection form $\omega$ and $\mathcal{H}$ are related by  $\omega_p(u)=(d\delta_p)_{1_{a(p)}}^{-1}(u_2)$, and $\omega$ is $\mathbb{X}$-equivariant;  $\omega_{pg}\bigl({\tau}_g(u)\bigr)={\lambda}_{g^{-1}}\omega_p(u)$.  For another choice of the connection $\tilde{\sigma}$ on the Lie groupoid, let us compute 
\begin{equation}\label{eqn:Difference tau tilde tau}
    \begin{split}
    \tau_g(u)-\tilde{\tau}_g(u)={} &  \bigl(0.(di)_{g^{-1}}\underbrace{(\sigma_{g^{-1}}-\tilde{\sigma}_{g^{-1}})(da)_p(u)}\bigr).
    \end{split}
\end{equation}
We note that the underbraced term in the  equation above belongs to $({\rm{ker}} \, ds)_g,$ and  thus, there exists a unique $\alpha \in A_{t(g)}$ such that 
\begin{equation}\label{eqn:dR(alpha)=sigma-tilde(sigma)}
\begin{split}
(dR_{g^{-1}})_{1_{t(g^{-1})}}(\alpha)={} & (\sigma_{g^{-1}}-\tilde{\sigma}_{g^{-1}})(da)_p(u).
\end{split}
\end{equation}
Then \cref{eqn:Difference tau tilde tau} reduces as
\begin{equation}
    \begin{split}
      \tau_g(u)-\tilde{\tau}_g(u)={} & \bigl(0.\underbrace{(di)_{g^{-1}}(dR_{g^{-1}})_{1_{t(g^{-1})}}}(\alpha)\bigr) \\
      &[{\rm The \, underbraced \, term \,}=(dL_g)_{1_{s(g)}}(di)_{1_{s(g)}} \, {\rm by \, \cref{diagram:i comp R=L comp i}}]\\
      ={} & \bigl(0.\underbrace{(dL_{g^{-1}})_{1_{s(g)}}(di)_{1_{s(g)}}(\alpha)}_{=0.(di)_{1_{s(g)} \, \rm{ by \, \cref{eqn:description of left translations at tangent level}}}}\bigr)\\
      ={} & \Bigl(\underbrace{0.\bigl(0.(di)_{1_{s(g)}}(\alpha)\bigr)}\Bigr) \\
      & [{\rm The \, underbraced \, term \,}=\bigl((0.0).(di)_{1_{s(g)}}(\alpha)\bigr) \, {\rm by \, \cref{diagram:multiplicative action of torsor at tangent space level}}] \\
      ={} & \bigl(\underbrace{(0.0)}_{=0}.(di)_{1_{s(g)}}(\alpha)\bigr)\\
       ={} &\bigl(\underbrace{0.(di)_{1_{s(g)}}(\alpha)}\bigr)\\
      & [{\rm The \, underbraced \, term \,}=(d\delta_{pg})_{1_{a(pg)}}(\alpha) \, {\rm by \, \cref{eqn:description of fundamental vector field map}}] \\
        ={} &  (d\delta_{pg})_{1_{s(g)}}(\underbrace{\alpha}) \\
      &  [{\rm The \, underbraced \, term \,}=(dR_{g^{-1}})_{1_{t(g^{-1})}}^{-1}(\sigma_{g^{-1}}-\tilde{\sigma}_{g^{-1}})(da)_p(u) \, {\rm by \, \cref{eqn:dR(alpha)=sigma-tilde(sigma)}}] \\
         ={} & (d\delta_{pg})_{1_{s(g)}}\bigl(\underbrace{(dR_{g^{-1}})_{1_{t(g^{-1})}}^{-1}}(\sigma_{g^{-1}}-\tilde{\sigma}_{g^{-1}})(da)_p(u)\bigr) \\
      & [{\rm The \, underbraced \, term \,}=(dR_g)_{g^{-1}} \, {\rm by \, \cref{eqn:Inverse of dR_g}}]\\
         ={} & (d\delta_{pg})_{1_{a(pg)}}\bigl(\underbrace{(dR_g)_{g^{-1}}(\sigma_{g^{-1}}-\tilde{\sigma}_{g^{-1}})}(da)_p(u)\bigr).
    \end{split}
\end{equation}
\begin{rem}
We point out here  that the map in the underbraced term above provides 
an isomorphsim between representations upto homotopy on the graded bundle $A \oplus TX_0$ induced by Ehresmann connections $\sigma, \tilde{\sigma}$ in Proposition 3.16 of \cite{cran}.
\end{rem}
Finally, we arrive at
\begin{equation}\label{eqn:inverse image of difference of tau g and tilde tau g on u}
    \begin{split}
   \tau_g(u)-\tilde{\tau}_g(u)={} & (d\delta_{pg})_{1_{a(pg)}}\bigl(\underbrace{(dR_g)_{g^{-1}}(\sigma_{g^{-1}}-\tilde{\sigma}_{g^-1})}_{=\alpha \, {\rm by \, \cref{eqn:dR(alpha)=sigma-tilde(sigma)}}}(da)_p(u)\bigr)=(d\delta_{pg})_{1_{a(pg)}}(\alpha).  
    \end{split}
    \end{equation}
In conclusion 
\begin{equation}\label{eqn:connection form difference of tau g and tilde tau g on u}
    \begin{split}
        \omega_{pg}\bigl(\tau_g(u)-\tilde{\tau}_g(u)\bigr)={} & \omega_{pg} \circ (d\delta_{pg})_{1_{a(pg)}}(\alpha)={}  \alpha.
    \end{split}
\end{equation}
The fundamental vector field map is equivariant regardless of the choice of a connection.  
\begin{equation} \label{eqn:equivariance of fundamental vector field irrespective of connections}
\begin{split}
(d\delta_{pg})_{1_{a(pg)}}\bigl(\lambda_{g^{-1}}\omega_p(u)\bigr) ={} & \tau_g\bigl((d\delta_p)_{1_{a(p)}} \circ \omega_p(u)\bigr) =   \tau_g(u_2).\\
(d\delta_{pg})_{1_{a(pg)}}\bigl(\tilde{\lambda}_{g^{-1}}\omega_p(u)\bigr) = {} & \tilde{\tau}_g \bigl((d\delta_p)_{1_{a(p)}} \circ \omega_p(u)\bigr)=\tilde{\tau}_g(u_2).
    \end{split}
\end{equation}
From \cref{eqn:equivariance of fundamental vector field irrespective of connections},  \cref{eqn:inverse image of difference of tau g and tilde tau g on u}, and with a bit of calculation, one finds

\begin{equation}\label{eqn: connection transormation different connections}
\begin{split}
    \omega_{pg}\bigl(\tilde{\tau}_g(u)\bigr) 
    ={} & \tilde{\lambda}_{g^{-1}}\omega_p(u)+(d\delta_{pg})_{1_{a(pg)}}^{-1}\bigl(\tilde{\tau}_g(u_1)-\tau_g(u_1)\bigr).
\end{split}
\end{equation}
Thus, the connection form on the Lie groupoid bundle for one choice of Ehresmann connection on the Lie groupoid will not satisfy the required condition of equivariancy for another choice of Ehresmann connection; that is, if $(\omega, {\sigma})$
is a connection pair on $\mbbX$-bundle, then $(\omega, \tilde{\sigma})$ may not be the same! However, we will observe that if $(\tilde \omega, \tilde{\sigma})$ is a connection pair, then there is an interesting relation between the two. 
\begin{prop}\label{prop:inducing a connection form from a theta}
    Let $(\omega, \sigma)$ be a connection pair on the $\mathbb{X}$-bundle $\pi\colon P\to M$ and $\mathcal H\subset TP$ the subbundle associated to $\omega.$ Choose another Ehressman connection $\tilde{\sigma}$  on the Lie groupoid $\mathbb{X}$.  If $\theta$ is a smooth Lie algebroid $(1^*\ker ds\to X_0)$-valued differential $1$-form satisfying 
\begin{equation}\label{eqn:Conditions on theta}
        \begin{split}
&\theta \circ d\delta=0,  \\      
            &\theta_{pg}\bigl(\tilde{\tau}_g(u)\bigr)={} \tilde{\lambda}_{g^{-1}}\theta_p(u)+(d\delta_{pg})_{1_{a(pg)}}^{-1}\bigl(\tilde{\tau}_g(u_1)-\tau_g(u_1)\bigr),
             \end{split}
    \end{equation}   
  for $u=u_1+u_2\in T_pP$ with $u_1\in \mathcal{H}_p, u_2\in (\ker 
 d\pi)_p,$ then $\bigl(\omega-\theta, \tilde{\sigma}\bigr)$ is  a connection pair  on $\mathbb{X}$-bundle $\pi\colon P\to M.$   Conversely if $(\omega, \sigma)$ and $(\tilde \omega, \tilde\sigma)$ are two connection pairs on   $\mathbb{X}$-bundle $\pi\colon P\to M,$ then $\omega-\tilde{\omega}$ satisfies transformation conditions in \cref{eqn:Conditions on theta}.
 
\end{prop}
\begin{proof}
   A direct consequence of \cref{eqn: connection transormation different connections} and \cref{eqn:equivariance of fundamental vector field irrespective of connections}.     
\end{proof}

In fact  \cref{prop:inducing a connection form from a theta} allows us to define a groupoid $\mathbb{CONNECTION}\bigl(\mathbb{X}, P\to M\bigr)$ of connections for a given $\mbbX$-bundle $\pi\colon P\to M.$ The objects are connection pairs $\bigl(\omega,\sigma\bigr)$ and
an arrow is given by $\bigl(\omega,\sigma,\theta,\tilde \sigma\bigr)$
with source and target $s\bigl(\omega,\sigma,\theta,\sigma'\bigr)=(\omega, \sigma)$ and $t\bigl(\omega,\sigma,\theta,\tilde\sigma\bigr)=(\omega-\theta, \tilde \sigma)$ respectively,  where $\theta$ is as in \cref{eqn:Conditions on theta}. 
\begin{equation}\label{eqn:Category of connections}
    \begin{split}
        {\rm{Obj}}\biggl({\mathbb{CONNECTION}}\bigl(\mathbb{X}, P\to M\bigr)\biggr)={}& \Bigl\{\bigl(\omega,\sigma)\bigr)\Bigr\},\\
{\rm{Mor}}\biggl({\mathbb{CONNECTION}}\bigl(\mathbb{X}, P\to M\bigr)\biggr)={}& \Bigl\{\bigl(\omega,\sigma,\theta,\tilde \sigma\bigr)\Bigr\}. 
    \end{split}
\end{equation}
The composition is 
$\bigl(\omega',\sigma',\theta ',\sigma''\bigr)\circ \bigl(\omega,\sigma,\theta ,\sigma'\bigr):= \bigl(\omega,\sigma,\theta '+\theta ,\sigma''\bigr).
$
\begin{rem}
In this paper, we will not make any further use of $\mathbb{CONNECTION}\bigl(\mathbb{X}, P\to M\bigr).$ We only make a passing remark that this category enjoys an intricate relationship with the gauge transformations on the  $\mbbX$-bundle $\pi\colon P\to M$, which will be explored in detail in our upcoming paper.
\end{rem}

\subsection{Connections and quasi connections}\label{subsection:connections and quasi connections}
As we remarked earlier,  splittings of the Atiyah sequence in \cref{sequence diagram:Atiyah sequence in classical setup} associated with a classical principal bundle and its connections are in one-to-one correspondence. In \cref{section:Atiyah sequence}  (\cref{sequence diagram:ATIYAH sequence}) we have introduced an analogue of the Atiyah sequence for a Lie groupoid $\mbbX$-bundle $\pi\colon P\to M.$ In this section we will address the relation between the sequence in  \cref{sequence diagram:ATIYAH sequence} and the framework of connections developed in \cref{section:connection as 1-form}. To be precise, we compare \cref{defn:quasi connection} with \cref{defn:connection as 1-form}, or equivalently by \cref{thm:chartacterization of connection form in terms of subbundle}, with a sub-bundle $\mathcal H\subset TP$ satisfying \eqref{eqn:subbundle splitting ker dpi} and \eqref{eqn:stability of connection under quasi action}.

A connection pair $(\omega, \mathbb{H})$ on a $\mbbX$-bundle $\pi\colon P\to M$ as in \cref{defn:connection as 1-form} induces a quasi connection $F=(F_0, F_1):\AtP \rightarrow \AdP$(\cref{defn:quasi connection}), $F_0\Bigl([p, u, v]\Bigr)=[p, \omega_p(u), v]$ and $F_1\Bigl([p, u, v], g\Bigr)=\Bigl([p, \omega_p(u), v], g\Bigr).$
It is not difficult to verify the well-definedness, functoriality, and smoothness of the maps defined above.  Induced map on the connected components of a fibre $m \in M$, 
\begin{equation}\label{eqn:quasi connection for connection on connected components}
    \begin{split}
      \Bar{F}:{\mathfrak{Conn}}\bigl(\Bar{\pi}_{\AtP}^{-1}(m)\bigr) & \rightarrow {\mathfrak{Conn}}\bigl(\Bar{\pi}_{\AdP}^{-1}(m)\bigr) \\
      \langle[p, u, v]\rangle & \rightarrow \langle[p, \omega_p(u), v]\rangle 
    \end{split} 
\end{equation}
conforms to  the linearity $\Bar{F}\bigl(\langle[p, u, v]\rangle+\langle[pg, u', v']\rangle\bigr)={}\Bar{F}\bigl(\langle[p,u, v]\rangle\bigr)+\Bar{F}\bigl(\langle[pg, u', v']\rangle\bigr).$
The converse generally does not hold due to the basic curvature $K_{\sigma}^{bas}$ of the Ehresmann connection. Since  $K_{\sigma}^{bas}$ vanishes for a Cartan connection, we have a one-to-one correspondence between quasi-connections and connections if the Lie groupoid connection is Cartan.  

\begin{prop}\label{prop:Connection and quasi connection in case of Cartan connection}
    Let $\pi:P \rightarrow M$ be a principal $\mbbX$-bundle with  a Cartan connection  $\sigma$ on $\mathbb{X}$.  Then, there is a one-to-one correspondence between quasi-connections and connection forms on $P.$ 
\end{prop}
\begin{proof}
    One direction has just been shown.  For the other, suppose $F \colon \AtP \rightarrow \AdP$ is a quasi connection.  Recall that for a Cartan connection, ${\rm{Obj}}\Bigl(\AtP\Bigr)=TP \oplus a^*TX_0$ and ${\rm{Obj}}\Bigl(\AdP\Bigr)=P \times_{X_0} A \oplus a^*TX_0.$  The object-level map, $(p, u, v) \rightarrow (p, (F^0)_p(u),v)$, defines a vector bundle map $TP \rightarrow P \times_{X_0} A$ given by $(p,u) \mapsto \bigl(p,(F^0)_p(u)\bigr).$  Hence, it is a $A$-valued $1$-form on $P.$ Moreover, as   $F$ is a left splitting of \cref{sequence diagram:ATIYAH sequence} for every $p \in P$ it satisfies  $(F^0)_p \circ d\delta_p(\alpha)=\alpha,$ for every $\alpha \in A_{a(p)}$.  The functoriality of $F$ implies $(F^0)_{pg}\bigl(\tau_g^0(u)\bigr)=\lambda_{g^{-1}}^0(F^0)_p(u).$  
 \end{proof}
\begin{funding}
 Naga Arjun S J acknowledges the research grant(09/0997(12538)/2021-EMR-I) from CSIR-UGC, Government of India.
\end{funding}
%\begin{appendices}*

%The contents...
%\end{appendices}

\begin{appendices}
\section{Calculations for \cref{prop: TP plus a^*TX_0 is a equivariant bundle upto homotopy} }\label{subsection: Calculations for atiyah bundle proposition}
Suppose $(p,g) \in P \times_{X_0} X_1$, then 
    \begin{equation} \label{eqn: computation of da comp Upsilon}
        \begin{split}
            (da)_{pg}\bigl(\tau_g^0(u))={} & \underbrace{(da)_{pg}\bigl(u.(di)_{g^{-1}}\sigma_{g^{-1}}(da)_p(u))}_{(ds)_g(di)_{g^{-1}}\sigma_{g^{-1}}(da)_p(u) \, {\rm by} \, \cref{eqn:reference example for relation of action with anchor}} \\
            ={} &d(\underbrace{s \circ i}_{t})_{g^{-1}}\sigma_{g^{-1}}(da)_p(u) \\
            ={} & \underbrace{(dt)_{g^{-1}}\sigma_{g^{-1}}(da)_p(u)}_{\lambda_{g^{-1}}^1(da)_p(u) \, {\rm  by \,} \cref{quasi action of X on TX_0}} \\
            ={} & \underbrace{\lambda_{g^{-1}}(da)_p(u)} \\
            & [{\rm The \, underbraced \, term }={} \tau_{g}^1(da)_p(u) {\rm by \, \cref{eqn: quasi action of X on a^*TX_0}}] \\
            ={} & \tau_{g}^1(da)_p(u).
        \end{split}
    \end{equation}
    Now suppose $(p,g,h)$ is a triple with $a(p)=t(g)$ and $s(g)=t(h),$ then from \cref{eqn:simplified version of error of quasi action on TP} we have that
    \begin{equation}\label{eqn: computation of deviation of tau^1 for atiyah bundle}
        \begin{split}
            \tau_{gh}^0(u)-\tau_h^0\tau_g^0(u)={} & \underbrace{0.l_{gh}K_{\sigma}^{bas}(h^{-1},g^{-1})(da)_p(u)}_{=\Upsilon(g,h)(da)_p(u) \, {by \,} \cref{eqn: Upsilon map for atiyah bundle}}\\
            ={} & \Upsilon(g,h)(da)_p(u).
        \end{split}
    \end{equation}
    For other quasi action, we notice that
    \begin{equation}\label{eqn: computation of da comp Upsilon for atiyah bundle}
        \begin{split}
            (da)_{pg}\Upsilon(g,h)(v)={} & \underbrace{(da)_{pg}\bigl(0.l_{gh}K_{\sigma}^{bas}(h^{-1},g^{-1})(v)\bigr)} \\
            &[{\rm The \, underbraced \, term \,}=(ds)_{gh}l_{gh}K_{\sigma}^{bas}(h^{-1},g^{-1})(v)\\
            & {\rm by \, applying \, \cref{eqn:reference example for relation of action with anchor} \, for \, the \, tuple\,} (p,gh)] \\
            ={} & (ds)_{gh}\underbrace{l_{gh}}K_{\sigma}^{bas}(h^{-1},g^{-1})(v) \\
            & [{\rm The \, underbraced \, term \,}=(dL_{gh})_{1_{s(gh)}}(di)_{1_{s(gh)}} \, {\rm by \, \cref{eqn:description of left translations at tangent level}}] \\
            ={} & \underbrace{(ds)_{gh} \circ (dL_{gh})_{1_{s(gh)}} \circ (di)_{1_{s(gh)}}}_{=dt \, {\rm by 
 \, taking \, differentials \, of \, s \circ L_g \circ i=t}}K_{\sigma}^{bas}(h^{-1},g^{-1})(v) \\
 ={} & \underbrace{(dt)_{1_{t(h^{-1})}}K_{\sigma}^{bas}(h^{-1},g^{-1})(v)} \\
 & [{\rm The \, underbraced \, term \,}=\tau_{gh}^1(v)-\tau_h^1\tau_g^1(v) \, {\rm by \, \cref{eqn: error of quasi action on TX_0=dt comp basic curvature}}] \\
 ={} & \tau_{gh}^1(v)-\tau_h^1\tau_g^1(v).
        \end{split}
    \end{equation}
In order to verify $        \tau_k^0\Upsilon(g,h)-\Upsilon(gh,k)+\Upsilon(g,hk)-\Upsilon(h,k)\tau_g^1={} 0
$, we compute each of the terms separately.
    \begin{equation}\nonumber
        \begin{split}
            \tau_k^1\Upsilon(g,h)(v)={} & \tau_k^1\bigl(0.l_{gh}K_{\sigma}^{bas}(h^{-1},g^{-1})(v)\bigr)\\
            ={} & \underbrace{\Bigl(0.l_{gh}K_{\sigma}^{bas}(h^{-1},g^{-1})(v)\Bigr).\Bigl((di)_{k^{-1}}\sigma_{k^{-1}}(da)_{pg}\bigl(0.l_{gh}K_{\sigma}^{bas}(h^{-1},g^{-1})(v)\bigr)\Bigr)}\\
             & [{\rm The \, underbraced \, term \,}={0.\Bigl(l_{gh}}K_{\sigma}^{bas}(h^{-1},g^{-1})(v).(di)_{k^{-1}}\sigma_{k^{-1}}(da)_{pg}(0.l_{gh}K_{\sigma}^{bas}(h^{-1},g^{-1})(v)\Bigr) \\
            &  {\rm by \, \cref{diagram:multiplicative action of torsor at tangent space level}\, 
             for \, the \, triple \,} (p,gh,k)] \\
            ={} & 0.\Bigl(l_{gh}K_{\sigma}^{bas}(h^{-1},g^{-1})(v).(di)_{k^{-1}}\sigma_{k^{-1}}\underbrace{(da)_{pg}\bigl(0.l_{gh}K_{\sigma}^{bas}(h^{-1},g^{-1})(v)}\bigr)\Bigr) \\
            & [{\rm The \, underbraced \, term \,}=(dt)_{1_{s(h)}}K_{\sigma}^{bas}(h^{-1},g^{-1})(v) \, {\rm by \, \cref{eqn: computation of da comp Upsilon for atiyah bundle}}] \\
            ={} & 0.\Bigl(\underbrace{l_{gh}}K_{\sigma}^{bas}(h^{-1},g^{-1})(v).(di)_{k^{-1}}\sigma_{k^{-1}}(dt)_{1_{s(h)}}K_{\sigma}^{bas}(h^{-1},g^{-1})(v)\Bigr) \\
            &[{\rm The \, underbraced \, term \,}=(dL_{gh})_{1_{s(gh)}}(di)_{1_{s(gh)}} \, {\rm by \, \cref{eqn:description of left translations at tangent level} \, for \, }gh]\\ 
            ={} &0.\Bigl(\underbrace{(dL_{gh})_{1_{s(gh)}}(di)_{1_{s(gh)}}}K_{\sigma}^{bas}(h^{-1},g^{-1})(v).(di)_{k^{-1}}\sigma_{k^{-1}}(dt)_{1_{s(h)}}K_{\sigma}^{bas}(h^{-1},g^{-1})(v) \Bigr) \\
            & [{\rm The \, underbraced \, term \,}=(di)_{h^{-1}g^{-1}}(dR_{h^{-1}g^{-1}})_{1_{t(h^{-1}})} \, {\rm by \, differentiating \, \cref{diagram:i comp R=L comp i} \, at \,} 1_{s(h)}] \\
            ={} & 0.\underbrace{\Bigl((di)_{h^{-1}g^{-1}}(dR_{h^{-1}g^{-1}})_{1_{s(h)}}K_{\sigma}^{bas}(h^{-1},g^{-1})(v).(di)_{k^{-1}}\sigma_{k^{-1}}(dt)_{1_{s(h)}}K_{\sigma}^{bas}(h^{-1},g^{-1})(v) \Bigr)} \\
            &[{\rm The \, underbraced \, term \,}=0.\Bigl((di)_{k^{-1}h^{-1}g^{-1}}\bigl(\sigma_{k^{-1}}(dt)_{1_{s(h)}}K_{\sigma}^{bas}(h^{-1},g^{-1})(v).(dR_{h^{-1}g^{-1}})_{1_{s(h)}} \\
             & K_{\sigma}^{bas}(h^{-1},g^{-1})(v)\bigr)\Bigr){\rm by \, \cref{diagram:relation of dm and di at the tangent level} \, for \, the \, tuple \, } (gh,k) ] \\
             ={} & 0.\Bigl((di)_{k^{-1}h^{-1}g^{-1}}\bigl(\underbrace{\sigma_{k^{-1}}(dt)_{1_{s(h)}}K_{\sigma}^{bas}(h^{-1},g^{-1})(v)}.(dR_{h^{-1}g^{-1}})_{1_{s(h)}}K_{\sigma}^{bas}(h^{-1},g^{-1})(v)\bigr) \Bigr) \\
             & [{\rm The \, underbraced \, term \,}=l_{k^{-1}}K_{\sigma}^{bas}(h^{-1},g^{-1})(v)-r_{k^{-1}}\psi_{k^{-1}}l_{k^{-1}}K_{\sigma}^{bas}(h^{-1},g^{-1})(v) \\
            & {\rm by \, applying \,} l_{k^{-1}}K_{\sigma}^{bas}(h^{-1},g^{-1})(v) \,  {\rm in \, third \, equation \, of \, \cref{eqn:relation between three definition of connections on lie groupoid} 
   }]
            \end{split}
            \end{equation}
            \begin{equation}\label{eqn: first term for homotopy map for atiyah case}
            \begin{split}
             & \hspace{1.6cm}={}0.\Bigl((di)_{k^{-1}h^{-1}g^{-1}}\bigl(l_{k^{-1}}K_{\sigma}^{bas}(h^{-1},g^{-1})(v)-r_{k^{-1}}\psi_{k^{-1}}l_{k^{-1}}K_{\sigma}^{bas}(h^{-1},g^{-1})(v).\underbrace{(dR_{h^{-1}g^{-1}})_{1_{s(h)}}}K_{\sigma}^{bas}(h^{-1},g^{-1})(v) \bigr) \Bigr) \\
            &\hspace{2cm} [{\rm The \, underbraced \, term \,}=r_{h^{-1}g^{-1}} \, {\rm by \, \cref{equation:small r map}}] \\
            & \hspace{1.6cm}={}  0.\Bigl((di)_{k^{-1}h^{-1}g^{-1}}\bigl(l_{k^{-1}}K_{\sigma}^{bas}(h^{-1},g^{-1})(v)-r_{k^{-1}}\underbrace{\psi_{k^{-1}}l_{k^{-1}}}K_{\sigma}^{bas}(h^{-1},g^{-1})(v).r_{h^{-1}g^{-1}}K_{\sigma}^{bas}(h^{-1},g^{-1})(v) \bigr) \Bigr) \\
            &\hspace{2cm}[{\rm The \, underbraced \, term \,}=-\lambda_{k^{-1}}^0 \, {\rm by \, \cref{quasi action of X on A}}] \\
            & \hspace{1.6cm}={}  0.\Bigl((di)_{k^{-1}h^{-1}g^{-1}}\bigl(l_{k^{-1}}K_{\sigma}^{bas}(h^{-1},g^{-1})(v)+r_{k^{-1}}\lambda_{g^{-1}}^0(K_{\sigma}^{bas}(h^{-1},g^{-1})(v)).r_{h^{-1}g^{-1}}K_{\sigma}^{bas}(h^{-1},g^{-1})(v)\bigr) \Bigr) \\
            & \hspace{1.6cm}={}  0. \biggl( (di)_{k^{-1}h^{-1}g^{-1}} \Bigl(\underbrace{l_{k^{-1}}K_{\sigma}^{bas}(h^{-1},g^{-1})(v).r_{h^{-1}g^{-1}}K_{\sigma}^{bas}(h^{-1},g^{-1})(v)}\Bigr)+\Bigl(r_{k^{-1}}\lambda_{k^{-1}}^0(K_{\sigma}^{bas}(h^{-1},g^{-1})(v)).0 \Bigr) \biggr) \\
            &\hspace{2cm}[{\rm The \, underbraced \, term \,}=0 \, {\rm by \, eq. 9} \, in \, \cite{cran} ] \\
            & \hspace{1.6cm}={}  0. \Bigl((di)_{k^{-1}h^{-1}g^{-1}}\bigl(\underbrace{r_{k^{-1}}\lambda_{k^{-1}}^0K_{\sigma}^{bas}(h^{-1},g^{-1})(v).0} \bigr) \Bigr) \\
            &\hspace{2cm}[{\rm The \, underbraced \, term}=(dR_{h^{-1}g^{-1}})_{k^{-1}} \, {\rm by \, differential \, of \, } R_{k^{-1}} \, { \rm map } ] \\
            & \hspace{1.6cm}={}  0. \Bigl((di)_{k^{-1}h^{-1}g^{-1}}\bigl((dR_{h^{-1}g^{-1}})_{k^{-1}}\bigl(\underbrace{r_{k^{-1}}}\lambda_{k^{-1}}^0K_{\sigma}^{bas}(h^{-1},g^{-1})(v) \bigr) \bigr) \Bigr) \\
            & \hspace{2cm}[{\rm The \, underbraced \, term \,}=(dR_{k^{-1}})_{1_{t(k^{-1})}} \, {\rm by \, \cref{equation:small r map}}] \\
            & \hspace{1.6cm}={}  0. \Bigl((di)_{k^{-1}h^{-1}g^{-1}}\bigl(\underbrace{(dR_{h^{-1}g^{-1}})_{k^{-1}}(dR_{k^{-1}})_{1_{t(k^{-1})}}}\lambda_{k^{-1}}^0K_{\sigma}^{bas}(h^{-1},g^{-1})(v) \bigr) \Bigr) \\
            &\hspace{2cm} [{\rm The \, underbraced \, term \,}=(dR_{k^{-1}h^{-1}g^{-1}})_{1_{t(k^{-1})}}] \\
            & \hspace{1.6cm}={}  0. \Bigl(\underbrace{(di)_{k^{-1}h^{-1}g^{-1}}(dR_{k^{-1}h^{-1}g^{-1}})_{1_{t(k^{-1})}}}\lambda_{k^{-1}}^0K_{\sigma}^{bas}(h^{-1},g^{-1})(v) \Bigr) \\
            &\hspace{2cm} [{\rm The \, underbraced \, term \,}=(dL_{ghk})_{1_{s(ghk)}}(di)_{1_{s(ghk)}} \, {\rm by \, differentiating \, \cref{diagram:i comp R=L comp i} \, at \,} 1_{s(ghk)}] \\
            & \hspace{1.6cm}={}  0. \Bigl(\underbrace{(dL_{ghk})(di)_{1_{s(ghk)}} }_{=l_{ghk} {\rm\, by \, \cref{eqn:description of left translations at tangent level}}} \lambda_{k^{-1}}^0K_{\sigma}^{bas}(h^{-1},g^{-1})(v) \Bigr) \\
            & \hspace{1.6cm}={}  0. \Bigl(l_{ghk}\lambda_{k^{-1}}^0K_{\sigma}^{bas}(h^{-1},g^{-1})(v) \Bigr).
        \end{split}
    \end{equation} 
    Now, by definition of $\Upsilon$ from \cref{eqn: Upsilon map for atiyah bundle} we get that,
    \begin{equation}\label{eqn: third term of homotopy map relation for atiyah case}
        \begin{split}
            \Upsilon(gh,k)(v)={} & 0. \Bigl(l_{ghk}K_{\sigma}^{bas}(k^{-1},h^{-1}g^{-1}(v) \Bigr).
        \end{split}
    \end{equation}
    and 
    \begin{equation}\label{eqn: second term of homotopy map relation for atiyah case}
        \begin{split}
            \Upsilon(g,hk)={} & 0. \Bigl(l_{ghk}K_{\sigma}^{bas}(k^{-1}h^{-1},g^{-1}) \Bigr).
        \end{split}
    \end{equation}
    Also, we have that, 
    \begin{equation}\label{eqn: fourth term of homotopy map relation for atiyah case}
        \begin{split}
            \Upsilon(h,k)\tau_g^1(v)={} & 0. \Bigl(\underbrace{l_{hk}}K_{\sigma}^{bas}(k^{-1},h^{-1})\lambda_{g^{-1}}^1(v) \Bigr) \\
            &[{\rm The \, underbraced \, term \,}=(dL_{hk})_{1_{s(hk)}}(di)_{1_{s(hk)}} \, {\rm by \, \cref{eqn:description of left translations at tangent level}}] \\
            ={} & 0. \Bigl(\underbrace{(dL_{hk})_{1_{s(k)}}(di)_{1_{s(k)}}}K_{\sigma}^{bas}(k^{-1},h^{-1})\lambda_{g^{-1}}^1(v) \Bigr) \\
            &[{\rm The \, underbraced \, term \,}=(di)_{k^{-1}}(dR_{k^{-1}h^{-1}})_{1_{t(k^{-1})}} \, {\rm by \, taking \, differentials \, of \, \cref{diagram:i comp R=L comp i} \, at \,} 1_{t(k^{-1})}] \\
            ={} & \underbrace{0}_{=0.0}. \Bigl((di)_{k^{-1}h^{-1}}(dR_{k^{-1}h^{-1}})_{1_{t(k^{-1})}}K_{\sigma}^{bas} (k^{-1},h^{-1})\lambda_{g^{-1}}^1(v)\Bigr)\\
            ={} & \underbrace{\Bigl(0.0\Bigr).\Bigl(di)_{k^{-1}h^{-1}}(dR_{k^{-1}h^{-1}})_{1_{t(k^{-1})}}K_{\sigma}^{bas}(k^{-1},h^{-1})\lambda_{g^{-1}}^1(v) \Bigr)} \\
            &[{\rm The \, underbraced \, term \,}=0.\Bigl(0.(di)_{k^{-1}}(dR_{k^{-1}h^{-1}})_{1_{t(k^{-1})}}K_{\sigma}^{bas}(k^{-1},h^{-1})\lambda_{g^{-1}}^1(v) \Bigr) \\
            &{\rm by \, \cref{diagram:multiplicative action of torsor at tangent space level} \, for \, the \, triple \, } (p,g,hk)] \\
            ={} & 0.\Bigl(\underbrace{0}.(di)_{k^{-1}}(dR_{k^{-1}h^{-1}})_{1_{t(k^{-1})}}K_{\sigma}^{bas}(k^{-1},h^{-1})\lambda_{g^{-1}}^1(v) \Bigr) \\
            & [{\rm The \, underbraced \, term \,}=(di)_{g^{-1}}(0)] \\
            ={} & 0.\Bigl(\underbrace{(di)_{g^{-1}}(0).(di)_{k^{-1}h^{-1}})(dR_{k^{-1}h^{-1}})_{1_{t(k^{-1})}}K_{\sigma}^{bas}(k^{-1},h^{-1})\lambda_{g^{-1}}^1(v)} \Bigr) \\
            & [{\rm The \, underbraced \, term \,}=(di)_{k^{-1}h^{-1}g^{-1}}\Bigl((dR_{k^{-1}h^{-1}})_{1_{t(k^{-1})}}K_{\sigma}^{bas}(k^{-1},h^{-1})\lambda_{g^{-1}}^1(v).0\Bigr) \\
            &{\rm by \, \cref{diagram:relation of dm and di at the tangent level} for \, the \, tuple \,} (hk,g)] \\
            ={} & 0. \Bigl((di)_{k^{-1}h^{-1}g^{-1}}\bigl(\underbrace{(dR_{k^{-1}h^{-1}})_{1_{t(k^{-1})}}K_{\sigma}^{bas}(k^{-1},h^{-1})\lambda_{g^{-1}}^1(v).0}\bigr) \Bigr)\\
            &[{\rm The \, underbraced \, term \,}=(dR_{g^{-1}})_{k^{-1}h^{-1}} \, {\rm by \, differential \, of \,} R_{g^{-1}} \, {\rm map}] \\
            ={} & 0. \Bigl( (di)_{k^{-1}h^{-1}g^{-1}} \bigl( \underbrace{(dR_{g^{-1}})_{k^{-1}h^{-1}} (dR_{k^{-1}h^{-1}})_{1_{t(k^{-1})}}}_{=(dR_{k^{-1}h^{-1}g^{-1}})_{1_{t(k^{-1})}}}K_{\sigma}^{bas}(k^{-1},h^{-1})\lambda_{g^{-1}}^1(v) \bigr)\Bigr) \\
            ={} & 0. \Bigl(\underbrace{ (di)_{k^{-1}h^{-1}g^{-1}})(dR_{k^{-1}h^{-1}g^{-1}})_{1_{t(k^{-1}h^{-1}g^{-1})}}}K_{\sigma}^{bas}(k^{-1},h^{-1})\lambda_{g^{-1}}^1(v) \Bigr) \\
            & [{\rm The \, underbraced \, term \,}=(dL_{ghk})_{1_{s(ghk)}}(di)_{1_{s(ghk)}} \, {\rm by \, taking \, differentials \, of \, \cref{diagram:i comp R=L comp i} \, at \,} 1_{s(ghk)}] \\
            ={} & 0. \Bigl( \underbrace{(dL_{ghk})_{1_{s(ghk)}}(di)_{1_{s(ghk)}}K_{\sigma}^{bas}(k^{-1},h^{-1})}_{=l_{ghk} \, {\rm by \, \cref{eqn:description of left translations at tangent level}}}\lambda_{g^{-1}}^1(v)\Bigr) \\
            ={} & 0. \Bigl(l_{ghk}K_{\sigma}^{bas}(k^{-1},h^{-1})\lambda_{g^{-1}}^1(v) \Bigr).
        \end{split}
    \end{equation}
\end{appendices}

\bibliography{main}

\begin{thebibliography}{42}
\providecommand{\natexlab}[1]{#1}
\providecommand{\url}[1]{\texttt{#1}}
\expandafter\ifx\csname urlstyle\endcsname\relax
  \providecommand{\doi}[1]{doi: #1}\else
  \providecommand{\doi}{doi: \begingroup \urlstyle{rm}\Url}\fi

\bibitem[Arias~Abad and Crainic(2013)]{cran}
Camilo Arias~Abad and Marius Crainic.
\newblock Representations up to homotopy and {B}ott's spectral sequence for
  {L}ie groupoids.
\newblock \emph{Adv. Math.}, 248:\penalty0 416--452, 2013.
\newblock \doi{https://doi.org/10.1016/j.aim.2012.12.022}.

\bibitem[Aschieri et~al.(2005)Aschieri, Cantini, and Jur\v~co]{MR2117631}
Paolo Aschieri, Luigi Cantini, and Branislav Jur\v~co.
\newblock Nonabelian bundle gerbes, their differential geometry and gauge
  theory.
\newblock \emph{Comm. Math. Phys.}, 254\penalty0 (2):\penalty0 367--400, 2005.
\newblock \doi{https://doi.org/10.1007/s00220-004-1220-6}.

\bibitem[Atiyah(1957)]{MR86359}
M.~F. Atiyah.
\newblock Complex analytic connections in fibre bundles.
\newblock \emph{Trans. Amer. Math. Soc.}, 85:\penalty0 181--207, 1957.
\newblock \doi{https://doi.org/10.2307/1992969}.

\bibitem[Barbosa-Torres and Neumann(2021)]{MR4170276}
Luis~Alejandro Barbosa-Torres and Frank Neumann.
\newblock Equivariant cohomology for differentiable stacks.
\newblock \emph{J. Geom. Phys.}, 160, 2021.
\newblock \doi{https://doi.org/10.1016/j.geomphys.2020.103974}.

\bibitem[Behrend and Xu(2011)]{MR2817778}
Kai Behrend and Ping Xu.
\newblock Differentiable stacks and gerbes.
\newblock \emph{J. Symplectic Geom.}, 9\penalty0 (3):\penalty0 285--341, 2011.
\newblock URL \url{http://projecteuclid.org/euclid.jsg/1310388899}.

\bibitem[Behrend(2005)]{BEHREND2005583}
Kai~A. Behrend.
\newblock On the de rham cohomology of differential and algebraic stacks.
\newblock \emph{Advances in Mathematics}, 198\penalty0 (2):\penalty0 583--622,
  2005.
\newblock URL
  \url{https://www.sciencedirect.com/science/article/pii/S0001870805002471}.

\bibitem[Biswas and Neumann(2014)]{MR3150770}
Indranil Biswas and Frank Neumann.
\newblock Atiyah sequences, connections and characteristic forms for principal
  bundles over groupoids and stacks.
\newblock \emph{C. R. Math. Acad. Sci. Paris}, 352\penalty0 (1):\penalty0
  59--64, 2014.
\newblock \doi{https://doi.org/10.1016/j.crma.2013.10.038}.

\bibitem[Biswas et~al.(2023)Biswas, Chatterjee, Koushik, and
  Neumann]{MR4592876}
Indranil Biswas, Saikat Chatterjee, Praphulla Koushik, and Frank Neumann.
\newblock Atiyah sequences and connections on principal bundles over {L}ie
  groupoids and differentiable stacks.
\newblock \emph{J. Noncommut. Geom.}, 17\penalty0 (2):\penalty0 407--437, 2023.
\newblock \doi{https://doi.org/10.4171/jncg/486}.

\bibitem[Biswas et~al.(2024)Biswas, Chatterjee, Koushik, and
  Neumann]{MR4721218}
Indranil Biswas, Saikat Chatterjee, Praphulla Koushik, and Frank Neumann.
\newblock Connections on {L}ie groupoids and {C}hern-{W}eil theory.
\newblock \emph{Rev. Math. Phys.}, 36\penalty0 (3):\penalty0 Paper No. 2450002,
  42, 2024.
\newblock \doi{https://doi.org/10.1142/S0129055X24500028}.

\bibitem[Blohmann(2008)]{MR2439561}
Christian Blohmann.
\newblock Stacky {L}ie groups.
\newblock \emph{Int. Math. Res. Not. IMRN}, 2008.
\newblock \doi{https://doi.org/10.1093/imrn/rnn082}.

\bibitem[Bursztyn et~al.(2020)Bursztyn, Noseda, and Zhu]{MR4139032}
Henrique Bursztyn, Francesco Noseda, and Chenchang Zhu.
\newblock Principal actions of stacky {L}ie groupoids.
\newblock \emph{Int. Math. Res. Not. IMRN}, \penalty0 (16):\penalty0
  5055--5125, 2020.
\newblock \doi{https://doi.org/10.1093/imrn/rny142}.

\bibitem[Chatterjee and Chaudhuri(2023)]{chatterjee2023parallel}
Saikat Chatterjee and Adittya Chaudhuri.
\newblock Parallel transport on a lie 2-group bundle over a lie groupoid along
  haefliger paths.
\newblock \emph{arXiv preprint arXiv:2309.05355}, 2023.
\newblock URL \url{https://arxiv.org/abs/2309.05355}.

\bibitem[Chatterjee and Koushik(2020)]{MR4124773}
Saikat Chatterjee and Praphulla Koushik.
\newblock On two notions of a gerbe over a stack.
\newblock \emph{Bull. Sci. Math.}, 163:\penalty0 102886, 31, 2020.
\newblock \doi{https://doi.org/10.1016/j.bulsci.2020.102886}.

\bibitem[Chatterjee et~al.(2022)Chatterjee, Chaudhuri, and Koushik]{MR4403617}
Saikat Chatterjee, Adittya Chaudhuri, and Praphulla Koushik.
\newblock Atiyah sequence and gauge transformations of a principal 2-bundle
  over a {L}ie groupoid.
\newblock \emph{J. Geom. Phys.}, 176, 2022.
\newblock \doi{https://doi.org/10.1016/j.geomphys.2022.104509}.

\bibitem[Christensen and Wu(2016)]{MR3467758}
J.~Daniel Christensen and Enxin Wu.
\newblock Tangent spaces and tangent bundles for diffeological spaces.
\newblock \emph{Cah. Topol. G\'eom. Diff\'er. Cat\'eg.}, 57\penalty0
  (1):\penalty0 3--50, 2016.
\newblock
  \doi{https://cahierstgdc.com/wp-content/uploads/2017/11/ChristensenWu.pdf}.

\bibitem[Crainic and Fernandes(2003)]{MR1973056}
Marius Crainic and Rui~Loja Fernandes.
\newblock Integrability of {L}ie brackets.
\newblock \emph{Ann. of Math. (2)}, 157\penalty0 (2):\penalty0 575--620, 2003.
\newblock \doi{https://doi.org/10.4007/annals.2003.157.575}.

\bibitem[Fernandes and
  Crainic(2006)]{fernandes2006lecturesintegrabilityliebrackets}
Rui~Loja Fernandes and Marius Crainic.
\newblock Lectures on integrability of lie brackets, 2006.
\newblock URL \url{https://arxiv.org/abs/math/0611259}.

\bibitem[Fischer(2021)]{MR4202191}
Simon-Raphael Fischer.
\newblock Curved {Y}ang-{M}ills-{H}iggs gauge theories in the case of massless
  gauge bosons.
\newblock \emph{J. Geom. Phys.}, 162:\penalty0 Paper No. 104104, 23, 2021.
\newblock ISSN 0393-0440,1879-1662.
\newblock \doi{10.1016/j.geomphys.2021.104104}.

\bibitem[Fischer et~al.(2024)Fischer, Farahani, Kim, and
  Saemann]{fischer2024adjustedconnectionsidifferential}
Simon-Raphael Fischer, Mehran~Jalali Farahani, Hyungrok Kim, and Christian
  Saemann.
\newblock Adjusted connections i: Differential cocycles for principal groupoid
  bundles with connection, 2024.
\newblock URL \url{https://arxiv.org/abs/2406.16755}.

\bibitem[Ginot(2013)]{ginot2013introduction}
Gr{\'e}gory Ginot.
\newblock Introduction to differentiable stacks (and gerbes, moduli spaces...).
\newblock \emph{Lecture Notes}, 2013.
\newblock URL
  \url{https://www.math.univ-paris13.fr/~ginot/papers/DiffStacksIGG2013.pdf}.

\bibitem[Ginot and Sti\'enon(2015)]{MR3480061}
Gr\'egory Ginot and Mathieu Sti\'enon.
\newblock {$G$}-gerbes, principal 2-group bundles and characteristic classes.
\newblock \emph{J. Symplectic Geom.}, 13\penalty0 (4):\penalty0 1001--1047,
  2015.
\newblock \doi{https://doi.org/10.4310/JSG.2015.v13.n4.a6}.

\bibitem[Haefliger(1984)]{MR755163}
Andr\'e Haefliger.
\newblock Groupo\"ides d'holonomie et classifiants.
\newblock Number 116, pages 70--97. 1984.
\newblock URL \url{http://www.numdam.org/item/AST_1984__116__70_0/}.
\newblock Transversal structure of foliations (Toulouse, 1982).

\bibitem[Iglesias-Zemmour(2013)]{iglesias2013diffeology}
P.~Iglesias-Zemmour.
\newblock \emph{Diffeology}.
\newblock Mathematical Surveys and Monographs. American Mathematical Society,
  2013.

\bibitem[Ivan(2001)]{MR1871544}
Gheorghe Ivan.
\newblock Principal fibre bundles with structural {L}ie groupoid.
\newblock \emph{Balkan J. Geom. Appl.}, 6\penalty0 (2):\penalty0 39--48, 2001.
\newblock URL \url{https://www.emis.de/journals/BJGA/v06n2/B06-2-IVAN.pdf}.

\bibitem[Laurent-Gengoux et~al.(2007)Laurent-Gengoux, Tu, and Xu]{MR2270285}
Camille Laurent-Gengoux, Jean-Louis Tu, and Ping Xu.
\newblock Chern-{W}eil map for principal bundles over groupoids.
\newblock \emph{Math. Z.}, 255\penalty0 (3):\penalty0 451--491, 2007.
\newblock \doi{https://doi.org/10.1007/s00209-006-0004-4}.

\bibitem[Lerman(2010)]{MR2778793}
Eugene Lerman.
\newblock Orbifolds as stacks?
\newblock \emph{Enseign. Math. (2)}, 56\penalty0 (3-4):\penalty0 315--363,
  2010.
\newblock \doi{https://doi.org/10.4171/LEM/56-3-4}.

\bibitem[Mackenzie(1987)]{mackenzie1987lie}
K.~Mackenzie.
\newblock \emph{Lie Groupoids and Lie Algebroids in Differential Geometry}.
\newblock Lecture note series / London mathematical society. Cambridge
  University Press, 1987.

\bibitem[Mackenzie(2005)]{MR2157566}
Kirill C.~H. Mackenzie.
\newblock \emph{General theory of {L}ie groupoids and {L}ie algebroids}, volume
  213 of \emph{London Mathematical Society Lecture Note Series}.
\newblock Cambridge University Press, Cambridge, 2005.
\newblock \doi{https://doi.org/10.1017/CBO9781107325883}.

\bibitem[Moerdijk(2003)]{MR2017529}
I.~Moerdijk.
\newblock Lie groupoids, gerbes, and non-abelian cohomology.
\newblock \emph{$K$-Theory}, 28\penalty0 (3):\penalty0 207--258, 2003.
\newblock \doi{https://doi.org/10.1023/A:1026251115381}.

\bibitem[Moerdijk and Mrcun(2003)]{moerdijk2003introduction}
I.~Moerdijk and J.~Mrcun.
\newblock \emph{Introduction to Foliations and Lie Groupoids}.
\newblock Cambridge Studies in Advanced Mathematics. Cambridge University
  Press, 2003.

\bibitem[Moerdijk(2002)]{moerdijk2002introduction}
Ieke Moerdijk.
\newblock Introduction to the language of stacks and gerbes.
\newblock \emph{arXiv preprint math/0212266}, 2002.
\newblock URL \url{https://arxiv.org/abs/math/0212266}.

\bibitem[Murray et~al.(2012)Murray, Roberts, and Stevenson]{MR3004105}
Michael Murray, David~Michael Roberts, and Danny Stevenson.
\newblock On the existence of bibundles.
\newblock \emph{Proc. Lond. Math. Soc. (3)}, 105\penalty0 (6):\penalty0
  1290--1314, 2012.
\newblock \doi{https://doi.org/10.1112/plms/pds028}.

\bibitem[Pervova(2016)]{PERVOVA2016269}
Ekaterina Pervova.
\newblock Diffeological vector pseudo-bundles.
\newblock \emph{Topology and its Applications}, 202:\penalty0 269--300, 2016.
\newblock URL
  \url{https://www.sciencedirect.com/science/article/pii/S0166864116000304}.

\bibitem[Pugliese et~al.(2023)Pugliese, Sparano, and
  Vitagliano]{doi:10.1142/S0219199721500929}
Fabrizio Pugliese, Giovanni Sparano, and Luca Vitagliano.
\newblock Multiplicative connections and their lie theory.
\newblock \emph{Communications in Contemporary Mathematics}, 25\penalty0
  (01):\penalty0 2150092, 2023.
\newblock \doi{https://doi.org/10.1142/S0219199721500929}.

\bibitem[Tang(2006)]{MR2238946}
Xiang Tang.
\newblock Deformation quantization of pseudo-symplectic ({P}oisson) groupoids.
\newblock \emph{Geom. Funct. Anal.}, 16\penalty0 (3):\penalty0 731--766, 2006.
\newblock \doi{https://doi.org/10.1007/s00039-006-0567-6}.

\bibitem[Trentinaglia(2018)]{MR3886162}
Giorgio Trentinaglia.
\newblock Regular {C}artan groupoids and longitudinal representations.
\newblock \emph{Adv. Math.}, 340:\penalty0 1--47, 2018.
\newblock \doi{https://doi.org/10.1016/j.aim.2018.10.006}.

\bibitem[Tu et~al.(2004)Tu, Xu, and Laurent-Gengoux]{MR2119241}
Jean-Louis Tu, Ping Xu, and Camille Laurent-Gengoux.
\newblock Twisted {$K$}-theory of differentiable stacks.
\newblock \emph{Ann. Sci. \'Ecole Norm. Sup. (4)}, 37\penalty0 (6):\penalty0
  841--910, 2004.
\newblock \doi{https://doi.org/10.1016/j.ansens.2004.10.002}.

\bibitem[Tu(2017)]{tu2017differential}
L.W. Tu.
\newblock \emph{Differential Geometry: Connections, Curvature, and
  Characteristic Classes}.
\newblock Graduate Texts in Mathematics. Springer International Publishing,
  2017.
\newblock URL \url{https://books.google.co.in/booksid=bmsmDwAAQBAJ}.

\bibitem[Viennot(2016)]{MR3566125}
David Viennot.
\newblock Non-abelian higher gauge theory and categorical bundle.
\newblock \emph{J. Geom. Phys.}, 110:\penalty0 407--435, 2016.
\newblock \doi{https://doi.org/10.1016/j.geomphys.2016.09.011}.

\bibitem[Vistoli(2005)]{MR2223406}
Angelo Vistoli.
\newblock Grothendieck topologies, fibered categories and descent theory.
\newblock In \emph{Fundamental algebraic geometry}, volume 123 of \emph{Math.
  Surveys Monogr.}, pages 1--104. Amer. Math. Soc., Providence, RI, 2005.
\newblock \doi{https://doi.org/10.1090/surv/123}.

\bibitem[Waldorf(2018)]{MR3894086}
Konrad Waldorf.
\newblock A global perspective to connections on principal 2-bundles.
\newblock \emph{Forum Math.}, 30\penalty0 (4):\penalty0 809--843, 2018.
\newblock \doi{https://doi.org/10.1515/forum-2017-0097}.

\bibitem[Xiao(2025)]{MR4828414}
Shuyu Xiao.
\newblock Abelianization of {L}ie algebroids and {L}ie groupoids.
\newblock \emph{J. Geom. Phys.}, 207:\penalty0 Paper No. 105372, 16, 2025.
\newblock \doi{https://doi.org/10.1016/j.geomphys.2024.105372}.

\end{thebibliography}


\end{document}